\documentclass[reqno,11pt,a4paper,final]{amsart}
\usepackage[a4paper,left=40mm,right=40mm,top=35mm,bottom=35mm,marginpar=25mm]{geometry} 

\usepackage{amsmath}
\usepackage{amssymb}
\usepackage{amsthm}
\usepackage{amscd}
\usepackage[ansinew]{inputenc}
\usepackage{bbm}
\usepackage{color}
\usepackage[english=american]{csquotes}
\usepackage{hyperref}
\usepackage{calc}
\usepackage{mathptmx}
\usepackage{amsmath}
\usepackage{amssymb}
\usepackage{amsthm}
\usepackage{amscd}
\usepackage{color}
\usepackage[english=american]{csquotes}
\usepackage{calc}
\usepackage{bm}
\usepackage{dsfont}
\usepackage{enumerate}
\usepackage{mathtools}
\usepackage[dvipsnames]{xcolor}

\usepackage{enumerate}

\numberwithin{equation}{section}

\newtheoremstyle{thmlemcorr}{10pt}{10pt}{\itshape}{}{\bfseries}{.}{10pt}{{\thmname{#1}\thmnumber{ #2}\thmnote{ (#3)}}}
\newtheoremstyle{thmlemcorr*}{10pt}{10pt}{\itshape}{}{\bfseries}{.}\newline{{\thmname{#1}\thmnumber{ #2}\thmnote{ (#3)}}}
\newtheoremstyle{defi}{10pt}{10pt}{\itshape}{}{\bfseries}{.}{10pt}{{\thmname{#1}\thmnumber{ #2}\thmnote{ (#3)}}}
\newtheoremstyle{remexample}{10pt}{10pt}{}{}{\bfseries}{.}{10pt}{{\thmname{#1}\thmnumber{ #2}\thmnote{ (#3)}}}
\newtheoremstyle{ass}{10pt}{10pt}{}{}{\bfseries}{.}{10pt}{{\thmname{#1}\thmnumber{ #2}\thmnote{ (#3)}}}

\theoremstyle{thmlemcorr}
\newtheorem{theorem}{Theorem}
\numberwithin{theorem}{section}
\newtheorem{lemma}[theorem]{Lemma}
\newtheorem{corollary}[theorem]{Corollary}
\newtheorem{proposition}[theorem]{Proposition}

\theoremstyle{thmlemcorr*}
\newtheorem{theorem*}{Theorem}
\newtheorem{lemma*}[theorem]{Lemma}
\newtheorem{corollary*}[theorem]{Corollary}
\newtheorem{proposition*}[theorem]{Proposition}
\newtheorem{problem*}[theorem]{Problem}
\newtheorem{conjecture*}[theorem]{Conjecture}

\theoremstyle{defi}
\newtheorem{definition}[theorem]{Definition}

\theoremstyle{remexample}
\newtheorem{remark}[theorem]{Remark}

\theoremstyle{ass}

\newcommand{\Crm}{\mathrm{C}}

\newcommand{\Lrm}{\mathrm{L}}

\newcommand{\Wrm}{\mathrm{W}}

\newcommand{\Acal}{\mathcal{A}}

\newcommand{\Dcal}{\mathcal{D}}
\newcommand{\Ecal}{\mathcal{E}}
\newcommand{\Fcal}{\mathcal{F}}
\newcommand{\Gcal}{\mathcal{G}}

\newcommand{\Lcal}{\mathcal{L}}
\newcommand{\Mcal}{\mathcal{M}}

\newcommand{\Pcal}{\mathcal{P}}

\newcommand{\Ebf}{\mathbf{E}}

\newcommand{\Ybf}{\mathbf{Y}}

\newcommand{\Abb}{\mathbb{A}}
\renewcommand{\Bbb}{\mathbb{B}}

\newcommand{\Nbb}{\mathbb{N}}

\newcommand{\Pbb}{\mathbb{P}}
\newcommand{\Qbb}{\mathbb{Q}}

\newcommand{\Sbb}{\mathbb{S}}

\newcommand{\Zbb}{\mathbb{Z}}

\DeclareMathOperator{\id}{id}

\DeclareMathOperator*{\wslim}{w*-lim}

\DeclareMathOperator{\curl}{curl}
\DeclareMathOperator{\dist}{dist}
\DeclareMathOperator{\rank}{rank}

\DeclareMathOperator{\spn}{span}

\DeclareMathOperator{\supp}{supp}

\newcommand{\ii}{\mathrm{i}}
\newcommand{\set}[2]{\left\{\, #1 \ \ \textup{\textbf{:}}\ \ #2 \,\right\}}
\newcommand{\setn}[2]{\{\, #1 \ \ \textup{\textbf{:}}\ \ #2 \,\}}
\newcommand{\setb}[2]{\bigl\{\, #1 \ \ \textup{\textbf{:}}\ \ #2 \,\bigr\}}
\newcommand{\setB}[2]{\Bigl\{\, #1 \ \ \textup{\textbf{:}}\ \ #2 \,\Bigr\}}
\newcommand{\setBB}[2]{\biggl\{\, #1 \ \ \textup{\textbf{:}}\ \ #2 \,\biggr\}}

\newcommand{\normBB}[1]{\biggl\|#1\biggr\|}

\newcommand{\abs}[1]{|#1|}

\newcommand{\absb}[1]{\bigl|#1\bigr|}

\newcommand{\dpr}[1]{\langle #1 \rangle}

\newcommand{\dprb}[1]{\bigl\langle #1 \bigr\rangle}

\newcommand{\ddpr}[1]{\langle\!\!\langle #1 \rangle\!\!\rangle}

\newcommand{\ddprb}[1]{\bigl\langle\!\!\!\bigl\langle #1 \bigr\rangle\!\!\!\bigr\rangle}

\newcommand{\cl}[1]{\overline{#1}}
\newcommand{\di}{\mathrm{d}}
\newcommand{\dd}{\;\mathrm{d}}

\newcommand{\N}{\mathbb{N}}
\newcommand{\R}{\mathbb{R}}

\newcommand{\Z}{\mathbb{Z}}
\newcommand{\loc}{\mathrm{loc}}
\newcommand{\sym}{\mathrm{sym}}

\newcommand{\per}{\mathrm{per}}

\newcommand{\toweak}{\rightharpoonup}
\newcommand{\toweakstar}{\overset{*}\rightharpoonup}
\newcommand{\toweakstarY}{\overset{*}\rightharpoonup}

\newcommand{\toup}{\uparrow}
\newcommand{\todown}{\downarrow}

\newcommand{\embed}{\hookrightarrow}
\newcommand{\cembed}{\overset{c}{\embed}}

\newcommand{\SmallO}{\mathrm{\textup{o}}}
\newcommand{\sbullet}{\begin{picture}(1,1)(-0.5,-2)\circle*{2}\end{picture}}
\newcommand{\frarg}{\,\sbullet\,}

\newcommand{\BD}{\mathrm{BD}}
\newcommand{\LD}{\mathrm{LD}}

\newcommand{\toY}{\overset{\Ybf}{\to}}

\newcommand{\area}[1]{\langle #1  \rangle}

\newcommand{\barea}{\area{\frarg}}

\DeclareMathOperator{\Tan}{Tan}

\newcommand{\proofstep}[1]{\textit{#1}}

\def\Xint#1{\mathchoice 
{\XXint\displaystyle\textstyle{#1}}% 
{\XXint\textstyle\scriptstyle{#1}}% 
{\XXint\scriptstyle\scriptscriptstyle{#1}}% 
{\XXint\scriptscriptstyle\scriptscriptstyle{#1}}% 
\!\int} 
\def\XXint#1#2#3{{\setbox0=\hbox{$#1{#2#3}{\int}$} 
\vcenter{\hbox{$#2#3$}}\kern-.5\wd0}} 
\def\dashint{\,\Xint-}

\newcommand{\restrict}{\begin{picture}(10,8)\put(2,0){\line(0,1){7}}\put(1.8,0){\line(1,0){7}}\end{picture}}

\newcommand{\eps}{\varepsilon}
\renewcommand{\phi}{\varphi}
\hypersetup{
    bookmarks=true,         % show bookmarks bar?
    unicode=false,          % non-Latin characters in Acrobat?s bookmarks
    pdftoolbar=true,        % show Acrobat?s toolbar?
    pdfmenubar=true,        % show Acrobat?s menu?
    pdffitwindow=false,     % window fit to page when opened
    pdfstartview={FitH},    % fits the width of the page to the window
    pdftitle={My title},    % title
    pdfauthor={Author},     % author
    pdfsubject={Subject},   % subject of the document
    pdfcreator={Creator},   % creator of the document
    pdfproducer={Producer}, % producer of the document
    pdfkeywords={keyword1, key2, key3}, % list of keywords
    pdfnewwindow=true,      % links in new PDF window
    colorlinks=true,       % false: boxed links; true: colored links
    linkcolor=blue,          % color of internal links (change box color with linkbordercolor)
    citecolor=blue,        % color of links to bibliography
    filecolor=magenta,      % color of file links
    urlcolor=cyan           % color of external links
}

\newcommand{\M}{\mathcal M}
\newcommand{\ysc}{\toweakstarY}

\newcommand{\CC}{\mathbb C}
\DeclareMathOperator{\Y}{\bf Y}
\DeclareMathOperator{\E}{\bf E}
\DeclareMathOperator{\e}{\bf E}
\newcommand{\A}{\Acal}

\DeclareMathOperator{\Div}{div}

\DeclareMathOperator{\Lip}{Lip}

\author[A.~Arroyo-Rabasa]{Adolfo Arroyo-Rabasa}
\address{\textit{A.A.-R.:} Mathematisches Institut, Universit\"at Leipzig, 53115 Bonn, Germany.}
\email{arroyo@math.uni-leipzig.de}
\author[G.~De Philippis]{Guido De Philippis}
\address{\textit{G.D.P:} SISSA, Via Bonomea 265, 34136 Trieste, Italy.}
\email{guido.dephilippis@sissa.it}

\author[F.~Rindler]{Filip Rindler}
\address{\textit{F.R.:} Mathematics Institute, University of Warwick, Coventry CV4 7AL, UK.}
\email{F.Rindler@warwick.ac.uk}

\hypersetup{
  pdfauthor = {Adolfo Arroyo-Rabasa (University of Bonn), Guido De Philippis (SISSA, Trieste), Filip Rindler (University of Warwick)},
  pdftitle = {Lower semicontinuity and relaxation of linear-growth integral functionals under PDE constraints},
  pdfsubject = {},
  pdfkeywords = {Lower semicontinuity, functional on measures, $\A$-quasiconvexity, generalized Young measure}
}

\title[Lower semicontinuity and relaxation of integral functionals]{Lower semicontinuity and relaxation of linear-growth integral functionals under PDE constraints} 

\begin{document}

\begin{abstract}
We show general lower semicontinuity and relaxation theorems for linear-growth integral functionals defined on vector measures that satisfy linear PDE side constraints (of arbitrary order). These results generalize several known lower semicontinuity and relaxation theorems for BV, BD, and for more general first-order linear PDE side constrains. Our proofs are based on recent progress in the understanding of singularities of measure solutions to linear PDEs and of the generalized convexity notions corresponding to these PDE constraints. %We also use the theory of generalized Young measures and the blow-up method.

\vspace{4pt}

%\noindent\textsc{MSC (2010):} xxXXxx (primary); xxXXxx, xxXXxx (secondary).
%
%\vspace{4pt}

\vspace{4pt}

\noindent\textsc{Date:} \today{}.
\end{abstract}

\maketitle

%\setcounter{tocdepth}{2} 
%\tableofcontents

\section{Introduction}

The theory of linear-growth integral functionals defined on vector-valued measures satisfying PDE constraints is central to many questions of the calculus of variations. In particular, their relaxation and lower semicontinuity properties have attracted a lot of attention, see for instance~\cite{AmbrosioDalMaso92,FonsecaMuller93,FonsecaMuller99,FonsecaLeoniMuller04,KristensenRindler10Relax,Rindler11,BaiaChermisiMatiasSantos13}. In the present work we unify and extend a large number of these results by proving general lower semicontinuity and relaxation theorems for such functionals. Our proofs are based on recent advances in the understanding of the \emph{singularities} that may occur in measures satisfying (under-determined) linear PDEs. 

Concretely, let $\Omega \subset \R^d$ be an open and bounded subset %a bounded domain (an open and connected subset) 
with $\Lcal^d(\partial \Omega)=0$ and consider %for a finite vector Radon measure $\mu \in \Mcal(\Omega;\R^N)$ on $\Omega$ with values in $\R^N$ 
the functional
\begin{equation}\label{Fintro}
 \Fcal^\#[\mu] :=  \int_\Omega f \biggl(x,\frac{\di\mu}{\di\Lcal^d}(x)\biggr) \dd x + \int_\Omega f^\# \biggl(x,\frac{\di\mu^s}{\di|\mu^s|}(x)\biggr) \dd|\mu^s|(x),
\end{equation}
defined for finite vector Radon measures $\mu \in \Mcal(\Omega;\R^N)$ on $\Omega$ with values in $\R^N$.
Here, $f \colon {\Omega} \times \R^N \to \R$ is a Borel integrand that 
%is uniformly Lipschitz in the second argument and 
has \emph{linear growth at infinity}, i.e.,
\[
  \abs{f(x,A)} \leq M(1+\abs{A})  \qquad
  \text{for all $(x,A) \in {\Omega} \times \R^N$,}
\]
whereby the \emph{(generalized) recession function}
\[
f^\#(x,A) \coloneqq \limsup_{\substack{x' \to x\\A' \to A\\t \to \infty}} \; \frac{f(x',tA')}{t}, \qquad (x,A) \in \cl\Omega \times \R^N,
\]
takes only finite values. Furthermore, on the candidate measures $\mu \in \Mcal(\Omega;\R^N)$ we impose the $k$th-order linear PDE  side constraint
\[
  \A \mu := \sum_{|\alpha| \leq k} A_\alpha \partial^\alpha \mu = 0  \qquad
  \text{in the sense of distributions.}
\]
The coefficient matrices $A_\alpha \in \R^{n}\otimes \R^N$ are assumed to be constant and we write $\partial^\alpha = \partial^{\alpha_1}_1 \dots \partial^{\alpha_d}_d$ for every multi-index $\alpha = (\alpha_1,\dots,\alpha_d) \in (\N \cup \{0\})^d$ with $\abs{\alpha} := \abs{\alpha_1} + \cdots + \abs{\alpha_d} \leq k$. We call measures $\mu \in \Mcal(\Omega;\R^N)$ with $\Acal \mu = 0$ in the sense of distributions \emph{$\A$-free}.

% \mnote{\color{red} A: This sentence reads a little bit disconnected from the context of the paragraph, but I am not sure how to amend it} 
%{\color{red}The space $\Mcal(\Omega;\R^N)$ contains all finite vector Radon measures on $\Omega$ with values in $\R^N$.}\mnote{F: I kept it as it is because I think the $\Mcal$ / $\Mcal_\loc$ notation is a bit more standard in this field (at least as I know it).} 

We will also assume that $\A$ satisfies \emph{Murat's constant rank condition} (see~\cite{Murat81,FonsecaMuller99}), that is, there exists $r \in \N$ such that
\begin{equation}\label{eq:constrank}
  \rank\Abb^k(\xi) = r \qquad
  \text{for all $\xi \in \Sbb^{d-1}$,}
\end{equation}
where
\[
  \Abb^k (\xi) := (2\pi\ii)^k \sum_{|\alpha| = k}\xi^\alpha A_\alpha,  \qquad
  \xi^\alpha = \xi^{\alpha_1}_1\cdots\xi^{\alpha_d}_d,
\]
is the \emph{principal symbol} of $\A$.  We also recall the notion of {\em wave cone} associated to \(\A\), which plays a fundamental role in the study of \(\A\)-free fields and first originated in the theory of compensated compactness~\cite{Tartar79,Tartar83,Murat78,Murat79,Murat81,DiPerna85}. 

\begin{definition}\label{def:wc}
Let \(\A\) be a $k$th-order linear PDE operator as above. The {\em wave cone} associated to \(\A\) is the set
% of \enquote{admissible amplitudes}
\[
\Lambda_\A := \bigcup_{|\xi|= 1} \ker \Abb^k(\xi)   \; \subset \R^N.
\]
\end{definition}

Note that the wave cone contains those amplitudes along which it is possible to construct highly oscillating \(\A\)-free fields. More precisely, if \(\A\) is homogeneous, i.e., \(\A= \sum_{|\alpha|=k} A_\alpha \partial^\alpha\), then \(P_0\in \Lambda_\A\) if and only if there exists \(\xi\ne 0\) such that 
\[
\A (P_0 \, h(x \cdot \xi)) = 0\qquad \text{for all $h \in \Crm^k(\R)$}.
\]

%\mnote{G: added the definition of WC}
%%  role in the sequel 
%%In general, this assumption cannot be omitted, see Remark~\ref{rem:MuratCR} below. 
%This very general setup originated in the Tartar--Murat theory of compensated compactness~\cite{Tartar79,Tartar83,Murat78,Murat79,Murat81,DiPerna85}. In the context of integral functionals the theory relevant to us was developed mainly in~\cite{FonsecaMuller99,FonsecaMullerPedregal98,FonsecaLeoniMuller04,BaiaChermisiMatiasSantos13}.

Our first main theorem concerns the case when $f$ is $\A^k$-quasiconvex in its second argument, where
\[
  \A^k := \sum_{|\alpha| = k} A_\alpha \partial^\alpha
\]
is the \emph{principal part} of $\A$. 
Recall from~\cite{FonsecaMuller99} that a Borel function $h \colon \R^N \to \R$ is called \emph{$\A^k$-quasiconvex} if
\[
  h(A) \leq \int_Q h(A + w(y)) \dd y
\]
for all $A \in \R^N$ and all $Q$-periodic $w \in \Crm^\infty(Q;\R^N)$ such that $\A^k w = 0$ and $\int_Q w \dd y = 0$, where $Q := (-1/2,1/2)^d$ is the open unit cube in $\R^d$.

In order to state our first result, we %well-defined 
introduce the notion of \emph{strong recession function} of $f$, which for $(x,A) \in \cl \Omega \times \R^N$ is defined as
\begin{equation}\label{eq:f_infty}
f^\infty(x,A) := \lim_{\substack {x' \to x \\ A' \to A \\ t \to \infty}} \frac{f(x',tA')}{t}, \qquad (x,A) \in \cl{\Omega} \times \R^N,
\end{equation}
provided the limit exists.

\begin{theorem}[lower semicontinuity] \label{thm:lsc}
Let $f \colon \Omega \times \R^N \to [0,\infty)$ be a continuous integrand. Assume that \(f\) has linear growth  at infinity, that is Lipschitz in its second argument, and that $f(x,\frarg)$ is $\A^k$-quasiconvex for all $x \in \Omega$. Further assume that either
\begin{itemize}
\item[(i)] $f^\infty$ exists in $\Omega \times \R^N$, or
\item[(ii)] $f^\infty$ exists in $\Omega \times \mathrm{span}\,\Lambda_{\A}$, and 
there exists a modulus of continuity $\omega \colon [0,\infty) \to [0,\infty)$ (increasing, continuous, $\omega(0) = 0$) such that
\begin{equation}\label{eq:modulus}
  |f(x,A) - f(y,A)|\le \omega(|x - y|)(1 + |A|)  \qquad
\text{for all $x,y \in {\Omega}$, $A \in \R^N$.}
\end{equation}
\end{itemize}
%%and 
%that the strong recession function 
%\begin{equation}\label{eq:f_infty}
%f^\infty(x,A) \coloneqq \lim_{t \to \infty} \; \frac{f(x,tA)}{t}\quad \text{exists for all \; $(x,A) \in \Omega \times \mathrm{span}\,\Lambda_{\A}$}.
%\end{equation}
%that is such that there exists a modulus of continuity $\omega$ with
%\begin{equation}\label{eq:modulus}
  %|f(x,A) - f(y,A)|\le \omega(|x - y|)(1 + |A|)  \qquad
%\text{for all $x,y \in {\Omega}$, $A \in \R^N$.}
%\end{equation}
Then, the functional 
%\mnote{A: I propose the either use $f^\#$ or denote the functional as $\Fcal^\infty$ (or something related) since it represents a conflict with the definition of $\Fcal$ given in the introduction}

\[
  \Fcal[\mu] :=  \int_\Omega f \biggl(x,\frac{\di\mu}{\di\Lcal^d}(x)\biggr) \dd x + \int_\Omega f^\infty \biggl(x,\frac{\di\mu^s}{\di|\mu^s|}(x)\biggr) \dd|\mu^s|(x)
\]
is sequentially weakly* lower semicontinuous on the space
\[
  \Mcal(\Omega;\R^N) \cap \ker \A := \setb{ \mu \in \Mcal(\Omega;\R^N) }{ \A \mu = 0 }.
\]
%if and only if $f(x,\frarg)$ is $\A^k$-quasiconvex for every $x \in \Omega$. 
\end{theorem}

Note that according to~\eqref{eq:sinwc} below, \(\Fcal[\mu]\) is well-defined for $\mu \in \Mcal(\Omega;\R^N) \cap \ker \A$ since the strong recession function is computed only at  amplitudes that  belong to \(\spn \Lambda_\A\). 

\begin{remark}
The $\A^k$-quasiconvexity of $f(x,\frarg)$ is not only a sufficient, but also a \emph{necessary} condition for the sequential weak* lower semicontinuity of $\Fcal$ on $\M(\Omega;\R^N) \cap \ker \A$. In the case of first-order partial differential operator, the proof of the necessity can be found in \cite{FonsecaMuller99}; the proof in the general case follows by verbatim repeating the same arguments.
\end{remark}  

\begin{remark}[asymptotic $\A$-free sequences] \label{thm:lsc_ext}
The conclusion of Theorem~\ref{thm:lsc} extends to sequences that are only asymptotically $\A$-free, that is,
\[
\Fcal[\mu] \le \liminf_{j \to \infty} \Fcal[\mu_j] 
\]
for all sequences $(\mu_j) \subset \Mcal(\Omega;\R^N)$ such that 
\[
\text{$\mu_j \toweakstar \mu$ in $\M(\Omega;\R^N)$ \quad and \quad $\A \mu_j \to 0$ in $\Wrm^{-k,q}(\Omega;\R^n)$} 
\]
for some $1 < q < d/(d-1)$ if $f(x,\frarg)$ is $\A^k$-quasiconvex for all $x \in \Omega$.

%The asymptotic constraint with respect to the $\Wrm^{-k,q}$ is natural for the \emph{pure} constraint $\A^k \mu = 0$. Indeed, for all such $q$'s, Morrey's inequality gives (for a Lipschitz domain $\omega \subset \R^d$)
%\[
%\M(\omega;\R^N) \cembed \Wrm^{-k+1,q}(\omega;\R^N).
%\]
%In particular, if $\mu_j \toweakstar 0$ in $\M(\omega;\R^N)$ and $\A^k \mu = 0$  on $\omega$, then 
%\[
%\A \mu_j \to 0\quad \text{in $\Wrm^{-k,q}(\omega;\R^n)$}.
%\]
\end{remark}

%{\color{red} G: Maybe we should add what is our definition of Radon measure and what is our notion of convergence, in particular we shall specify if for us Radon measure have finite or locally finite mass, i.e., if we look to the dual of $C_0$ or of $\Crm_c$. In the second case (that I would prefer) for the proof we are doing we shall also assume that $\sup_j |\mu_j|(\Omega)$ is finite (in the first case is a consequence of the Uniform Bounded Principle. If we agree on the second option we shall also remove all the notion of $\mathcal M_{\rm loc}$ since now it makes no sense.}

Notice that $f^\infty$ in~\eqref{eq:f_infty} is a limit and, contrary to $f^\#$, it may fail to exist for $A \in (\spn \Lambda_\A) \setminus \Lambda_\A$ (for $A \in \Lambda_\A$ the existence of $f^\infty(x,A)$ follows from the $\A^k$-quasiconvexity, see Corollary~\ref{cor:coneconvexity}). If we remove the assumption that $f^\infty$ exists for points in the subspace generated by the wave cone $\Lambda_\A$, we still have the partial lower semicontinuity result formulated in Theorem \ref{thm:lsc_partial} below (cf.~\cite{FonsecaLeoniMuller04}).

\begin{remark} \label{rem:special}
As special cases of Theorem~\ref{thm:lsc} we get, among others, the following well-known results:\begin{enumerate}[(i)]
\item For $\A = \curl$, one obtains  BV-lower semicontinuity results in the spirit of  Ambrosio--Dal Maso~\cite{AmbrosioDalMaso92} and Fonseca--M\"{u}ller~\cite{FonsecaMuller93}.

% To avoid further technicalities we have decided not to consider unsigned integrands as in%However, up to adding a boundary term, similar results to ones presented here for unsigned integral functionals {\color{red}[Cite Adolfo's thesis]}.
\item For  $\A=\curl \curl$, where 
\[
  \curl\curl \mu:= \biggl( \sum_{i=1}^d \partial_{ik} \mu_{i}^j+\partial_{ij} \mu_{i}^k-\partial_{jk} \mu_{i}^i-\partial_{ii} \mu_{j}^k \biggr)_{j,\,k \, = \, 1,\ldots,d}
\]
is the second order operator expressing the Saint-Venant compatibility conditions (see~\cite[Example 3.10(e)]{FonsecaMuller99}), we re-prove the lower semicontinuity and relaxation theorem in the space of functions of bounded deformation (BD) from~\cite{Rindler11}.
\item For first-order operators $\A$, a  similar result was proved in~\cite{BaiaChermisiMatiasSantos13}.
% for weakly* convergent sequences of measures under the (local) convergence assumption $\A \mu_j \to 0$ in $\Wrm^{-k,q}_\loc(\Omega;\R^n)$. 
\item[(iv)] Earlier work in this direction is in~\cite{FonsecaMuller99,FonsecaLeoniMuller04}, but there the singular (concentration) part of the functional was not considered.
\end{enumerate}

%{\color{red}
%G: I think that under  our assumptions we can substitute\eqref{eq:modulus} with the following (weaker) assumption:
%$
%|f(x,A) - f(y,A)|\le \omega(|x - y|)(1 + |A|)
%$
%This will require just notational adjustments here and there.  Note indeed we always consider sequences such that $|\mu_j|(\Omega)$ is uniformly bounded, and in this case the above condiion is enough (in~\cite{ FonsecaLeoniMuller04} I think the condition is used to deal with non-coercive functionals)}
\end{remark}

\begin{theorem}[partial lower semicontinuity]\label{thm:lsc_partial}
Let $f \colon \Omega \times \R^N \to [0,\infty)$ be a continuous integrand such that $f(x,\frarg)$ is $\A^k$-quasiconvex for all $x \in \Omega$. Assume that \(f\) has linear growth  at infinity and is Lipschitz in its second argument, uniformly in $x$. Further, suppose that there exists a modulus of continuity $\omega$  as in~\eqref{eq:modulus}.
%\begin{equation}\label{eq:modulus}
%  |f(x,A) - f(y,A)|\le \omega(|x - y|)(1 + |A|)  \qquad
%\text{for all $x,y \in {\Omega}$, $A \in \R^N$.}
%\end{equation}
Then, 
\begin{align*}
\int_\Omega f \bigg(x& , \frac{\dd \mu}{\dd \Lcal^d}(x)\bigg)  \dd x \le 
\liminf_{j \to \infty} \Fcal^\#[\mu_j]
\end{align*}
for all sequences $\mu_j \toweakstar \mu$ in $\M(\Omega;\R^N)$ such that $\A \mu_j \to 0$ in $\Wrm^{-k,q}(\Omega;\R^N)$.
Here,
\[
\Fcal^\#[\mu] \coloneqq \int_\Omega  f \biggl(x,\frac{\di\mu}{\di\Lcal^d}(x)\biggr) \dd x 
 + \int_\Omega  f^\# \biggl(x,\frac{\di\mu^s}{\di|\mu^s|}(x)\biggr) \dd|\mu^s|(x),
\]
and $1 < q < d/(d-1)$.
%\[
% \overline{\Gcal}[\mu] := \inf\setB{\liminf_{j \to \infty} \Fcal[u_j]}{\text{$\mu_j \toweakstar \mu$ and $\A u_j \to 0$ in $\Wrm^{-k,q}$}}.
%\]
\end{theorem}

% for which it also holds that
%\[
%  \Fcal[\mu] \leq \liminf_{j \to \infty}\Fcal[\mu_j].
%\]

If we dispense with the assumption of $\A^k$-quasiconvexity on the integrand, we have the following two relaxation results:

%\begin{theorem}[Relaxation I]\label{thm:relax}
%Let $f \colon \cl{\Omega} \times \R^N \to \R$ as in Theorem~\ref{thm:lsc}. Assume that the recession function $(Q_{\A^k}f)^\infty$ exists and is (jointly) continuous. Then the lower semicontinuous envelope
%\[
%\overline{ \Gcal}(\mu) := \inf\left\{\liminf_{m \to \infty} \Gcal(u_n) : u_n \toweakstar \mu \; \text{in $\Mcal(\Omega;\R^N)$ and $\A u_n \to 0$ in $\Wrm^{-k,q}(\Omega;\R^N)$} \right\}
%\]
%of the functional 
%\[
%\Gcal[u] := \int_{\Omega} f(x,u(x))  \dd x, \qquad u \in \Lrm^1(\Omega;\R^N),
%\]
%is given by the functional
%	\begin{align*}
%\int_{\Omega} (Q_{\A^k}f(x,\frarg))\biggl( \frac{\di \mu}{\di |\mu|}(x) \biggr)\dd x  + \int_{\Omega} (Q_{\A^k}f(x,\frarg))^\infty \left(\frac{\di \mu^s}{\di |\mu^s|}(x)\right)\dd |\mu^s|(x)  
%	\end{align*}
%%	where for an upper semicontinuous function $g\colon \R^d \to \R$ we have set 
%%	\[
%%	\overline{\Gcal}[\mu] := Q_{\A}g(z) := \left\{ \int_{Q} f(z + w(y)) \dd y : w \in \Crm^\infty_\per(\Omega;\R^N) \cap \ker \A, \int_Q w(y) \dd y = 0 \right\},
%%	\]
%%	is the $\A$-free quasiconvex envelope of $g$.
%\end{theorem}

\begin{theorem}[relaxation] \label{thm:relax}
%\mnote{A: changed style, I think there were too many {\it that}'s}
Let $f \colon \Omega \times \R^N \to [0,\infty)$ be a continuous integrand that is Lipschitz in its second argument, uniformly in $x$. Assume also that $f$ has linear growth at infinity (in its second argument) and is such that there exists a modulus of continuity $\omega$ as in~\eqref{eq:modulus}. Further, suppose that $\A$ is a homogeneous PDE operator and that the strong recession function
\[
\text{$f^\infty(x,A)$ \, exists for all \, $(x,A) \in \Omega \times \textnormal{span}\,\Lambda_{\A}$}.
\]
%where $Q_{\A^k}f(x,\frarg)$ denotes the $\A^k$-quasiconvex envelope of $f(x,\frarg)$ with respect to the second argument (see Definition~\ref{def:quasi} below) and $(Q_{\A^k}f)^\infty(x,\frarg) := (Q_{\A^k}f(x,\frarg))^\infty$ is its strong recession function with respect to the second argument.
%and that is such that there exists a modulus of continuity $\omega$ as in~\eqref{eq:modulus}.
%for which 
%\begin{equation*}
%f(x,A) - f(y,A)\le \omega(|x - y|)(1 + f(y,A))  \qquad
%\text{for all $x,y \in \cl{\Omega}$, $A \in \R^N$.}
%\end{equation*}
Then, for the functional 
\[
  \Gcal[u] := \int_{\Omega} f(x,u(x))  \dd x, \qquad u \in \Lrm^1(\Omega;\R^N),
\]
the (sequentially) weakly* lower semicontinuous envelope of $\Gcal$, defined to be
\begin{align*}
  \overline{\Gcal}[\mu] := \inf\setB{\liminf_{j \to \infty} \Gcal[u_j]}{ &(u_j) \subset \Lrm^1(\Omega;\R^N), \; \text{$u_j \, \Lcal^d \toweakstar \mu$ in $\M(\Omega;\R^N)$} \\
  &\text{and $\A u_j \to 0$ in $\Wrm^{-k,q}$}},
\end{align*}
where $\mu \in \M(\Omega;\R^N) \cap \ker \A$ and $1 < q < d/(d-1)$, is given by
\begin{align*}
  \overline{\Gcal}[\mu] &= \int_{\Omega} Q_{\A}f \biggl( x, \frac{\di \mu}{\dd \Lcal^d}(x) \biggr)\dd x
  + \int_{\Omega} (Q_{\A}f)^\# \left( x, \frac{\di \mu^s}{\di |\mu^s|}(x)\right)\dd |\mu^s|(x).
\end{align*}
Here, $Q_{\A}f(x,\frarg)$ denotes the $\A$-quasiconvex envelope of $f(x,\frarg)$ with respect to the second argument (see Definition~\ref{def:quasi} below).
\end{theorem}

If we want to find the relaxation in the space $\Mcal(\Omega;\R^N) \cap \ker \A$ we need to assume that $\Lrm^1(\Omega;\R^N) \cap \ker \A$ is dense in $\M(\Omega;\R^N) \cap \ker \A$ with respect to a finer topology than the natural weak* topology (in this context also see~\cite{AR2}).
%assume that $\A$ is an homogeneous operator.

\begin{theorem} \label{thm:relax2}
Let $f \colon \Omega \times \R^N \to [0,\infty)$ be a continuous integrand that is Lipschitz in its second argument, uniformly in $x$. Assume also that $f$ has linear growth at infinity (in its second argument) and is such that there exists a modulus of continuity $\omega$ as in~\eqref{eq:modulus}.
Further, suppose that $\A$ is a homogeneous PDE operator, that the strong recession function
\[
\text{$f^\infty(x,A)$ \, exists for all \, $(x,A) \in \Omega \times \mathrm{span}\,\Lambda_{\A}$},
\]
%Further assume that the strong recession function of its $\Acal^k$-quasiconvex envelope $(Q_{\A^k}f)^\infty$ exists on $\mathrm{span} \, \Lambda_\A$. % and that is such that there exists a modulus of continuity $\omega$ as in~\eqref{eq:modulus}.
%\begin{equation*}
%f(x,A) - f(y,A)\le \omega(|x - y|)(1 + f(y,A))  \qquad
%\text{for all $x,y \in \cl{\Omega}$, $A \in \R^N$.}
%\end{equation*}
and that for all $\mu \in \M(\Omega;\R^N) \cap \ker \A$ there exists a sequence $(u_j) \subset \Lrm^1(\Omega;\R^N) \cap \ker \A$ such that
\begin{equation}\label{density}
u_j \, \Lcal^d \toweakstar \mu \quad \text{in $\M(\Omega;\R^N)$} \quad \text{and} \quad \area{u_j \, \Lcal^d}(\Omega) \to \area{\mu}(\Omega),
\end{equation}where $\barea$ is the area functional defined in~\eqref{eq:area}. 
Then, for the functional
\[
  \Gcal[u] := \int_\Omega f(x,u(x)) \dd x,   \qquad u \in \Lrm^1(\Omega;\R^N) \cap \ker \A, 
\] 
the weakly* lower semicontinuous envelope of $\Gcal$, defined to be
\begin{align*}
  \overline{\Gcal}[\mu] := \inf\setB{\liminf_{j \to \infty} \Gcal[u_j]}{ (u_j) \subset \Lrm^1(\Omega;\R^N) \cap \ker \A, \; \text{$u_j \, \Lcal^d \toweakstar \mu$ in $\M(\Omega;\R^N)$}},
\end{align*}
%of the functional
%\[
%  \Gcal[u] := \int_\Omega f(x,u(x)) \dd x,   \qquad u \in \Lrm^1(\Omega;\R^N) \cap \ker \A, 
%\] 
%with respect to weak*-convergence in the space $\Mcal(\Omega;\R^N) \cap \ker \A $, 
is given by
\begin{align*}
  \overline{\Gcal}[\mu] &= \int_{\Omega} Q_{\A}f \biggl( x, \frac{\di \mu}{\di \Lcal^d}(x) \biggr)\dd x
  + \int_{\Omega} (Q_{\A}f)^\# \biggr( x, \frac{\di \mu^s}{\di |\mu^s|}(x)\biggr)\dd |\mu^s|(x).
\end{align*}
\end{theorem}

\begin{remark}[density assumptions]\label{rem:assumption} Condition~\eqref{density} is automatically fulfilled in the following cases:
\begin{itemize}
	\item[(i)] For $\A = \curl$, the approximation property (for general domains) is proved in the appendix of~\cite{KristensenRindler10} (also see Lemma~B.1 of~\cite{Bildhauer03book} for Lipschitz domains). The same argument further shows the area-strict approximation property in the BD-case (also see Lemma~2.2 in~\cite{BarrosoFonsecaToader00} for a result which covers the strict convergence).%\mnote{A: Why is strict convergence in (i)-BD enough? No area-strict needed or is it just that their recovery sequence is actually also area-strict convergent?\\F: I'm just saying that area-strict approximation also holds in BD, even if this is not written down anywhere (AFAIK)}
	
%	. Indeed in both cases condition  Condition~\eqref{density} reduces to the density of $\Crm^\infty $-functions in $\BV(\Omega)$ (respectively $\BD(\Omega)$. with respect of the ``area'' strict convergence. This can be achieved by following the classical Meyers-Serrin  approximation procedure, see for instance~\cite{}. 
%	
%	
%	the spaces $\BV(\Omega;\R^d), \BD(\Omega;\R^d)$ 
%	%of functions with bounded variation obtained 
%	by setting $\A = \curl$ and $\A = \curl \curl$ respectively (the $\BV$ case is a simple result of the properties of the extension operator for Lipschitz domains while the case for $\BD$ has been treated in~\cite{BarrosoFonsecaToader00} );
	\item[(ii)] If $\Omega$ is a {\it strictly star-shaped} domain, i.e., there exists $x_0 \in \Omega$ such that 
	\[
	\overline{(\Omega - x_0)} \subset t(\Omega - x_0) \qquad \text{for all $t>1$,}
	\]
	then~\eqref{density} holds for every homogeneous operator $\A$. Indeed, for $t>1$ we can  consider the  dilation of $\mu$ defined on  $t(\Omega - x_0)$ and then mollify it at a sufficiently small scale. We refer for instance to~\cite{Muller87} for details.
%	 where such an assumption is made to deal with an homogenization problem.
\end{itemize}
\end{remark}

%{\color{red}
As a consequence of Theorem~\ref{thm:relax2} and of Remark~\ref{rem:assumption} we explicitly state the following corollary, which extends the lower semicontinuity result of~\cite{Rindler11} into a full relaxation result. The only other relaxation result in this direction, albeit for \emph{special} functions of bounded deformation, seems to be in~\cite{BarrosoFonsecaToader00}; other results in this area are discussed in~\cite{Rindler11} and the references therein.

\begin{corollary} \label{cor:BD}
Let $f \colon \Omega \times \R^{d\times d}_\sym  \to [0,\infty)$ be a continuous integrand, uniformly Lipschitz in the second argument, with linear growth at infinity, and such that there exists a modulus of continuity $\omega$ as in~\eqref{eq:modulus}.
Further, suppose that the strong recession function
\[
\text{$f^\infty(x,A)$ \, exists for all \, $(x,A) \in\Omega \times \R^{d\times d}_\sym$}.
\]
%Further assume that the strong recession function of its $\Acal^k$-quasiconvex envelope $(Q_{\A^k}f)^\infty$ exists on $\mathrm{span} \, \Lambda_\A$. % and that is such that there exists a modulus of continuity $\omega$ as in~\eqref{eq:modulus}.
%\begin{equation*}
%f(x,A) - f(y,A)\le \omega(|x - y|)(1 + f(y,A))  \qquad
%\text{for all $x,y \in \cl{\Omega}$, $A \in \R^N$.}
%\end{equation*}
%\mnote{A: rearranged a bit}{\color{red}
Consider the functional 
\[
  \Gcal[u] := \int_\Omega f(x, \Ecal u(x)) \dd x,
\]
defined for $u \in \LD(\Omega) := \setn{ u \in \BD(\Omega) }{ E^s u = 0 }$, where $Eu \coloneqq (Du+Du^T)/2 \in \Mcal(\Omega;\R^{d\times d}_\sym)$ is the symmetrized distributional derivative of $u \in \BD(\Omega)$ and where
\[
  Eu = \Ecal u \, \Lcal^d \restrict \Omega + E^su
\]
is its Radon--Nikod\'{y}m decomposition with respect to $\Lcal^d$.%}

%Let us define for \(u\in \LD(\Omega) := \setn{ u \in \BD(\Omega) }{ E^s u = 0 } \), where \(Eu=(Du+Du^T)/2 \in \Mcal(\Omega;\R^{d\times d}_\sym)\) is the \emph{symmetrized distributional derivative} of \(u \in \BD(\Omega)\), which has the Lebesgue--Radon--Nikod\'{y}m decomposition
%\[
%  Eu = \Ecal u \, \Lcal^d + \frac{\di E^su }{\di |E^su|} |E^su|,  \qquad
%  \Ecal u = \frac{1}{2} (\nabla u + \nabla u^T),
%\]
%the functional
%\[
%  \Gcal[u] := \int_\Omega f(x, \Ecal u(x)) \dd x,
%\]}
Then, the lower semicontinuous envelope of \(\Gcal\) 
with respect to weak*-convergence in \(\BD(\Omega)\) is given by the functional
\begin{align*}
  \overline{\Gcal}[u] &\coloneqq \int_{\Omega} SQf ( x,  \Ecal u(x))\dd x
  + \int_{\Omega} (SQf)^\# \biggr( x, \frac{\di E^su }{\di |E^su|}(x)\biggr)\dd |E ^su |(x),
\end{align*}
where $SQf$ denotes the \emph{symmetric-quasiconvex envelope} of $f$ with respect to the second argument (i.e.,\ the $\curl\curl$-quasiconvex envelope of $f(x,\frarg)$ in the sense of Definition~\ref{def:quasi}).
\end{corollary}%}

Our proofs are based on new tools to study singularities in PDE-constrained measures. Concretely, we exploit the recent developments on the structure of $\A$-free measures obtained in~\cite{DePhilippisRindler16}. We remark that the study of the singular part -- up to now the most complicated argument in  the proof -- now only requires a fairly straightforward (classical) convexity argument. More precisely, the main theorem of~\cite{KirchheimKristensen16} establishes that the restriction of $f^\#$ to 
the linear space spanned by  
the wave cone
%\[
% := \bigcup_{|\xi|= 1} \ker \Abb^k(\xi) \subset \R^N
%\]
is in fact \emph{convex} at all points of  \(\Lambda_\A\)
(in the sense that a supporting hyperplane exists). By~\cite{DePhilippisRindler16}, 
\begin{equation}\label{eq:sinwc}
  \frac{\di\mu^s}{\di|\mu^s|}(x) \in \Lambda_\A  \qquad
  \text{for $\abs{\mu^s}$-a.e.\ $x \in \Omega$.}
\end{equation}
Thus, combining these two assertions, we gain classical convexity for  $f^\#$ at singular points, 
%\mnote{A: not so sure about the structure of this sentence towards the end} 
which can be exploited via the theory of 
generalized Young measures developed in~\cite{DiPernaMajda87,AlibertBouchitte97,KristensenRindler10}.

\begin{remark}[different notions of recession function] \label{rmk:rec}
Note that both in Theorem~\ref{thm:lsc} and Theorem~\ref{thm:relax} the existence of the {\em strong} recession function \(f^\infty\) is assumed, in contrast with the results in~\cite{AmbrosioDalMaso92, FonsecaMuller93, BaiaChermisiMatiasSantos13} where this is not imposed. 

The need for this assumption comes from the use of Young measure techniques %in the proof
which seem to be  better suited to deal with the singular part of the measure, as we already discussed above. In the aforementioned references a direct blow up approach is instead performed and this allows to deal directly with the functional in~\eqref{Fintro}. %\mnote{A: modified the content a bit, I think there was an overlooked concept}{\color{red} 
The blow-up techniques, however, rely strongly on the fact that  \(\A\) is a homogeneous first-order operator. Indeed, it is not hard to check that for all ``elementary'' \(\A\)-free measures of the form
\[
\mu=P_0 \lambda, \qquad \text{where}\qquad
P_0\in \Lambda_\A,\; \lambda\in \Mcal^+(\R^d),
\]
the scalar measure \(\lambda\) is necessarily translation invariant along orthogonal directions to the {\em characteristic set}
\[
\Xi(P_0) \coloneqq \setb{\xi \in \R^d}{P_0 \in \ker \Abb(\xi)},
\]
which turns out to be a subspace of $\R^d$ whenever $\A$ is a first-order operator. 
The subspace structure and the aforementioned translation invariance is then used to perform homogenization-type arguments. Due to the lack of linearity of the map 
\[
\xi \mapsto \Abb^k(\xi) \quad \text{for $k > 1$},
\] 
the structure of elementary \(\A\)-free measures for general operators is more complicated and not yet fully understood (see however~\cite{Rindler11, DePhilippisRindler16b} for the case \(\A=\curl\curl\)). %}
%Indeed, in this case  an ``elementary'' \(\A\)-free measure of the form
%\[
%\mu=P_0 \lambda, \qquad \text{where}\qquad
%P_0\in \Lambda_\A,\; \lambda\in \Mcal(\R^d),
%\]
%it is easy to show that the scalar measure \(\nu\) necessarily has to be translation invariant along the direction of the {\em characteristic set}
%\[
%\Xi(P) \coloneqq \setb{\xi \in \R^d}{P \in \ker \Abb(\xi)}.
%\]
%This translation invariance is then used to perform homogenization-type arguments.
%For general operator \(\A\) the structure of elementary \(\A\)-free measure is more complicated and not fully understood (see however~\cite{Rindler11, DePhilippisRindler16b} for the case \(\A=\curl\curl\)).
This prevents, at the moment, the use of ``pure'' blow-up techniques and forces us to pass through the combination of the results of~\cite{DePhilippisRindler16,KirchheimKristensen16} 
with the Young measure approach.
% for which one need to assume the existence of \(f^\infty\).
\end{remark}

%\mnote{A: removed that we attack the bulk part by classical methods}%at singular points. The understanding of the bulk part, on the other hand, remains classical and follows a similar strategy to the one in~\cite{FonsecaLeoniMuller04}.

%Additionally, we also decided to include a very short and concise proof of the lower semicontinuity theorem (Theorem~\ref{thm:lsc}), which only employs the Young measure framework (cf. Section~\ref{sc:lsc_special}). However, in this case we need to make the extra assumption that the strong recession function $f^\infty$ of $f$ (see~\eqref{eq:finfty}) exists and is jointly continuous.
%As a matter of fact, due to the possibility of diffuse concentrations, it is not clear to us whether a pure Young measure approach suffices to prove Theorem~\ref{thm:lsc} in its full generality.

This paper is organized as follows: First, in Section~\ref{sc:notation}, we introduce all the necessary notation and prove auxiliary results. Then, in Section~\ref{sc:Jensen}, we establish the central Jensen-type inequalities, which immediately yield the proof of Theorems~\ref{thm:lsc} and~\ref{thm:lsc_partial} in  Section~\ref{sc:lsc_special}. The proofs of Theorems~\ref{thm:relax} and~\ref{thm:relax2} are given in  Section~\ref{sc:relax_proof}. 
%In the final Section~\ref{sc:lsc_general} we discuss some open problems  followed by the proof of Theorem~\ref{thm:lsc_partial}. 

\subsection*{Acknowledgments} 

A.~A.-R.\ is supported by a scholarship from the Hausdorff Center of Mathematics and the University of Bonn through a DFG grant; the research conducted in this paper forms part of his Ph.D.\ thesis at the University of Bonn. G.~D.~P.\ is supported by the MIUR SIR-grant ``Geometric Variational Problems" (RBSI14RVEZ). F.~R.\ acknowledges the support from an EPSRC Research Fellowship on ``Singularities in Nonlinear PDEs'' (EP/L018934/1).

The authors would like to thank the anonymous referee for her/his careful reading of the manuscript which led to a substantial improvement of the presentation.

\section{Notation and  preliminaries} \label{sc:notation}
%
%\mnote{A: changed the structure of this sentence (felt a little bit difficult to read).\\
%Also added the second sentence.}
%We collect notation and prove several technical lemmas in this section.
We write $\M(\Omega;\R^N)$ and $\M_\loc(\Omega;\R^N)$ to denote the spaces of \emph{bounded Radon measures} and \emph{Radon measures} on $\Omega\subset \R^d$ and with values in $\R^N$, which are the duals of $\Crm_0(\Omega;\R^N)$ and $\Crm_c(\Omega;\R^N)$ respectively. Here, $\Crm_0(\Omega;\R^N)$ is the completion of $\Crm_c(\Omega;\R^N)$ with respect to the $\|\frarg\|_\infty$-norm, and, in the second case, $\Crm_c(\Omega;\R^N)$ is understood as the inductive limit of the Banach spaces $\Crm_0(K_m)$ where each $K_m$ is a compact subset of $\R^d$ and $K_m \nearrow \Omega$. The set of \emph{probability measures} over a locally compact space $X$ shall be denoted by 
\[
\M^1(X) \coloneqq \setB{\mu \in \M(X)}{\text{$\mu$ is a positive measure, and $\mu(X) = 1$}}.
\] 
We will often make use of the following metrizability principles:
\begin{enumerate}
\item Bounded sets of $\M(\Omega;\R^N)$ are metrizable in the sense that there exists a metric $d$ which induces the weak* topology, that is,
\[
\qquad \sup_{j \in \Nbb} |\mu_j|(\Omega) < \infty \quad \text{and} \quad d(\mu_j,\mu) \to 0 \quad \Leftrightarrow \quad \mu_j \toweakstar \mu \quad \text{in $\M(\Omega;\R^N)$}.
\]
\item There exists a complete metric $d$ on $\M_\loc(\Omega;\R^N)$. Moreover, convergence with respect to this metric coincides with the weak* convergence of Radon measures (see Remark 14.15 in \cite{Mattila95book}).
\end{enumerate}

%\mnote{A: At the end is probably better not to identify $u$ with $u \Lcal^d$. We have differentiated them all along...}%Where no confusion can arise, we shall identify $u \in \Lrm^1(\Omega;\R^N)$ with the measure $u \Lcal^d \in \M(\Omega;\R^N)$.
We write the Radon--Nikod\'{y}m decomposition of a measure $\mu \in \Mcal(\Omega;\R^N)$ as
\begin{equation} \label{eq:RN}
\mu = \frac{\di \mu}{\di \Lcal^d}\, \Lcal^d \restrict \Omega + \mu^s,
\end{equation}
where $\frac{\di \mu}{\di \Lcal^d} \in \Lrm^1(\Omega;\R^N)$ and $\mu^s \in \M(\Omega;\R^N)$ is singular with respect to $\Lcal^d$.

%In order to keep a simple presentation, we will often identify $u \in \Lrm^1(\Omega;\R^N)$ with the measure $u \Lcal^d \in \M(\Omega;\R^N)$.
 
%\Sout{Denote by $\Mcal(\Omega;\R^N)$ and $\Mcal_\loc(\Omega;\R^N)$ the spaces of vector Radon measures on $\Omega \subset \R^d$ (which is measurable) with values in $\R^N$ that are finite or locally finite, respectively}.

\subsection{Integrands and Young measures }\label{sec: Young}
For $f \in \Crm(\Omega \times \R^N$) we define the transformation 
\[
(Sf)(x,\hat{A}) := (1 - |\hat{A}|)f\left(x,\frac{\hat{A}}{1 - |\hat{A}|}\right), \qquad (x,\hat A) \in \cl{\Omega}\times  \Bbb^N,
\]
where $\Bbb^N$ denotes the open  unit ball in $\R^N$. Then, $Sf \in \Crm(\Omega \times \Bbb^N)$. We set 
\[
\E(\Omega;\R^N):= \setb{ f \in \Crm(\Omega \times \R^N) }{ \text{$Sf$ extends to $\Crm(\overline{\Omega \times  \Bbb^N})$} }.
\]
In particular, all $f \in \e(\Omega;\R^N)$ have linear growth at infinity, i.e., there exists a positive constant $M$ such that
$|f(x,A)| \leq M(1 + |A|)$ for all $x \in \Omega$ and all $A \in \R^N$. With the norm 
\[
\| f \|_{ \e(\Omega;\R^N)} := \| Sf \|_\infty, \qquad f \in  \e(\Omega;\R^N),
\]
the space $\e(\Omega;\R^N)$ turns out to be a Banach space. 
Also, by definition, for each $f \in \e(\Omega;\R^N)$ the limit
\[
f^\infty(x,A) := \lim_{\substack {x' \to x \\ A' \to A \\ t \to \infty}} \frac{f(x',tA')}{t}, \qquad (x,A) \in \cl{\Omega} \times \R^N,
\]
exists and defines a positively $1$-homogeneous function called the {\it strong recession function} of $f$. 
 Even if one drops the dependence on $x$,  the recession function $h^\infty$ might not exist for  $h \in \Crm(\R^N)$. % with linear growth at infinity~\cite{Muller92Indiana}.
 Instead, one can always define the {\it upper} and {\it lower recession functions} 
\begin{align*}
f^{\#}(x,A) &:= \limsup_{\substack {x' \to x\\A' \to A \\ t \to \infty}} \,\frac{f(x',tA')}{t},  \\
f_\#(x,A) &:= \liminf_{\substack {x' \to x\\A' \to A \\ t \to \infty}} \, \frac{f(x',tA')}{t},
\end{align*}
which again can be seen to be positively $1$-homogeneous. 
If $f$ is $x$-uniformly Lipschitz continuous in the $A$-variable and there exists a modulus of continuity $\omega \colon [0,\infty) \to [0,\infty)$ (increasing, continuous, and $\omega(0) = 0$) such that
\[
|f(x,A) - f(y,A)|\le \omega(|x - y|)(1 + |A|),  \qquad x,y \in \Omega, \, A \in \R^N,
\] 
then the definitions of $f^\infty$, $f^\#$, and $f_\#$ simplify to
\begin{align*}
f^\infty(x,A) &:= \lim_{\substack {t \to \infty}} \,\frac{f(x,tA)}{t}, \\
f^\#(x,A) &:= \limsup_{\substack {t \to \infty}} \, \frac{f(x,tA)}{t},\\
f_\#(x,A) &:= \liminf_{\substack {t \to \infty}} \, \frac{f(x,tA)}{t}.
\end{align*} 

A natural action of $\E(\Omega;\R^N)$ on the space $\M(\Omega;\R^N)$ is given by
	\[
	\mu \mapsto \int_{\Omega} f\left(x,\frac{\di \mu}{\di \Lcal^N}(x)\right) \, \textnormal{d}x + \int_{\Omega} f^\infty\left(x,\frac{\di \mu^s}{\di |\mu^s|}(x)\right) \, \textnormal{d}|\mu^s|(x).
	\]
In particular, for $f(x,A) = \sqrt{1 + |A|^2} \in \E(\Omega;\R^N)$, for which $f^\infty(A) = |A|$, 
%Next, 
we define the {\it  area functional}
\begin{equation}\label{eq:area}
\area{\mu}(\Omega) := \int_{\Omega} \sqrt{1 + \Big|\frac{\di \mu}{\di \Lcal^N}\Big|^2} \dd x + |\mu^s|(\Omega), \qquad \mu \in \M(\Omega;\R^N).
\end{equation}

In addition to the well-known weak* convergence of measures, we say that a sequence $(\mu_j)$ \emph{converges area-strictly} to $\mu$ in $\M(\Omega;\R^N)$ if 
\[\text{$\mu_j \toweakstar \mu$ in $\M(\Omega;\R^N)$} \qquad \text{and} \qquad 
\area{\mu_j}(\Omega) \to \area{\mu}(\Omega).\]
This notion of convergence turns out to be stronger than the conventional {\it strict convergence} of measures, which means that
 \[
 \mu_j \toweakstar \mu \quad \text{in $\M(\Omega;\R^N)$} \qquad \text{and} \qquad 
  |\mu_j|(\Omega) \to |\mu|(\Omega).
 \]
Indeed, the area-strict convergence, as opposed to the usual strict convergence, prohibits oscillations of the absolutely continuous part. The meaning of area-strict convergence becomes clear when considering the following version of Reshetnyak's continuity theorem, which entails that the topology generated by area-strict convergence is the coarsest topology under which the natural action of $\E(\Omega;\R^N)$ on $\M(\Omega;\R^N)$ is continuous.

\begin{theorem}[Theorem~5 in~\cite{KristensenRindler10Relax}] \label{thm:Reshet}For every integrand $f \in \E(\Omega;\R^N)$, the functional  
	\[
	\mu \mapsto \int_{\Omega} f\left(x,\frac{\di \mu}{\di \Lcal^N}(x)\right) \, \textnormal{d}x + \int_{\Omega} f^\infty\left(x,\frac{\di \mu^s}{\di |\mu^s|}(x)\right) \, \textnormal{d}|\mu^s|(x),
	\]
	is area-strictly continuous on $\M(\Omega;\R^N)$.
\end{theorem}

\begin{remark}\label{rem:area}
Notice that if $\mu\in \Mcal(\R^d;\R^N)$, then $\mu_\varepsilon \to \mu$ area-strictly, where $\mu_{\varepsilon} $ is the mollification of $\mu$ with a family of standard convolution kernels, $\mu_{\varepsilon}\coloneqq \mu\ast\rho_\varepsilon$ and $\rho_\varepsilon(x) \coloneqq \varepsilon^{-d}\rho(x/\varepsilon)$ for $\rho\in \Crm_c^\infty (B_1)$ a positive and even function satisfying {$\int \varrho \dd x =1$}.
\end{remark}

%\begin{remark}It can easily be seen that area-strict convergence is a sharp condition for the continuity of integral functionals defined on measures by taking the area functional by taking $f(x,A) = .
%\end{remark}\label{rem:sharpness}

Generalized Young measures form a set of dual objects to the integrands in~$\e(\Omega;\R^N)$. We recall briefly some aspects of this theory, which was introduced by DiPerna and Majda in~\cite{DiPernaMajda87} and later extended in~\cite{AlibertBouchitte97,KristensenRindler10}.
\begin{definition}[generalized Young measure]
A generalized Young measure, parameterized by an open set $\Omega \subset \R^d$, and with values in $\R^N$, is a triple $\nu = (\nu_x,\lambda_\nu,\nu_x^\infty)$, where
\begin{itemize}
\item[(i)] $(\nu_x)_{x \in \Omega} \subset \Mcal^1(\R^N)$ is a parameterized  family of probability measures on $\R^N$, \vskip5pt
\item[(ii)] $\lambda_\nu \in \M_+(\cl{\Omega})$ is a positive finite Radon  measure on $\cl{\Omega}$, and\vskip5pt
\item[(iii)] $(\nu_x^\infty)_{x \in \cl{\Omega}} \subset \Mcal^1(\Sbb^{N-1})$ is a parametrized family of probability measures on the unit sphere $\Sbb^{N-1}$.
\end{itemize}
Additionally, we require that
\begin{itemize}
\item[(iv)] the map $x \mapsto \nu_x$ is weakly* measurable with respect to $\Lcal^d$, \vskip5pt
\item[(v)] the map $x \mapsto \nu_x^\infty$ is weakly* measurable with respect to $\lambda_\nu$, and \vskip5pt
\item[(vi)] $x \mapsto \dprb{ |\frarg|,\nu_x} \in \Lrm^1(\Omega)$. 
\end{itemize}
The set of all such Young measures is denoted by $\Y(\Omega;\R^N)$. 

Similarly we say that $\nu \in \Y_\loc(\Omega;\R^N)$ if $\nu \in \Y(E;\R^N)$ for all open $E \Subset \Omega$.
\end{definition}

Here, weak* measurability means that the functions $x \mapsto \dpr{ f(x,\frarg),\nu_x}$ (respectively $x \mapsto \dpr{f^\infty (x,\frarg),\nu^\infty_x}$) are Lebesgue-measurable (respectively $\lambda_\nu$-measurable) for all Carath\'{e}odory integrands $f \colon {\Omega} \times \R^N \to \R$ (measurable in their first argument and continuous in their second argument).

For an integrand $f \in \E(\Omega;\R^N)$ and a Young measure $\nu \in \Y(\Omega;\R^N)$, we define the \emph{duality paring} between $f$ and $\nu$ as follows:
\[
\ddprb{f, \nu} := \int_{\Omega} \dprb{ f(x,\frarg),\nu_x} \dd x + \int_{\cl\Omega}  
\dprb{ f^\infty (x,\frarg),\nu^\infty_x} \dd\lambda_\nu(x).
\]

In many cases it will be sufficient to work with functions $f \in \e(\Omega;\R^N)$ that are Lipschitz continuous. The following density lemma can be found in~\cite[Lemma~3]{KristensenRindler10}.
\begin{lemma}\label{lem:separation}
There exists a countable set of functions $\{f_m\} = \{\varphi_m\otimes h_m \in \Crm(\cl{\Omega}) \times 
\Crm(\R^N) : m \in \N\} \subset \E(\Omega;\R^N)$ such that for two Young measures $\nu_1, \nu_2 \in \Y(\Omega;\R^N)$ the implication
\[
\text{$\ddpr{ f_m,\nu_1} = \ddpr{ f_m,\nu_2}\quad \forall \, m \in \N \qquad \Longrightarrow \qquad \nu_1 = \nu_2$}
\]
holds.
Moreover, all the $h_m$ can be chosen to be Lipschitz continuous and all the \(\varphi_m\) can be chosen to be non-negative.
\end{lemma}

Since $\Y(\Omega;\R^N)$ is contained in the dual space of $\E(\Omega;\R^N)$ via the duality pairing $\ddpr{ \frarg, \frarg }$, we say that a sequence of Young measures $(\nu_j) \subset \Y(\Omega;\R^N)$ {\it converges weakly*} to $\nu \in \Y(\Omega;\R^N)$, in symbols $\nu_j \toweakstarY \nu$, if
\[
\ddprb{ f,\nu_j } \to \ddprb{ f, \nu } \qquad \text{for all $f \in \E(\Omega;\R^N)$}.
\]

Fundamental for all Young measure theory is the following compactness result, see~\cite[Section 3.1]{KristensenRindler10} for a proof. 
\begin{lemma}[compactness] \label{lem:YM_compact}
Let $(\nu_j) \subset \Y(\Omega;\R^N)$ be a sequence of Young measures satisfying 
\begin{itemize}
\item[(i)] the functions $x \mapsto \dprb{ |\cdot|,\nu_j}$ are uniformly bounded in $\Lrm^1(\Omega)$,\vskip5pt
\item[(ii)] $\sup_j \lambda_{\nu_j}(\cl{\Omega}) < \infty$.
\end{itemize}
Then, there exists a subsequence (not relabeled) and $\nu \in \Y(\Omega;\R^N)$ such that
$\nu_j \toweakstarY \nu$ in $\Y(\Omega;\R^N)$.
\end{lemma}

The Radon--Nikod\'{y}m decomposition~\eqref{eq:RN} induces a natural embedding of $\Mcal({\Omega};\R^N)$ into $\Y(\Omega;\R^N)$ via the identification $\mu \mapsto \delta[\mu]$, where
\[
(\delta[\mu])_x := \delta_{\frac{\di \mu}{\di \Lcal^d}(x)}, \qquad \lambda_{\delta[\mu]} := |\mu^s|, \qquad (\delta[\mu])^\infty_x := \delta_{\frac{\di \mu^s}{\di |\mu^s|}(x)}.
\]
In this sense, we say that the sequence of measures $(\mu_j)$ {\it generates} the Young measure $\nu$ if $\delta[{\mu_h}] \toweakstarY \nu$ in $\Y(\Omega;\R^N)$; we write 
\[
  \mu_j \toY \nu.
\]

The \emph{barycenter} of a Young measure $\nu \in \Y({\Omega};\R^N)$ is defined as the measure
\[
[\nu] \coloneqq \dprb{ \id,\nu_x} \, \Lcal^d \restrict \Omega + \dprb{ \id, \nu^\infty_x } \, \lambda_\nu \in \Mcal(\cl{\Omega};\R^N). 
\]
Using the notation above it is clear that for $(\mu_j) \subset \Mcal(\Omega;\R^N)$ it holds that $\mu_j \toweakstar [\nu]$, as measures on $\cl{\Omega}$, if $\mu_j \toY \nu$.

\begin{remark}\label{rem:strict}
For a sequence $(\mu_j) \subset \Mcal(\Omega;\R^N)$ that area-strictly converges to some limit $\mu \in \Mcal(\Omega;\R^N)$, it is relatively easy to characterize the (unique) Young measure it generates. Indeed, an immediate consequence of the Separation Lemma~\ref{lem:separation} and Theorem~\ref{thm:Reshet} is that 
\[
\mu_j \to \mu \; \text{area-strictly in $\Omega$} \qquad \Longleftrightarrow \qquad \mu_j \toY
\delta[\mu] \in \Y(\Omega;\R^N).
\]
\end{remark}

Young measures generated by means of periodic homogenization can be easily computed, see Lemma A.1 in~\cite{BallMurat84}. 

\begin{lemma}[oscillation measures]\label{lem:homogenization}
Let $1 \le p < \infty$ and let $w \in \Lrm_\loc^p(\R^d;\R^N)$ be a $Q$-periodic function and let $m \in \N$. Define the $(Q/m)$-periodic functions $w_m(x) \coloneqq w(mx)$.
Then,  
\[
w_m  \toweak  \cl w(x) \coloneqq \int_{Q}  w(y) \dd y %  \qquad \text{\big($\toweakstar$ if $p = \infty$\big)}
\]
in $\Lrm^p_\loc(\R^d;\R^N)$.
%for every measurable $\Omega \subset \R^d$. 

In particular, the sequence $(w_m) \subset \Lrm^1_\loc(\R^d;\R^N)$ generates the homogeneous (local) Young measure $\nu = (\overline{\delta_w},0,\frarg) \in \Y_\loc(\R^d;\R^N)$ (since $\lambda_\nu$ is the zero measure, the component $\nu_x^\infty$ can be occupied by any parameterized family of probability measures in $\M^1(\Sbb^{N-1})$), where
\[
\dpr{h,\overline{\delta_w}} \coloneqq \int_Q h(w(y)) \dd y \quad \text{for all $h \in \Crm(\R^d)$ with linear growth at infinity}.
\] 
\end{lemma}

In some cases it will be necessary to determine the smallest linear space containing the support of a Young measure. With this aim in mind, we state the following version of Theorem 2.5 in~\cite{AlibertBouchitte97}:
%\mnote{A: For example, here there is an abuse of notation, with our definition $u_j$ does not generate $\nu$, but $(u_j\Lcal^d)$ does. This is why I added a comment in the beginning of this section.}
\begin{lemma}\label{thm:alibert version} Let $(u_j)$ be a sequence in $\Lrm^1(\Omega;\R^N)$ generating a Young measure $\nu \in \Y({\Omega};\R^N)$
	and let $V$ be a subspace of $\R^N$ such that $u_j(x) \in V$ for $\Lcal^d$-a.e. $x \in \Omega$. Then,
	\begin{itemize}
		\item[(i)] $\supp \nu_x \subset V$ for $\Lcal^d$-a.e.\ $x \in \Omega$,\vskip5pt
		\item[(ii)] $\supp \nu^\infty_x \subset V \cap \Sbb^{N-1}$ for $\lambda_\nu$-a.e.\ $x \in \Omega$.	
	\end{itemize}
\end{lemma} 

Finally, we have the following approximation lemma, see~\cite[Lemma~2.3]{AlibertBouchitte97} for a proof.

\begin{lemma} \label{lem:E_approx}
Let $f \colon \Omega \times \R^{N} \to \R$ be an upper semicontinuous integrand with linear growth at infinity. Then,  there exists a decreasing sequence $(f_m) \subset \Ebf(\Omega;\R^{N})$ such that %of the form $f_m = \sum_{j = 1}^{l(m)} \varphi_j \otimes h_j$ (with the choices of $\varphi_j, h_j$ depending on $m$) with
\[
  \inf_{m \in \N} f_m = \lim_{m \to \infty} f_m = f, \qquad
    \inf_{m \in \N} f_m^\infty = \lim_{m \to \infty} f_m^\infty = f^\# \qquad\textnormal{(pointwise).}
\]
Furthermore, the linear growth constants of the $f_m$'s can be chosen to be bounded by the linear growth constant of $f$.
\end{lemma}

By approximation, we thus get:

\begin{corollary}\label{cor:representation} Let $f \colon \Omega \times \R^N \to \R$ be an upper semicontinuous Borel integrand. Then the functional
\[
\nu \mapsto \int_\Omega \dpr{f(x,\frarg),\nu_x} \dd x + \int_{\cl{\Omega}} \dpr{f^\#(x,\frarg),\nu_x^\infty} \dd \lambda_\nu(x) 
\]
is sequentially weakly* upper semicontinuous on $\Y(\Omega;\R^N)$.

 Similarly, if $f \colon \Omega \times \R^N \to \R$ is a lower semicontinuous Borel integrand, then the functional
\[
\nu \mapsto \int_\Omega \dpr{f(x,\frarg),\nu_x} \dd x + \int_{\cl{\Omega}} \dpr{f_\#(x,\frarg),\nu_x^\infty} \dd \lambda_\nu(x) 
\]
is sequentially weakly* lower semicontinuous on $\Y(\Omega;\R^N)$.
\end{corollary}

\subsection{Tangent measures}\label{sec: tangent}

In this section we recall the notion of tangent measures, as introduced by Preiss~\cite{Preiss87} (with the exception that we always include the zero measure as a tangent measure).

Let $\mu \in \Mcal(\Omega;\R^N)$ and consider the map $T^{(x_0,r)}(x) := (x - x_0)/r$, 
which blows up $B_r(x_0)$, the open ball around $x_0 \in \Omega$ with radius $r > 0$, into the open unit ball $B_1$. The \emph{push-forward} of $\mu$ under $T^{(x_0,r)}$ is given by the measure
\[
T^{(x_0,r)}_\# \mu(B) := \mu(x_0+rB ), \qquad B \subset r^{-1}(\Omega-x_0) \;\text{ a Borel set.}
\]
We say that $\nu$ is a \emph{tangent measure} to $\mu$ at a point $x_0 \in \R^d$ if there exist sequences $r_m > 0$, $c_m > 0$ with $r_m \todown 0$ such that 
\[
c_m T^{(x_0,r_m)}_\# \mu  \toweakstar \nu\quad \text{in $\Mcal_\loc(\R^d;\R^N)$}.
\]
The set of all such tangent measures is denoted by $\Tan(\mu,x_0)$ and the sequence $c_m T^{(x_0,r_m)}_\# \mu$ is called a {\it blow-up sequence}. Using the canonical zero extension that maps the space $\Mcal(\Omega;\R^N)$ into the space $\Mcal(\R^d;\R^N)$ we may use most of the results contained in the general theory for tangent measures when dealing with tangent measures defined on smaller domains.

Since we will frequently restrict tangent measures to the $d$-dimensional unit cube $Q := (-1/2,1/2)^d$, we set
\[
\Tan_Q(\mu,x_0) := \setb{\sigma \restrict \cl Q}{\sigma \in \Tan(\mu,x_0)}.
\]

One can show (see Remark 14.4 in~\cite{Mattila95book}) that for any non-zero $\sigma \in \Tan(\mu,x_0)$ it is always possible to choose the scaling constants $c_m > 0$ in the blow-up sequence to be
\[
c_m := c \mu(x_0 + r_m \cl{U})^{-1} 
\]
for any open and bounded set $U \subset \R^d$ containing the origin and with the property that $\sigma(U) > 0$, for some positive constant $c = c(U)$ (this may involve passing to a subsequence).

A special property of tangent measures is that at $|\mu|$-almost every $x_0 \in \R^d$ it holds that
\begin{equation}\label{eq:both}
\sigma = \text{w*-}\lim_{m \to \infty}  c_m T^{(x_0,r_m)}_\# \mu \quad \Longleftrightarrow \quad |\sigma| = \text{w*-}\lim_{m \to \infty}  c_m T^{(x_0,r_m)}_\# |\mu|, 
\end{equation}
where the weak* limits are to be understood in the spaces $\Mcal_\loc(\R^d;\R^N)$ and $\M_\loc^+(\R^d)$, respectively. A proof of this fact can be found in Theorem~2.44 of~\cite{AmbrosioFuscoPallara00book}.
In particular, this implies 
\[
\Tan(\mu,x_0) = \frac{\di \mu}{\di |\mu|}(x_0) \cdot \Tan(|\mu|,x_0)\quad \text{for $|\mu|$-almost every $x_0 \in \R^d$}.
\]

If $\mu, \lambda \in \M^+_\loc(\R^d)$ are two Radon measures with the property that $\mu \ll \lambda$, i.e., that $\mu$ is absolutely continuous with respect to $\lambda$, then (see Lemma 14.6 of~\cite{Mattila95book})
\begin{equation}\label{eq: tangent absolute}
\Tan(\mu,x_0) = \Tan(\lambda,x_0) \quad \text{for $\mu$-almost every $x_0 \in \R^d$},
\end{equation}
and in particular if $f \in \Lrm^1_\loc(\R^d,\lambda;\R^N)$, i.e., $f$ is is $\lambda$-integrable, 
\[
\Tan(f \lambda,x_0) = f(x_0)\cdot \Tan(\lambda,x_0) \quad \text{for $\lambda$-a.e. $x_0 \in \R^d$}.
\]
On the other hand, at every $x_0 \in \supp \mu$ such that 
\[
\lim_{r \todown 0} \frac{\mu(B_r(x_0) \setminus E)}{\mu(B_r(x_0))} = 0
\]
for some Borel set $E \subset \R^d$, it holds that 
\[
\Tan(\mu,x_0) = \Tan(\mu\restrict E,x_0).
\]
%\mnote{A: changed a bit the argument because we require the measures to have same sign to be able to use~\eqref{eq: tangent absolute}}
A simple consequence of~\eqref{eq: tangent absolute} is
\[\Tan(|\mu|,x_0) = \Tan\big(\Lcal^d, x_0\big) \qquad \text{for $\frac{\di |\mu|}{\di \Lcal^d} \, \Lcal^d$-a.e.\ $x_0 \in \R^d$.}
\]
%while from~\eqref{eq: tangent density} it follows that
%\[
%  \Tan\left(\frac{\di \mu}{\di \Lcal^d}\, \Lcal^d,x_0\right) = \frac{\di \mu}{\di \Lcal^d}(x_0) \Tan(\Lcal^d,x_0).
%\]
%The arguments above imply that
This implies %\mnote{A: I changed the domain of $\alpha$ to be the non-negative real numbers.}
\begin{equation}\label{eq:ae tangent}
\Tan(\mu,x_0) = \setBB{\alpha \, \frac{\di \mu}{\di \Lcal^d}(x_0) \, \Lcal^d}{\alpha \in [0,\infty) } \quad \text{for $\Lcal^d$-a.e.\ $x_0 \in \R^d$}.
\end{equation}
We shall refer to such points as {\it regular points} of $\mu$. Furthermore, for every regular point $x_0$ there exists a sequence $r_m \todown 0$ and a positive constant $c$ such that             
\[
c r^{-d}_m (T^{(x_0,r_m)}_\# \mu) \toweakstar \frac{\di \mu}{\di \Lcal^d}(x_0) \, \Lcal^d \quad \text{in $\M_\loc(\R^d;\R^N)$}.
\]

\subsection{Rigidity results}\label{sec:rigid}
As discussed in the introduction, for a linear operator $\A := \sum_{|\alpha| \leq k} A_\alpha \partial^\alpha$,  the wave cone 
%To a $k$th-order linear PDE operator $\A := \sum_{|\alpha| \leq k} A_\alpha \partial^\alpha$ as in the introduction, we associate its  {\it wave cone}
% of \enquote{admissible amplitudes}
\[
\Lambda_\A := \bigcup_{|\xi|= 1} \ker \Abb^k(\xi)  \; \subset \R^N
\]
contains those amplitudes along which is possible to have ``one-directional'' oscillations or concentrations, or equivalently, it contains the amplitudes along which the system loses its ellipticity.
%\mnote{G: I modify a bit since the wave cone has already been introduced}

%
%
%
%which plays a crucial role in the compensated compactness theory for sequences of $\A$-free maps, see~\cite{Tartar79,Tartar83,Murat78,Murat79,Murat81,DiPerna85,FonsecaMuller99}. Intuitively, the wave cone $\Lambda_\A$ consists of those amplitudes $P \in \R^N$ for which there is a non-zero frequency $\xi \in \Sbb^{d-1}$ such that
%\[
%\A \bigl(P \cdot \ee^{2\pi\ii x \cdot \xi}\bigr) = 0.
%\]

%In this sense, one could expect that $\Lambda_\A$ (or perhaps $\text{span} \, \Lambda_\A$) contains the values that may be attained by means of oscillation or concentration effects on $\A$-free sequences. As a matter of fact, in a recent development de Philippis and Rindler have been able to show the following characterization on $\A$-free Radon measures:

The main result of~\cite{DePhilippisRindler16} asserts that the polar vector of the singular part of an \(\A\)-free measure \(\mu\)  necessarily has to lie in \(\Lambda_\A\):
 
\begin{theorem}\label{thm:guido filip}
	Let $\Omega \subset \R^d$ be an open set and let $\mu \in \Mcal(\Omega;\R^N)$ be an $\A$-free Radon measure on $\Omega$ with values in $\R^N$, i.e.,
	\[
	\A \mu = 0 \quad \text{in the sense of distributions}.
	\] Then,
	\[
	\frac{\di \mu}{\di |\mu|}(x) \in \Lambda_\A \qquad \text{for $|\mu^s|$-a.e.\ $x \in \Omega$}.
	\]
\end{theorem}

\begin{remark} \label{rem:MuratCR}
The proof of this result does not require $\A$ to satisfy Murat's constant rank condition~\eqref{eq:constrank}. However, for the present work, this requirement cannot be dispensed with in the following  decomposition by Fonseca and M\"uller~\cite[Lemma 2.14]{FonsecaMuller99}, where it is needed for the Fourier projection arguments.
\end{remark}

\begin{lemma}[projection]\label{lem:fonseca constant} %\mnote{A: please be careful in the use of $\Wrm^{-k,q}_\per(Q)$ and $\Wrm^{-k,q}(Q)$}
	Let $\A$ be a homogeneous differential operator satisfying the constant rank property~\eqref{eq:constrank}. Then, for every $1 < p < \infty $, there exists a linear projection operator
	\[
	  \Pcal : \Lrm^p(Q;\R^N) \to \Lrm^p(Q;\R^N)
    \]
    and a positive constant  $c_p > 0$ such that
	\[
	\A(\Pcal u) = 0, \qquad \int_Q \Pcal u \dd y = 0, \qquad
	\| u - \Pcal u \|_{\Lrm^p(Q)} \leq c_p\| \A u\|_{\Wrm^{-k,p}_\per(Q)},
	\]
	for every $u \in \Lrm^p(Q;\R^N)$ with $\int_Q u \dd y = 0$.
\end{lemma}

\begin{remark} \label{rem:W-kq_F}
Here, $\Wrm^{k,p}_\per(Q)$ ($1 < p < \infty$) denotes the space of $\Wrm^{k,p}(Q)$-maps, which can be $Q$-periodically extended to a $\Wrm^{k,p}_\loc(\R^d)$-map; the space $\Wrm^{-k,q}_\per(Q)$ with $1/p + 1/q = 1$ is its dual. Note that the dual norm is equivalent to
\[
  \normBB{\Fcal^{-1}\biggl[\frac{\hat{u}(\xi)}{(1+\abs{\xi}^2)^{k/2}}\biggr]}_{\Lrm^q(Q)},
\]
where, for $\xi \in \Z^d$, $\hat{u}(\xi)$ denotes the Fourier coefficients on the torus and $\Fcal^{-1}$ is the inverse Fourier transform. In the case $\int_Q u \dd x = 0$ (hence $\hat{u}(0) = 0$) this norm is also equivalent to the norm
\[
  \normBB{\Fcal^{-1}\biggl[\frac{\hat{u}(\xi)}{\abs{\xi}^k}\biggr]}_{\Lrm^q(Q)}
\]
since the Fourier multipliers $(1+\abs{\xi}^2)^{-k/2}$ and $\abs{\xi}^{-k}$ are comparable (by the Mihlin multiplier theorem) for all $\xi$ with $\abs{\xi} \geq 1$.
\end{remark}

\begin{proof}
The proof given in \cite{FonsecaMuller99} technically applies only to first-order differential operators. However, the result can be extended to operators of any degree, as long as they are homogeneous. %\mnote{A: added this because of the case $k =0$; please check if such remark is needed ahead in the paper}{\color{red}If $\A$ is an zero-order operator the sought properties hold trivially for the projection $\Pcal u = 0$. Hereafter, we assume that $\A$ is an homogeneous operator of degree $k \ge 1$.}  %of the operator as long as it is homogeneous. 
We shortly recall how this is done. By definition,  
\begin{equation}\label{projection1}
\rank \Abb^k(\xi) = \rank \Abb(\xi) = r \quad \text{for all $\xi \in \Sbb^{d-1}$}. 
\end{equation}
For each $\xi \in \R^d$ we %have $\R^N = \ker \Abb(\xi) \oplus \ran \Abb(\xi)$, we 
write $\Pbb(\xi) : \R^N \to \R^N$ to denote the \emph{orthogonal projection} onto $\ker \Abb(\xi)$, and by $\Qbb(\xi)$ we denote the left inverse of $\Abb(\xi)$. 

It follows from the positive homogeneity of $\Abb$ that $\Pbb : \R^d\setminus\{0\} \to \R^N \otimes \R^N$ is $0$-homogeneous. Moreover, $\id_{\R^N} - \Pbb(\xi) = \Qbb(\lambda \xi)\Abb(\lambda \xi) = \lambda^k \Qbb(\lambda \xi)\Abb(\xi)$ and hence $\Qbb: \R^d\setminus\{0\} \to \R^N \otimes \R^N$ is homogeneous of degree $-k$. In light of~\eqref{projection1}, both maps are smooth (see Proposition 2.7 in~\cite{FonsecaMuller99}).

Since the map $\xi \mapsto \Pbb(\xi)$ is homogeneous of degree 0 and is infinitely differentiable in $\Sbb^{d-1}$, by Proposition 2.13 in~\cite{FonsecaMuller99}, the map defined on $\Crm^\infty_\per(Q;\R^N)$ by
\[
\Pcal u(w) \coloneqq  \sum_{\xi \in \Zbb^d \setminus \{0\}} \Pbb(\xi) \hat u(\xi) e^{2\pi \ii \xi \cdot w}, 
\]
where $\{\hat u(\xi)\}_{\Zbb^d}$ are the Fourier coefficients of $u \in \Lrm^p(Q;\R^N)$,
extends to a $(p,p)$-\emph{Fourier multiplier} $\Pcal$ on $\Lrm^p(Q;\R^N)$ for all $1 < p < \infty$. %and every $i,j \in \{1,\dots,N\}$.
%Next, we define a tensor bounded linear functional acting on vector functions by setting 
%\[
%\Pcal u:\Lrm^p(\Tbb^d;\R^N) \to \Lrm^p(\Tbb^d;\R^N) :u  \mapsto \sum_{j = 1}^N T_{ij}u_j.
%\] 

Since $\Pbb(\xi)$ is a projection, so it is $\Pcal$:
\begin{align*}
(\Pcal \circ \Pcal) u & = \sum_{\xi \in \Zbb^d \setminus \{0\}} (\Pbb(\xi) \circ \Pbb(\xi)) \hat u(\xi) e^{2\pi \ii \xi \cdot w} \\
& = \sum_{\xi \in \Zbb^d \setminus \{0\}} \Pbb(\xi) \hat u(\xi) e^{2\pi \ii \xi \cdot w} = \Pcal u.
\end{align*}
Moreover, 
\[
\widehat{(\A (\Pcal u))}(\xi) = \Abb(\xi) \widehat{(\Pcal u)}(\xi) = \Abb(\xi)[\Pbb(\xi)\hat u(\xi)] = 0
\]
for all $\xi \in \Zbb^d \setminus \{0\}$.  Since $\widehat{(\Pcal u)}(0) = 0$, we get 
\[
\int_Q \Pcal u \dd y = 0, \quad \text{and} \quad \A(\Pcal u) = 0.
\]

Finally, let $u \in \Crm^\infty_\per(Q;\R^N)$. We use that $\Abb$ and $\Qbb$ are $k$-homogeneous and $(-k)$-homogeneous, respectively, to show that
\begin{align*}
\hat u(\xi) - \widehat{\Pcal u}(\xi) & = (\id_{\R^N} - \Pbb(\xi))\hat u(\xi) \\
& = \Qbb(\xi)\Abb(\xi)\hat u(\xi) \\
&= \Qbb\bigg(\frac{\xi}{|\xi|}\bigg)\frac{1}{|\xi|^k}\Abb(\xi)\hat u(\xi),
\end{align*}
for all $\xi \in \Zbb^{d} \setminus \{0\}$. Therefore, the Mihlin multiplier theorem and Remark~\ref{rem:W-kq_F} imply
%\mnote{F\&G: it needs Mihlin -- you cannot just pointwise estimate the Fourier transform (you need some smoothness in multiplier)} 
that
\[
\| u - \Pcal u\|_{\Lrm^{p}(Q)} \le c_p\|\A u\|_{\Wrm^{-k,p}_\per(Q)}
\]
for all $u \in \Crm^\infty_\per(Q;\R^N)$ with $\int_Q u \dd y = 0$. The general case follows by approximation.
\end{proof}

Lemma~\ref{lem:fonseca constant} implies that every $Q$-periodic $u \in \Lrm^p_\loc(\R^d;\R^N)$ with $1 < p < \infty$ and mean value zero can be decomposed as the sum
  \[
  u = v + w,  \qquad v = \Pcal u,
  \]
  where
\[
\A v = 0  \qquad\text{and}\qquad
\| w \|_{\Lrm^p(Q)} \leq c_p \|\A u \|_{\Wrm^{-k,p}_\per(Q)}. 
\]

A crucial issue in lower semicontinuity problems is the understanding of oscillation and concentration effects in weakly (weakly*) convergent sequences. In our setting, we are interested in sequences of asymptotically $\A$-free measures generating what we naturally term {\it $\A$-free Young measures}. The study of general $\A$-free Young measures can be reduced to understanding oscillations in the class of \emph{periodic} $\A$-free fields. This is expressed in the next lemma, which is a variant of Proposition~3.1 in~\cite{FonsecaLeoniMuller04} for higher-order operators (see also Lemma 2.20 in~\cite{BaiaChermisiMatiasSantos13}).

\begin{lemma}\label{lem:boundary values}
	Let $\A$ be an homogeneous linear partial differential operator satisfying the constant rank property~\eqref{eq:constrank}. Let $(u_j),(v_j) \subset \Lrm^1(Q;\R^N)$ be sequences such that 
	\[
		u_j - v_j \toweakstar 0  \quad\text{in $\Mcal(Q;\R^N)$}
		\qquad\text{and}\qquad
		|u_j| + |v_j| \toweakstar \Lambda  \quad\text{in $\M^+(\cl{Q})$}
	\]
	with $\Lambda(\partial Q) = 0$ and
	\[
	  \A(u_j - v_j)\to 0  \quad \text{in $\Wrm^{-k,q}(Q;\R^n)$} \qquad
	  \text{for some $1 < q < d/(d-1)$.}
    \]
  Assume that the sequence $(u_j)$ generates the Young measure $\nu \in \Y(Q;\R^N)$. Then, there exists another sequence $(z_j)\subset \Crm^\infty_\per(Q;\R^N)$ such that
	\[
	\A z_j = 0, \qquad \int_Q z_j = 0, \qquad z_j \toweakstar 0 \quad \text{in $\Mcal(Q;\R^N)$},
	\]
	and (up to taking a subsequence of the $v_j$'s) the sequence $(v_j + z_j)$ also generates the Young measure $\nu$, i.e., 
	\[
	(v_j + z_j)  \toY \nu \quad \text{in $\Y(Q;\R^N)$}.
	\]
	Moreover, for every   $f : \R^N \to \R$  Lipschitz  it holds that
	\begin{equation} \label{eq:pizza}
	\liminf_{j \to \infty} \int_Q f(u_j) \dd x \ge \liminf_{j \to \infty} \int_Q f(v_j + z_j) \dd x.
	\end{equation}
\end{lemma}

%Note that the sequence $(z_j)$ may depend on the choice of $f$ (since $\ONE \otimes f$ is not necessarily in $\Ebf(\Omega;\R^N)$).

\begin{proof} Consider a family of cut-off functions $\psi_m \in \Crm^\infty_c(Q;[0,1])$ with
	$\psi_m \equiv 1$ in the set $\set{y \in Q }{\dist(y,\partial Q) > 1/m}$ and define
	\[
	w^m_j :=  (u_j - v_j)\psi_m \in \Lrm^1(Q;\R^N).
	\]
	Since $\psi_m \in \Crm^\infty_c(Q)$, it also holds that
	\[
	w^m_{j } \toweakstar 0 \quad \text{in $\M(Q;\R^N)$ \quad as $j \to \infty$, \quad for every $m \in \N$}.
	\]
 Furthermore, 
	\begin{equation}\label{eq:product}
	\A w^m_j =  \A(u_j - v_j) \psi_m + \sum_{\substack{|\alpha|= k,\\1 \le |\beta| \le  k}} c_{\alpha \beta}A_\alpha \partial^{\alpha - \beta} (u_j - v_j) \partial^\beta \psi_m
	\end{equation}
	where $c_{\alpha \beta} \in \Nbb$.
The convergence $u_j - v_j \toweakstar 0$ and the compact embedding $\Mcal(Q;\R^N) \cembed \Wrm^{-1,q}(Q;\R^N)$ entail, via~\eqref{eq:product}, the strong convergence
	\begin{equation}\label{eq:strongkq}
	\A w^m_j \to 0 \quad \text{in $\Wrm^{-k,q}(Q;\R^n)$} \qquad \text{as $j \to \infty$}.
	\end{equation}
	
	Let, for $\eps > 0$, $\rho_\eps(x) \coloneqq \rho(x/\eps)$ where $\rho \in \Crm^\infty_c(B_1)$ is an even mollifier. %As a consequence of the $\Lrm^1$-convergence $(w_j^m \ast \rho_\eps) \to w_j^m$, we might find $\eps= \eps(j,m) > 0$ sufficiently small so that, for 
	For every $m \in \Nbb$, let $(\eps(j,m))_j$ be a sequence with $\eps(j,m) \todown 0$ as $j \to \infty$ such that for $\hat w_j^m \coloneqq w_j^m \ast \rho_{\eps(j,m)}$ it holds that
\[
  \|w_j^m - \hat w_j^m \|_{\Lrm^1(Q)} \leq \frac{1}{j}.
\]
%	\begin{equation}\label{eq:hat1}
%	\|w_j^m - \hat w_j^m\|_{\Lrm^1(Q)} = \BigO(j),
%	\end{equation}
%	$\hat w_j^m \coloneqq w_j^m \ast \rho_\eps$ (here $\rho_\eps = \rho(x/\eps)$ and $\rho$ is a standard mollifier) 	\|w_j^m - \hat w_j^m\|_{\Lrm^1(Q)} = \BigO(j), \quad \text{and} \quad \| \A \hat w_j^m\|_{\Wrm^{-k,q}(Q)} = \BigO(j),
	%\]
	%where $\BigO(j) \to 0$ as $j \to \infty$.  
	
%	 Fix $\phi \in \Crm(Q;\R^N)$. 
%	 A consequence of the convergence $\hat w_j^m \to w_j^m$ in $\Lrm^1(Q)$ is that
%	 \begin{align*}
%	|\dprb{\hat w_j^m,\phi}|  & = |\dprb{w_j^m,\phi}| + \BigO(j) \\
%	& = \dprb{u_j - v_j,\phi\psi_m} + \BigO(j) = c(m,\|\phi\|_\infty)\BigO(j),
%	 \end{align*}
%	 where we have used that $\psi_m \phi \in \Crm_c(Q;\R^N)$ to pass to the last equality. Since $\phi \in \Crm(Q;\R^N)$ was arbitrarily chosen, this proves that 
%	 \[
%	 \hat w_j^m \toweakstar 0 \quad \text{in $\M(Q;\R^N)$}.
%	 \]
	Fix $\phi \in \Wrm^{k,q}(Q;\R^n) \cap \Crm_c(Q;\R^n)$ and fix $m \in \Nbb$. Then, for $j \in \Nbb$ sufficiently large, it holds that
	\begin{align*}
	\absb{\dprb{\A \hat w_j^m, \phi}} & = \absb{\dprb{\A w_j^m, \phi \ast \rho_{\eps(j,m)}}} \notag \\
	%& \le \|\A  w_j^m\|_{\Wrm^{-k,q}(2Q)} \|\phi\|_{\Wrm^{k,q}_0(2Q)} \notag \\
	& \le \|\A w_j^m\|_{\Wrm^{-k,q}(Q)} \|\phi \ast \rho_{\eps(j,m)}\|_{\Wrm^{k,q}(Q)}\\\notag
	& \le \|\A w_j^m\|_{\Wrm^{-k,q}(Q)} \|\phi\|_{\Wrm^{k,q}(Q)}. \notag
	\end{align*}
The case when $\phi$ belongs to $\Wrm^{k,q}_0(Q;\R^n)$ follows by approximation.
	Hence, from~\eqref{eq:strongkq} %and~\eqref{eq:hat1} 
	we obtain that 
	\begin{equation}\label{eq:hat2}
	\|\A \hat w_j^m\|_{\Wrm^{-k,q}(Q)} \to 0 \quad \text{as $j \to \infty$, \quad for every $m \in \Nbb$}.
	\end{equation}

%After  mollification and taking into account Remarks~\ref{rem:area} and~\ref{rem:strict}, we may assume without loss of generality that $w^m_j \in \Lrm^q \cap \Lrm^2$ for every $j,m \in \N$. 

The second step consists of applying the projection of Lemma \ref{lem:fonseca constant} to the mollified functions $\hat w_j^m$. Define $\tilde w_j^m := \hat w^m_j - \int_Q \hat w_j^m \dd x$ (by a slight abuse of notation, we also denote by $\tilde w_j^m$ its $Q$-periodic extension to $\R^d$) and $z^m_j := \Pcal \tilde  w_j^m$. Note that since \(\tilde w_j^m\in \Crm^\infty (Q)\) the same holds for \(z^m_j\) since the projection operators commutes with the Fourier  multiplier \(|\xi|^s\) for all \(s\in \R\). 
It follows from  Lemma~\ref{lem:fonseca constant}  that 
	\begin{align}
	\lim_{j \to \infty} \| \hat w^m_j - z^m_j \|_{\Lrm^1(Q)} & \leq  \lim_{j \to \infty}\|\tilde w^m_j - z^m_j\|_{\Lrm^q(Q)} + \lim_{j \to \infty}\left| \int_Q \hat w^m_j \dd y \right|  \notag \\
	&\leq  c_q \cdot \lim_{j \to \infty} \|\A \hat w^m_j \|_{\Wrm^{-k,q}_\per(Q)}  + \lim_{j \to \infty}\left| \int_Q w^m_j \dd y \right|  \notag \\
	&= 0,  \label{eq: mq}
	\end{align}
	where in the first inequality we have exploited  Jensen's inequality, and for the last inequality we have used the equality of the norms  
	\begin{align*}
%	\| u \|_{\Lrm_\per^p(Q)} & = \| u \|_{\Lrm^p(Q)},\\
	\| u \|_{\Wrm_\per^{-k,p}(Q)} & = \| u \|_{\Wrm^{-k,p}(Q)},
	\end{align*} 
	which holds for functions $u \in \Crm_\per^\infty(Q)$ with $u = 0$ on $\partial Q$ and all $1 < p < \infty$,  together with~\eqref{eq:hat2}.
	
%Fix  $\varphi \otimes g \in \Crm(\overline Q) \times \Wrm^{1,\infty}(\R^N)$ with $\phi \otimes g \in \E(Q;\R^N)$. Using the Lipschitz continuity of $g$, we have that
Let now \(g:\R^N\to \R\) be Lipschitz and let \(\varphi \in \Crm(\overline Q)\) with \(\varphi\ge 0\). Then,
\begin{align}
\int_Q \phi \,g(u_j) \dd y
&= \int_Q \phi\, g(u_j - v_j + v_j) \dd y  \notag\\
&\ge \int_Q \phi\, g(\hat w^m_j + v_j) \dd y - \|\phi\|_\infty\cdot \text{Lip}(g) \cdot \int_Q |1 - \psi_m|(|u_j| + |v_j|) \dd y \notag\\
& \qquad - \|\phi\|_\infty\cdot \text{Lip}(g)\cdot \|w_j^m - \hat w_j^m \|_{\Lrm^1(Q)} \notag\\
&\ge \int_Q \phi\, g(z^m_j + v_j) \dd y - \|\phi\|_\infty\cdot  \text{Lip}(g) \cdot \bigg(\int_Q |1 - \psi_m|(|u_j| + |v_j|) \dd y \notag\\
& \qquad + \|w_j^m - \hat w_j^m \|_{\Lrm^1(Q)} +  \|\hat w^m_j - z^m_j \|_{\Lrm^1(Q)}\bigg).  \label{mare}
\end{align}
Similarly, 
\begin{align}
\int_Q \phi \,g(u_j) \dd y\le &\int_Q \phi\, g(z^m_j + v_j) \dd y +\|\phi\|_\infty\cdot  \text{Lip}(g) \cdot \bigg(\int_Q |1 - \psi_m|(|u_j| + |v_j|) \dd y \notag\\
& \qquad + \|w_j^m - \hat w_j^m \|_{\Lrm^1(Q)} +  \|\hat w^m_j - z^m_j \|_{\Lrm^1(Q)}\bigg)   \label{mare2}
\end{align}
Let \(\{\phi_h\otimes g_h\}_{h=1}^\infty\) be the family of integrands appearing in Lemma~\ref{lem:separation}  and let \(\nu\) be the Young measure generated by \((u_j)\). We have that 
\[
\lim_{j\to\infty} \int_Q \phi_h \,g_h(u_j)=\ddprb{\phi_m \otimes g_m, \nu}\qquad\textrm{for all \(h=1,2,\ldots\)}
\]
and thus using \eqref{mare} and \eqref{mare2} above we infer that 
\begin{equation*}
\begin{split}
\limsup_{m \to \infty}\,
\limsup_{j  \to \infty}  \int_Q \phi_h \, g_h(z^m_j + v_j) \dd y 
  &\leq \lim_{j  \to \infty} \int_Q \phi_h\,g_h(u_j) \dd y \\
  & \leq \liminf_{m \to \infty}\,
  \liminf_{j  \to \infty}\int_Q \phi_h\,g_h(z^m_j + v_j) \dd y.
  \end{split}
\end{equation*}
for all \(h\in \N\) where we have also exploited that $\Lambda(\partial Q) = 0$.
By a  diagonalization argument on $z_j^m$ we may find a sequence 
 $(z_j) \subset \Crm^\infty_\per(Q;\R^N) \cap \ker \A$ such that 
 \[ 
\int_Q z_j \dd y = 0 \quad \text{for all $j \in \N$}, \qquad 
z_j \toweakstar 0 \quad \text{in $\M(Q;\R^N)$},\]
and, for all \(h\in \mathbb N\),
\begin{equation}\label{eq:sandwich}
%\lim_{j \to \infty} \int_Q  \phi_h \,g_h(z_j + v_j) \dd y 
%\le 
 \lim_{j \to \infty} \int_Q \phi_h \,g_h(u_j) \dd y = 
 \lim_{j \to \infty} \int_Q  \phi_h \,g_h(z_j + v_j) \dd y.
\end{equation}
Since $(z_j + v_j)$ is uniformly bounded in $\Lrm^1(Q;\R^N)$, by Lemma~\ref{lem:YM_compact} we may find a subsequence $(z_{j(i)} + v_{j(i)}) \toY \tilde \nu \in \Y(Q;\R^N)$. In particular,
 \begin{gather*}
\ddprb{\phi_h \otimes g_h, \tilde \nu} = \ddprb{\phi_h \otimes g_h, \nu}, 
\end{gather*}
for all \(h\) and thus \(\tilde \nu=\nu\) by Lemma~\ref{lem:separation}.  Inequality~\eqref{eq:pizza} now follows by taking the limit inferior in~\eqref{mare} with \(g=f\) and \(\varphi\equiv 1\).
\end{proof}

\subsection{Scaling properties of $\A$-free measures} 
%In the following we analyze the rescaling map $T^{(x_0,r)}_\# : \M(\Omega;\R^N) \to \M((x_0 - \Omega)/r;\R^N)$ applied to $\A$-free measures.
If $\A$ is a homogeneous operator, then 
\[
\A[T^{(x_0,r)}_\# \mu] = 0 \quad \text{on $(x_0 - \Omega)/r$},
\]
for all $\A$-free measures $\mu \in \M(\Omega;\R^N)$. In general, the re-scaled measure $T^{(x_0,r)}_\# \mu$ is a $(T^r_\ast\A)$-free measure in $(x_0 - \Omega)/r$, where $T^r_\ast\A$ is the operator defined by 
\[
T^r_\ast\A \coloneqq \sum_{h = 0}^k r^{k - h} \A^h,
\]
with $k$ the degree of the operator $\A$ and 
\[
\A^h \coloneqq \sum_{|\alpha|=h} A_\alpha \partial^\alpha, \quad \text{for $h = 0,\dots,k$}.
\]
Notice that, with this convention, $(T^r_\ast \A)^k = \A^k$.

In the sequel it will be often convenient to work with weak* convergent sequences whose elements are $(T^r_\ast\A)$-free measures. The following two results will be useful. %in our intent  to provide a more transparent framework in our proofs.    
%  
%\begin{proposition}[high-order oscillations]
%Let $r \todown 0$ be a sequence of positive numbers and let $(\mu_j)$ be a sequence of $\A$-free measures in $\M(\Omega;\R^N)$ with $\mu_j \toweakstar \mu$ in $\M(\Omega;\R^N)$. Then, there exists a subsequence $(\mu_{j(r)})$ parametrized by $r$ such that
%\[
%\A^k \mu_j \to 0 \quad \text{in $\Wrm^{-k,q}(\Omega)$}.
%\]
%\end{proposition}

%\proofstep{Step~2. $\B(T^{(x_0,r)}_\# \mu_j) = 0$ where $\B = \sum_\alpha r^{|\alpha| - k}A_\alpha \partial^\alpha$.} A change of variables and the $\A$-freeness of $\mu_j$ yield
%\[
%\B \mu_j = \sum_{h = 1}^k r^h \A^h \mu_j = T^{(x_0,r)}_\# [\A\mu_j] = 0.
%\]
%
%\proofstep{Step~3. Conclusion.} 
%Therefore, 
%	\begin{equation}\label{eq:change}
%	\A^k (T^{(x_0,r_m)}_\# \mu) = - \sum_{l = 0}^{k-1} \A^l (r_m^{k-l}T^{(x_0,r_m)}_\# \mu),
%	\end{equation}
%
%\end{proof}
%\end{proposition}

\begin{proposition} \label{lem:nonh}
Let $r_m \todown 0$ be a sequence of positive numbers and let $(\mu_m)$ be a sequence of $\A$-free measures in $\M(\Omega;\R^N)$ with the following property: there are positive constants $c_m$ such that%, for every open $\omega \Subset \R^d$, there exists a positive constant $C_\omega$ (independent of $r$) verifying 
\begin{equation}\label{eq:nonh}
%\lim_{j \to \infty} c_{m}T^{(x_0,r_m)}_\# |\mu_j|(\omega \cap T^{(x_0,r_m)}(\Omega)) \le C_\omega \quad \text{for all $m \in \Nbb$}.
\gamma_m \coloneqq c_{m}T^{(x_0,r_m)}_\# \mu_m \toweakstar \gamma \quad \text{in $\M_\loc(\R^d;\R^N)$}.
\end{equation}
Then,
\[
\A^k (c_mT^{(x_0,r_m)}_\#\mu_{m}) \to 0 \quad \text{in $\Wrm^{-k,q}_\loc(\R^d;\R^n)$} \quad \text{for all $1 < q < d/(d-1)$}.
\]
\end{proposition}

\begin{proof} Fix $r > 0$. 
Then,
%$(T^r_\ast\A)$-freeness of each $T^{(x_0,r)}_\# \mu_m$ yields
	\begin{equation}\label{eq:changeb}
	\A^k(T^{(x_0,r)}_\# \mu_m) =  - \sum_{h = 0}^{k-1} \A^h(r^{k -h}T^{(x_0,r)}_\# \mu_m).
	\end{equation}
 Since %lets us extract a subsequence $(\mu_m)$ (where $\mu_m \coloneqq \mu_{j}$) such that
	\begin{equation*}
	r_m^{k-h}c_mT^{(x_0,r_m)}_\#\mu_m \toweakstar 0 \quad \text{in $\M_\loc(\R^d;\R^N)$,  \quad \text{for every $h = 0,\dots,k-1$}},
	\end{equation*}
	the compact embedding $\Mcal_\loc(\R^d;\R^N) \cembed \Wrm^{-1,q}_\loc(\R^d;\R^N)$ entails the strong convergence 
	\[
	r_m^{k-h}c_mT^{(x_0,r_m)}_\# \mu_m\to 0 \quad \text{in $\Wrm^{-1,q}_\loc(\R^d;\R^N)$} \qquad \text{for every $h = 0, \ldots,  k-1$}.
	\]
Hence, 
	\begin{equation}\label{eq:blowb}
	\A^h (r_m^{k-h}c_mT^{(x_0,r)}_\#\mu_m) \to 0 \quad \text{in $\Wrm^{-k,q}_\loc(\R^d;\R^n)$}
	\end{equation}
	for every $h = 0, \ldots,  k-1$.
The assertion then follows from~\eqref{eq:changeb} and~\eqref{eq:blowb}.
\end{proof}

\subsection{Fourier coefficients of $\A^k$-free sequences} We shall denote the subspace generated by the wave cone $\Lambda_{\A}$ by 
\[
V_{\A} := \mathrm{span}\, \Lambda_{\A} \subset \R^N.
\]

%we denote by $V_\A$ the linear subspace of $\R^d$ spanned by the wave cone $\Lambda_\A$. 

Using Fourier series it is relatively easy to understand the rigidity of $\A^k$-free periodic fields. 
To fix ideas, let $u$ be a $Q$-periodic field in $\Lrm^2_\loc(\R^d;\R^N) \cap \ker \A^k$  with mean value zero (or equivalently $\hat u(0) = 0$). Applying the Fourier transform to $\A^k u = 0$, we find that 
\begin{equation*}\label{eq: torus equation}
0 = \Fcal(\A^k u)(\xi) = \Abb^k(\xi) \hat u(\xi) \qquad \text{for all $\xi \in \mathbb Z^d$}.
\end{equation*}
Hence, $\hat u(\xi) \in \ker_{\CC} \mathbb A^k(\xi)$ for every $\xi \in \Z^d$ (here, $\mathbb A^k(\xi)$ is understood as a complex-valued tensor). In particular, 
\[
\setb{\hat u(\xi) }{\xi \in \mathbb Z^d}\subset \mathbb C \Lambda_\A.
\]
Since $u$ is a real vector-valued function, it immediately follows that
\begin{equation}\label{eq: subspace}
u \in \Lrm^2_\loc(\R^d;V_\A).
\end{equation}
Using a density argument one can show that, up to a constant term, also $Q$-periodic functions in $\Lrm^1_{\loc}(Q;\R^N) \cap \ker \A^k$ take values only in $V_{\A}$. The relevance of this observation will be used later in conjunction with Lemma~\ref{lem:boundary values} in Lemma~\ref{lem:subspace}.

\subsection{$\A$-quasiconvexity}
We state some well-known and some more recent results regarding the properties of $\A$-quasiconvex integrands. This  notion was first  introduced by Morrey~\cite{Morrey66book}  in the case of curl-free vector fields, where it is known as {\it quasiconvexity}, and  later extended by Dacorogna~\cite{Dacorogna82b} and Fonseca--M\"uller~\cite{FonsecaMuller99} to general linear PDE-constraints. 

A Borel function $h \colon \R^N \to \R$ is called \emph{$\A$-quasiconvex} if
\[
  h(A) \leq \int_Q h(A + w(y)) \dd y
\]
for all $A \in \R^N$ and all $Q$-periodic $w \in \Crm^\infty(\R^d;\R^N)$ such that 
\[\A w = 0 \qquad \text{and} \qquad
\int_Q w  \dd x = 0.
\]

For functions $h$ that are not $\A$-quasiconvex one may define the largest $\A$-quasiconvex function below $h$. %(see Lemma~\ref{lem:qce_property} below the definition).
% where $\A^k = \sum_{|\alpha| = k} A_\alpha \partial^\alpha$ is the $k$th order part of $\A$, and $\int_Q w(y) \dd y = 0$.

%
%	A Borel function  $f \colon \R^N \to \R$ is said to be $\A$-quasiconvex if 
%	\[
%	f(z) \leq \int_Q f(v + w(y)) \dd y,
%	\]
%	for all $z \in \R$ and all $Q$-periodic $w \in \Crm^\infty(Q;\R^l)$ such that $\A w = 0$ and $\int_Q w(y) \dd y= 0$. 
	
	\begin{definition}[$\A$-quasiconvex envelope]\label{def:quasi}
	Given a Borel function $h\colon \R^N \to \R$ we define the {$\A$-quasiconvex} envelope of $h$ at $A \in \R^N$ as
	\begin{align*}
	\label{eq: quasi} (Q_{\A}h)(A) := \inf\setBB{\int_{Q} h&(A + w(y)) \dd y }{ w \in \Crm^\infty_\per(Q;\R^N) \cap \ker \A, \; \int_Qw \dd y = 0}.
	\end{align*} 
	For a map $f \colon \Omega \times \R^N \to \R$ we write $Q_{\A}f(x,A)$ for $(Q_{\A}f(x,\frarg))(A)$ by a slight abuse of notation.
\end{definition}

We recall from~\cite{FonsecaMuller99} that the $\A$-quasiconvex envelope of an  upper semicontinuous function is $\A$-quasiconvex and that it is actually the largest $\A$-quasiconvex function below $h$.

\begin{lemma}\label{lem:qce_property}
If $h\colon \R^N \to [0,\infty)$ is upper semicontinuous, then 
	$Q_{\A}h$ is upper semi-continuous and $\A$-quasiconvex. Furthermore, $Q_{\A}h$ is the largest $\A$-quasiconvex function below $h$.
\end{lemma}

% \mnote{A: Added\\
% G: I put together the old~\ref{lem:qce_property} and 2.18. I also removed the proof since I believe people in l.s.c. should now this,  furthermore it is quite straightforward from the definition.}
%\begin{lemma}\label{cor:largest} If $h \colon \R^N \to \R$ is upper semicontinuous, then $Q_{\A^k}h$ is the largest $\A^k$-quasiconvex function below $h$.
%\end{lemma}
%\begin{proof} From the previous lemma we know that $Q_{\A^k}h$ is $\A^k$-quasiconvex. To see that it is the largest function with this property below $h$ we argue as follows:
%Let $g \colon \R^N \to \R$ be an $\A^k$-quasiconvex function such that $g \le h$. Then, by definition,
%\[
%g(A) \le \int_Q g(A + w(y))\dd y \le \int_Q h(A + w(y))\dd y, \quad A \in \R^N,
%\]
%for all $w \in \Crm^\infty_\per(Q;\R^N) \cap \ker \A^k$ with $\int_Q w \dd y = 0$.
%
%Taking the infimum over all such $w$'s and using the definition of $Q_{\A^k}h$ we get
%\begin{align*}
%g(A) &\le Q_{\A^k}h(A).
%\end{align*}
%\end{proof}

\subsection{$\mathcal D$-convexity}\label{Dconv}Let $\mathcal D$ be a balanced cone of directions in $\R^N$, i.e., we assume that $tA \in \mathcal D$ for all $A \in \mathcal D$ and every $t \in \R$. A real-valued function $h \colon \R^N \to \R$ is said to be $\mathcal D$-convex provided its restrictions to all line segments in $\R^N$ with directions in $\mathcal D$ are convex.  Here, $\Dcal$ will always be %A particularly important balanced cone is 
the wave cone $\Lambda_{\A}$ for the linear PDE operator $\A$.

% In the next lines we adapt an argument from~\cite{FonsecaMuller99} to show that integrands with linear growth at infinity which are $\A^k$-quasiconvex are also $\Lambda_{\A}$-convex.

\begin{lemma}\label{lem:cone_convex} 
Let $h:\R^N \to [0,\infty)$ be an integrand with linear growth at infinity. Further, suppose that $h$ is $\A^k$-quasiconvex. Then, $h$ is $\Lambda_{\A}$-convex.
\end{lemma}

%\mnote{A: I modified this lemma in order to sharpen assumptions of Thm~\ref{thm:lsc}. Notice that one does not require $h$ usc.
%\\
%G: Since we do not really need it, I put the older version in order to shorten the paper.}

\begin{proof} Let $\xi \in \Sbb^{d-1}$ and let $A_1, A_2 \in \R^d$ with $P \coloneqq A_1 - A_2 \in \ker \Abb^k(\xi)$. We claim that
\[
h(\theta A_1 + (1 - \theta)A_2) \le \theta h(A_1) + (1-\theta)h(A_2), \quad \text{for all $\theta \in (0,1)$}.
\]
Fix such a $\theta$ and consider the one-dimensional {$1$-periodic} function
\[
\chi(s) \coloneqq (1 - \theta)\mathds 1_{[0,\theta)}(s) - \theta \mathds 1_{[\theta,1)}(s), \quad s \in [0,1), %\theta \in (0,1),
\]
which has zero mean value. Fix $\eps \in \min\{\theta/2,(1 - \theta)/2\}$ so that the mollified function $\chi_\eps \coloneqq \chi \ast \rho_\eps$ has the following properties:
\[
\big|\setb{s}{\chi_\eps = 1 - \theta}\big| \ge \theta - 2\eps, \quad \big|\setb{s}{\chi_\eps = -\theta}\big| \ge (1- \theta) - 2\eps.
\] 

Define the sequence of $Q$-periodic functions
\[
u_\eps \coloneqq P\chi_\eps(y\cdot\xi).
\]
By construction, this is a $\Crm^\infty_\per(Q;\R^N)$ function, it has zero mean value in $Q$, and
%and by Lemma~\ref{lem:homogenization}, we further have that $u_m \toweak 0$ in $\Lrm^1_\per(Q;\R^d)$ and 
%\begin{equation}\label{eq:young_hom}
%u_m \toY \nu = (\cl{\delta_{P\cdot\chi}},0,\delta_0) \quad \text{in $Q$}.
%\end{equation}
since $P \in \ker \Abb^k(\xi)$, it is easy to check that 
\[
\A^k u_\eps = (2\pi\ii)^{-k} \frac{\dd^k  \chi_\eps}{\dd s^k}(y \cdot \xi) \ \Abb^k(\xi) P = 0 \quad \text{in the sense of distributions}.
\]
%Hence, by Lemma~\ref{lem:boundary values}, there exists a sequence $(z_m) \subset \Crm^\infty_\per(Q;\R^N)$ such that
%\[
%\A^k z_m = 0, \quad \int_Q z_m \dd y = 0, \quad z_m \toweakstar 0 \; \text{in $\M(\cl{Q};\R^N)$},
%\]
%for all $m \in \Nbb$, and $z_m \toY \nu$ in $Q$. Hence, by Lemma~\ref{lem:}
Hence, by the definition of $\A^k$-quasiconvexity and our choice of $\eps$, we have 
\begin{align*}
h(\theta A_1 + (1 - \theta)A_2) & \le \int_Q h(\theta A_1 + (1 - \theta)A_2 + u_\eps) \dd y \\
& \le (\theta - 2\eps)h(A_1) + ((1 - \theta) - 2\eps)h(A_2) \\
& \quad + M(1 + |A_1| + |A_2| + |P|)4\eps
\end{align*}
Letting $\eps \todown 0$ in the previous inequality yields the claim.
\end{proof}

The following is an immediate consequence of Lemmas~\ref{lem:qce_property} and~\ref{lem:cone_convex}.

\begin{corollary}\label{cor:coneconvexity} If $h\colon \R^N \to [0,\infty)$ is upper semicontinuous, then 
	$(Q_{\A^k} h)^\#$ is an $\A^k$-quasiconvex and $\Lambda_\A$-convex function. 
	\end{corollary}

To continue our discussion we define the notion of \emph{convexity at a point}. Let $h:\R^N \to \R$ be a Borel function. We recall that Jensen's definition of convexity states that $h$ is convex if and only if
\begin{equation}\label{Jensenineq}
f\bigg(\int_{\R^N} A \dd \nu(A) \bigg) \le \int_{\R^N} h(A) \dd \nu(A) 
\end{equation}
for all probability measures $\nu \in \M^1(\R^N)$.

A Borel function $h:\R^N \to \R$ is said to be \emph{convex at a point} $A_0 \in \R^N$ if~\eqref{Jensenineq} holds for 
for all probability measures $\nu$ with barycenter $A_0$, that is, every $\nu \in \M^1(\R^N)$ with $\int_{\R^N} A \dd \nu = A_0$.

Returning to the convexity properties of $\A^k$-quasiconvex functions, it was recently shown by  Kirchheim and Kristensen~\cite{KirchheimKristensen11,KirchheimKristensen16} that   $\A^k$-quasiconvex and positively $1$-homogeneous integrands are actually {\em convex at points} of $\Lambda_\A$ as long as 
\begin{equation}\label{eq:span}
\text{span} \, \Lambda_\A = \R^N. 
\end{equation}

In fact, their result is valid in the more general framework of $\mathcal D$-convexity:
\begin{theorem}[Theorem~1.1 of~\cite{KirchheimKristensen16}]\label{thm:KK}
	Let $\mathcal D$ be a balanced cone of directions in $\R^N$ and assume that $\mathcal D$ spans $\R^N$. If $h\colon \R^N \to \R$ is $\mathcal D$-convex and positively $1$-homogeneous, then $h$ is convex  at each point of $\mathcal D$.
%	, that is, for each $A_0 \in \mathcal D$ and every probability measure $\kappa \in \Pcal(\R^N)$ with $\int_{\R^d} A \dd\kappa(A) = A_0$, the Jensen inequality
%	\[
%	h(A_0) \leq \int_{\R^N} h(A) \dd\kappa(A)
%	\]
%	holds.
	\end{theorem}

%The previous result is of utter importance when dealing with lower semi-continuity of integral functionals with linear growth because of the very structure of $f^\infty$: an $\A$-quasiconvex and positively $1$-homogeneous function.
%The loss of equi-integrability of bounded sequences conveys the need to look into a generalized functional that adds a term on the singular part of measures via the recession function of $f$. A well-known fact is that convexity of an integrand is a sufficient condition for lower semi-continuity among several types of weak topologies. 

Condition~\eqref{eq:span} holds in several applications, for example in the space of gradients ($\A = \curl$) or the space of divergence-free fields ($\A = \Div$). However, it does not  necessarily hold in our framework as it is evidenced by the operator
\[
\A := A_0 \Delta = \sum_{i=1}^d A_0 \partial_{ii}, % \quad A_{ij} \coloneqq \delta_{ij} A_0,
\]
where $A_0\in \R^{n}\otimes \R^N$  with  $\ker A_0 \neq \R^N$.    

%Perhaps not surprisingly (cf.~\eqref{eq: subspace} and  Theorem~\ref{thm:guido filip}) 
Nevertheless, for our purposes 
it will be sufficient to use the convexity of $f^\#|_{V_{\A}}(x,\frarg)$ in $\Lambda_{\A}$,
which is a direct consequence of Theorem~\ref{thm:KK}. 

\begin{remark}[automatic convexity]\label{rem:qc-convexo}
%\mnote{A: Added. Notice that only the second badge requires usc of $f$.\\
%G: Cahnged according to the previous rmk} 
%Under the assumption that $f:\Omega \times \R^N \to \R$ is upper semicontinous, Theorem~\ref{thm:KK}
%and Corollary~\ref{rem:convexo} ensure that
%\[
%\text{$Q_{\Acal^k}f(x,\frarg)^\#|_{V_{\A}}$  is convex in $\Lambda_{\Acal}$}.
%\]
%\end{remark} 
%
%\mnote{A: usc needed?}
Summing up, in the following we will often make use of the implications from Lemma~\ref{lem:qce_property}, Corollary~\ref{cor:coneconvexity} and Theorem~\ref{thm:KK}: If $f \colon \Omega \times \R^N \to \R$ is an integrand with linear growth at infinity, then
\[
\text{$f(x,\frarg)$ is $\A^k$-quasiconvex} \quad \Longrightarrow \quad 
\begin{cases}
	\quad\text{$f(x,\frarg)$ is $\Lambda_{\A}$-convex in $\R^N$ and}\\[7pt]
\;\text{$f^\#|_{V_{\A}}(x,\frarg)$
 is convex at points in $\Lambda_{\A}$}
	\end{cases},
\]
\[
\text{$f$ upper semicontinuous} \quad \Longrightarrow \quad \begin{cases}
	\;\quad \text{$Q_{\A^k}f(x,\frarg)$ is $\Lambda_{\A}$-convex in $\R^N$ and}\\[7pt]
 \;\, \text{$(Q_{\A^k}f)^\#|_{V_{\A}}(x,\frarg)$
 is convex at points in $\Lambda_{\A}$}
	\end{cases}.
\]
\end{remark}

 \subsection{Localization principles for Young measures}
 
 We state two general localization principles for Young measures, one at {\it regular} points and another one at {\it singular} points. These are $\A$-free versions of the localization principles developed for gradient Young measures and $\BD$-Young measures in~\cite{Rindler11,Rindler12BV}.
 
% The convenience of such localization principles lies in the connection via tangent measures which links the theory of Young measures with the classic blow-up technique introduced by Fonseca and M\"uller in~\cite{FonsecaMuller99}. We believe that one can independently employ both techniques to reach similar conclusions (cf.~\cite{BaiaChermisiMatiasSantos13}); however, we also believe that using the machinery of Young measures provides a more natural and better understanding on the ideas behind our results. \mnote{This paragraph should be edited: We did not show this equivalence of methods!} 
 
\begin{definition}[$\A$-free Young measure] We say that a Young measure $\nu \in \Y(\Omega;\R^N)$ is an $\A$-free Young measure in $\Omega$, in symbols $\nu \in \Y_{\A}(\Omega;\R^N)$, if and only if there exists a sequence $(\mu_j) \subset \M(\Omega;\R^N)$ with $\A \mu_j \to 0$ in $\Wrm^{-k,q}$ for some $1 < q <  d/(d-1)$, and such that $\mu_j \toY \nu$ in $\Y(\Omega;\R^N)$.
\end{definition}

   \begin{proposition}\label{prop: localization regular}
   	Let $\nu \in \Y_{\A}(\Omega;\R^N)$ be an $\A$-free Young measure. Then for $\Lcal^d$-a.e.\ $x_0 \in \Omega$ there exists a {regular tangent $\A^k$-free Young measure} $\sigma \in \Y_{\A^k}(Q;\R^N)$ to $\nu$ at $x_0$, that is, $\sigma$ is generated by a sequence of asymptotically $\A^k$-free measures and
	\begin{align*}
   	[\sigma] \in \Tan_Q([\nu],x_0), \quad &  \quad \sigma_y = \nu_{x_0} \; \text{a.e.}, \\
   	\lambda_\sigma = \frac{\di \lambda_\nu}{\di \Lcal^d}(x_0) \, \Lcal^d\in \Tan_Q(\lambda_\nu,x_0), \quad & \quad \sigma_y^\infty = \nu^\infty_{x_0} \; \text{$\lambda_\sigma$-a.e.}
   	\end{align*}
	Moreover, there exists a sequence $(w_j) \subset \Crm^\infty_\per(Q;\R^N) \cap \ker \A^k$ such that $
	 w_j \Lcal^d \toY \sigma$ in $\Y(Q;\R^N)$.
   \end{proposition}

 \begin{proposition}\label{prop: localization}
 	Let $\nu \in \Y_{\A}(\Omega;\R^N)$ be an $\A$-free Young measure. Then there exists a set 
 	$S \subset \Omega$ with $\lambda^s_\nu(\Omega \setminus S) = 0$ such that for all $x_0 \in S$
 	there exists a non-zero singular tangent $\A^k$-free Young measure $\sigma \in \Y_{\A^k}(Q;\R^N)$ to $\nu$ at $x_0$, that is, $\sigma$ is generated by a sequence of asymptotically $\A^k$-free measures and
 	\begin{align*}
 	[\sigma] \in %|\langle \id , \nu_{x_0}^\infty \rangle| \cdot 
	\Tan_Q([\nu],x_0), \quad &  \quad \sigma_y = \delta_0 \; \text{a.e.}, \\
 	\lambda_\sigma \in \Tan_Q(\lambda_\nu^s,x_0),\qquad \lambda_\sigma(Q) = 1, \quad &
 	 \lambda_\sigma(\partial Q) = 0, \quad \sigma_y^\infty = \nu^\infty_{x_0} \; \text{$\lambda_\sigma$-a.e.}
 	\end{align*}
 \end{proposition}

The proofs for the first part of the statements above are by now standard (see, for instance,~\cite{Rindler12BV}). The existence of an $\A^k$-free generating sequence in Proposition~\ref{prop: localization regular} is obtained by Lemma~\ref{lem:boundary values}. For the sake of readability, the proofs are postponed to the appendix.

\section{Jensen's inequalities}  \label{sc:Jensen}

%We say that $\sigma$ is an $\A$-free tangent Young measure of $\nu$ at $x_0$ if and only if $\sigma$ is a tangent Young measure of 
%$\nu$ at $x_0$ and there exists an open neighborhood $U(x_0)$ of $x_0$ where $\A [\nu] = 0$. In particular, the space of tangent Young measures of a given $\A$-free Young measure in $\Y_{\A}(\Omega;\R^N)$ coincides with the space of $\A$-free tangent Young measures at every $x_0 \in \Omega$. 
%
%Notice that, if $\sigma$ is an $\A$-free tangent Young measure for $\nu$ at $x_0$ generated by a blow-up sequence $\gamma_m := c_m T^{(x_0,r_m)}_\# [\nu]$, then for every open and bounded set $\omega \subset \R^d$ there exists a positive constant  $M = M(\omega)$ such that 
%\begin{equation}\label{eq: A local}
%\A \gamma_m = 0 \quad \text{in $\omega$}, \quad \text{for every $n \ge M$}.
%\end{equation}
%In this sense, $\A$-free tangent Young measures are generated by $\A$-free blow-up sequences at least locally and In particular, $\A [\sigma] = 0$. The next Lemma tells us that $\A$-free tangent Young measures are also $\A^k$-free Young measures. 

%{\color{red}\mnote{F:Moved up}
In this section we establish generalized Jensen inequalities, which can be understood as a local manifestation of lower semicontinuity.  The proof of Theorem~\ref{thm:lsc}, under Assumption~(i), which reads
\begin{equation*}
f^\infty(x,A) \coloneqq \lim_{t \to \infty} \; \frac{f(x,tA)}{t}\quad \text{ exists for all \; $(x,A) \in \Omega \times \R^N$},
\end{equation*}
will easily follow from Propositions~\ref{prop: jensen regular} and~\ref{prop: jensen singular}.

On the other hand, to prove the Theorem~\ref{thm:lsc} under the weaker Assumption~(ii), 
\[
f^\infty(x,A) \coloneqq \lim_{t \to \infty} \; \frac{f(x,tA)}{t}\quad \text{ exists for all \; $(x,A) \in \Omega \times \mathrm{span}\, \Lambda_\A$},
\]
requires to perform a direct blow-up argument for what concerns the regular part of \(\mu\) and only Proposition~\ref{prop: jensen singular} is used in the proof. %Nevertheless, we state both Jensen inequalities since we believe them to be of independent interest.

\subsection{Jensen's inequality at regular points}

%\mnote{G; Added a little bit of explanation}

We first consider regular points.

\begin{proposition}
	\label{prop: jensen regular} Let $\nu \in \Y_{\A}(\Omega;\R^N)$ be an $\A$-free Young measure. Then, for $\Lcal^d$-almost every $x_0 \in \Omega$ it holds that 
	\[
	h\left(\dprb{ \id, \nu_{x_0}} + \dprb{ \id , \nu^\infty_{x_0}} \frac{\di \lambda_\nu}{\di \Lcal^d}(x_0) \right) \leq 
	\dprb{ h ,\nu_{x_0} } + \dprb{ h^\# , \nu_{x_0}^\infty } \frac{\di \lambda_\nu}{\di \Lcal^d}(x_0),
	\] 
	for all upper semicontinuous and $\A^k$-quasiconvex $h \colon \R^N \to [0,\infty)$ with linear growth at infinity.  
\end{proposition}
%\mnote{G: It seems it is enough  to assume \(h\) usc, this helps in the proof of Corollary~\ref{cor:principle}}
\begin{proof}
	We make use of Lemma~\ref{lem:E_approx} to get a collection $\{h_m\} \subset \e(\Omega;\R^N)$ such that 
$h_m \todown h$, $h_m^\infty \todown h^\#$ pointwise in $\Omega$ and $\cl{\Omega}$ respectively, all $h_m$ are Lipschitz continuous and have uniformly bounded  linear growth constants. 
Fix $x_0 \in \Omega$ such that there exists a regular tangent measure $\sigma \in \Y_{\A^k}(Q;\R^N)$ of $\nu$ at $x_0$ as in Proposition
\ref{prop: localization regular}, which is possible for $\Lcal^d$-a.e.\ $x_0 \in \Omega$.
The  localization principle for regular points tells us that $[\sigma] = A_0 \Lcal^d$ with
\[
A_0 := \dprb{ \id, \nu_{x_0} } + \dprb{ \id, \nu_{x_0}^\infty } \frac{\di \lambda_\nu}{\di \Lcal^d}(x_0)  \quad \in \R^N,
\]
and that we may find a sequence $z_j \in \Crm^\infty_\per(Q;\R^N) \cap \ker \A^k$ with 
$\int_{Q} z_j \dd y = 0$ and satisfying
\begin{equation}\label{eq: promedio}
(A_0 + z_j)\Lcal^d \toY \sigma \quad \text{in $\Y(Q;\R^N)$}.
\end{equation}

Fix $m \in \N$. We use the fact that 
$\int_{Q} z_j \dd y = 0$,~\eqref{eq: promedio} and the $\A^k$-quasiconvexity of $h$, to get for every $m \in \N$ that
\begin{align*}
\dprb{ h_m , \nu_{x_0}}  + \dprb{ h_m^\infty , \nu^\infty_{x_0} } \frac{\di \lambda_\nu}{\di \Lcal^d}(x_0)& = \frac{1}{|Q_r|} \ddprb{ \mathds 1_{Q} \otimes h_m, \sigma } \\
& = \lim_{j \to \infty} \dashint_{Q} h_m(A_0 + z_j(y)) \dd y \\
& \ge \limsup_{j \to \infty} \dashint_{Q} h(A_0 + z_j(y)) \dd y \\
& \ge h(A_0).
\end{align*}
The result follows by letting $m \to \infty$ in the previous inequality and using the 
monotone convergence theorem.
\end{proof}

\subsection{Jensen's inequality at singular points} The strategy for {\it singular points} differs from the  regular case as one cannot simply use the definition of $\A^k$-quasiconvexity. The latter difficulty arises because tangent measures at a singular point may not be multiples of the $d$-dimensional Lebesgue measure. 

In order to circumvent this obstacle, we will first show that, for $\Acal$-free Young measures, the support of the singular part $\nu^\infty$ at singular points is contained in the subspace $V_\A$ of $\R^N$ (see Lemma~\ref{lem:subspace} below). Based on this, we invoke Theorem~\ref{thm:KK}, which states that an $\A^k$-quasiconvex and positively {$1$-homogeneous} function is actually convex at points in $\Lambda_\A$ when restricted to $V_{\A}$. Then, the Jensen inequality for $\A$-free Young measures at singular points follows. 

\begin{lemma}\label{lem:subspace}
Let $\sigma \in \Y_{\A^k}(Q;\R^N)$ be an $\A^k$-free Young measure with ${\lambda_\sigma(\partial Q) = 0}$.
%	Let $(u_j) \subset \Lrm^1(Q;\R^N)$ be a sequence  $\A^k$-free Young measure $\sigma \in\Y_{\A^k}(Q;\R^N)$, i.e., $\A^k  u_n \to 0$ in the space $\Wrm^{-k,q}(Q;\R^N)$ for some $q \in (1,d/(d-1))$ and $u_j \toY \sigma$.
%Furthermore,  assume that $|u_h| \toweakstar \Lambda$ in $\M_+(Q)$ with $\Lambda(\partial Q) = 0$ 
Assume also that 
\[
[\sigma] \in\M(Q;V_{\A}).
\]
Then,%t\mnote{F: $\nu_x^\infty$ to $\sigma_x^\infty$}
	%\begin{itemize}
     %\item[(i)] $\supp \nu_x \subset V_\A$ for a.e.\ $x \in \Omega$, and \vspace{3pt}
     %\item[(ii)] 
     \[
   \supp \sigma_y^\infty \subset V_\A \cap \Sbb^{N-1} \qquad \text{for $\lambda_\sigma$-a.e.\ $y \in Q$}.\]
	%\end{itemize}
\end{lemma}
\begin{proof} By definition, we may find a sequence $(\mu_j) \subset \M(Q;\R^N)$ with $\A \mu_j \to 0$ in $\Wrm^{-k,q}(Q)$ for some $q \in (1,d /(d -1))$, and such that $(\mu_j)$ generates the Young measure $\sigma$. Notice that, since $\A^k$ is a homogeneous operator and $Q$ is a strictly star-shaped domain, we may re-scale and mollify  each $\mu_j$ into some $u_j \in \Lrm^2(Q;\R^N)$ with the following property: the  sequence $(u_j)$ also generates $\sigma$ and $\A u_j \to 0$ in $\Wrm^{-k,q}(Q)$. In particular,
	\[
	u_j \Lcal^d \toweakstar [\sigma] \quad \text{in $\Mcal(Q;\R^N)$}.
	\]
	On the other hand, $\A^k([\sigma]) = 0$ and for every $0 < r < 1$ the measure 
	$T^{(0,r)}_\# [\sigma]$ is still an $\A^k$-free measure on $Q$. Thus, letting $r \toup 1$ and mollifying the measure $T^{(0,r)}_\#[\sigma]$ on a sufficiently small scale (with respect to $1 -r$) we might find a sequence $(v_j)\subset \Lrm^2(Q;\R^N) \cap \ker \A^k$ such that
	\[
	v_j\Lcal^d \toweakstar [\sigma]  \quad \text{in $\Mcal(Q;\R^N)$}.
	\] 
	Hence, 
	\[
	u_j \Lcal^d - v_j \Lcal^d \toweakstar 0 \quad \text{in $\Mcal(Q;\R^N)$, \quad $|u_j\Lcal^d| + |v_j\Lcal^d| \toweakstar \Lambda$ in $\M^+(\overline Q)$}
	\]
	and $\Lambda(\partial Q) = 0$. Here, we have used that $\lambda_\sigma(\partial Q) = 0$.
	
	We are now in position to apply Lemma~\ref{lem:boundary values} to the sequences
	$(u_j)$, $(v_j)$. There exists (possibly passing to a subsequence in the $v_j$'s) a sequence  $z_j \in \Crm^\infty_\per(Q;\R^N) \cap \ker \A^k$ with $z_j\Lcal^d \toweakstar 0$ and such that
	\[
	v_j\Lcal^d + z_j \Lcal^d\toY \sigma \quad \text{in $\Mcal(Q;\R^N)$}.
	\]
Recall from observation~\eqref{eq: subspace} that $v_j, z_j \in \Lrm^2(Q;V_{\A})$ for every $j \in \N$. 
%Even more, since $T^{(r,0)}_\#$, cutoff and mollification are automorphisms on spaces of the form $\Mcal(\R^d;V)$ (with $V$ a subspace of $\R^d$) it must also hold that $v_h(x) \in V_\A$ for a.e.\ $x \in V_\A$ and every $h \in \N$. 
Therefore, 
\[
(v_j + z_j) \in  \Lrm^2(Q;V_{\A}) \quad \text{for all $j \in \Nbb$}.
\]
We conclude with an application of Lemma~\ref{thm:alibert version}~(ii) to the sequence $(v_j + z_j)$, which yields
\begin{gather*}
%\supp \nu_x \subset V_\A \quad \text{for a.e.\ $x \in Q$,}\label{eq: characterization1} \\
\supp \sigma_y^\infty \subset V_\A \cap \Sbb^{N-1} \qquad \text{for $\lambda_\sigma$-a.e.\ $y \in Q$}.\label{eq: characterization}
\end{gather*}
This finishes the proof.
\end{proof}

\begin{proposition}\label{prop: jensen singular}
	Let $\nu \in \Y_{\A}(\Omega;\R^N)$ be an $\A$-free Young measure. Then for $\lambda^s_\nu$-almost every $x_0 \in \Omega$ it holds that 
	\[
	g\big(\dprb{ \id , \nu^\infty_{x_0}}\big) \leq \dprb{ g, \nu^\infty_{x_0}} 
	\]
	for all $\Lambda_{\A}$-convex and positively $1$-homogeneous functions {$g : \R^N \to \R$}. 
\end{proposition}
\begin{proof}
%For the sake of clarity, we partition the proof into two steps. We first show that for $\lambda^s_\nu$-a.e.\ point in $\Omega$, the support of the singular part of the generalized Young measure $\nu$ is contained in the subspace $V_{\A} \subset \R^N$ generated by the wave cone $\Lambda_\A$. Next, we use this fact and the improved convexity from Theorem~\ref{thm:KK} to show a Jensen type inequality on such points. 
	\proofstep{Step~1: Characterization of the support of $\A$-free Young measures.}
	Let $S$ be the set given by Proposition~\ref{prop: localization}, which has full $\lambda_\nu^s$-measure. %\mnote{F: There was a $\mu$ before, which was not defined, redefined $S'$} 
	Further, also the set
	\[
	S' := \setb{x \in \Omega }{ \dprb{ \id , \nu^\infty_{x}} \in \Lambda_\A } \subset \Omega
	\]
	has full $\lambda_\nu^s$-measure: Observe first that
	\[
	  [\nu]^s = \dprb{\id, \nu_x^\infty} \, \lambda_\nu^s(\di x).
	\]
	Since $[\nu]$ is $\A$-free, we thus infer from Theorem~\ref{thm:guido filip} that $\dpr{\id, \nu_x^\infty} \in \Lambda_\A$ for $\abs{[\nu]^s}$-a.e.\ $x \in \Omega$. On the other hand, $\dpr{\id, \nu_x^\infty} = 0 \in \Lambda_\A$ for $\lambda_\nu^\ast$-a.e.\ $x \in \Omega$, where $\lambda_\nu^\ast$ is the singular part of $\lambda_\nu^s$ with respect to $\abs{[\nu]^s}$. This shows that $S'$ has full $\lambda_\nu^s$-measure.
	
	Fix $x_0 \in S \cap S'$ (which remains of full $\lambda_\nu^s$-measure in $\Omega$).
	%Due to the essential nature of the statement it will suffice to prove the result for points in $S \cap S'$.
	Let $\sigma \in \Y_{\A^k}(Q;\R^N)$ be the non-zero singular tangent Young measure to $\nu$ at $x_0$ given by Proposition~\ref{prop: localization} which according to the same proposition satisfies that $\lambda_\sigma(Q) = 1$ and $\lambda(\partial Q) = 0$. 
	%Recall that, by a similar argument as in the proof of~\eqref{WLOG}, we may assume there exists a sequence $(u_j) \subset \Lrm^1(Q;\R^N)$ such that
%		\begin{equation}\label{eq:all}
%		\text{$u_j\Lcal^d \toY \sigma$ in $\Y(Q;\R^N)$} \qquad\text{and}\qquad
%		\A^k u_j \to 0 \quad \text{in $\Wrm^{-k,q}(Q;\R^N)$}.
%		\end{equation}
%%	 Here, we use that the lower-order oscillations disappear for blow-up sequences of $\A$-free measures (cf. Lemma~\ref{prop:nonhomo}), and that the same holds for asymptotically $\A$-free sequences.
		On the one hand, since $x_0 \in S$, it holds that
	\[
	\sigma_y = \delta_0 \quad \text{$\Lcal^d$-a.e.} \qquad \text{and} \qquad
	\sigma_y^\infty = \nu_{x_0}^\infty \quad \text{$\lambda_\sigma$-a.e.}
	\]
	On the other hand, we use the fact that $x_0 \in S'$ to get
	\begin{equation}\label{eq: bary}
	\dprb{ \id , \nu^\infty_{x_0}} \in \Lambda_\A 
	\qquad\text{and}\qquad
	[\sigma]  = \dprb{ \id , \nu^\infty_{x_0} } \, \lambda_\sigma \in \M(Q;V_{\A}).
	\end{equation}
	Note that, by~\eqref{eq: bary},  all the hypotheses of Lemma~\ref{lem:subspace} are satisfied for $\sigma$.
	Thus, 
	\begin{equation*}
	\supp \nu_{x_0}^\infty = \supp \sigma_y^\infty \subset V_\A \qquad \text{for $\lambda_\sigma$-a.e.\ $y \in Q$}.
	\end{equation*}
	This equality and the fact that $\lambda_\sigma(Q) > 0$ (recall that $\sigma$ is a non-zero singular measure) yield 
	\begin{equation}\label{eq: paso}
	\supp \nu_{x_0}^\infty \subset V_\A  \qquad
	\text{for $\lambda_\nu^s$-a.e.\ $x_0 \in \Omega$.}
    \end{equation}
    
    \proofstep{Step~2: Convexity of $g$ on $\Lambda_\A$.}
    The Kirchheim--Kristensen Theorem~\ref{thm:KK} states that the restriction $g|_{V_{\A}} \colon V_{\A} \subset \R^N \to \R$ is a convex function at points $A_0 \in \Lambda_\A$. In other words, for every probability measure $\nu \in \M^1(\R^N)$ with $\dpr{\id,\nu} \in \Lambda_{\A}$ and $\supp \nu \subset V_\A$, the Jensen inequality
	\[
	g\left( \int_{\R^N} A \dd\nu(A) \right) \leq \int_{\R^N} g(A) \dd\nu(A)
	\]  
	holds. Hence, because of~\eqref{eq: bary} and~\eqref{eq: paso}, it follows that
	\[
	g\big(\dprb{ \id , \nu^\infty_{x_0}}\big) \leq \dprb{ g, \nu^\infty_{x_0}}.
	\]
	This proves the assertion.
\end{proof}

The following simple corollary will be important in the proof of Theorem~\ref{thm:relax}.

\begin{corollary}\label{cor:principle}
%\mnote{A: Added. The purpose of this lemma is to give a short proof of the lower bound (Thm.~\ref{thm:relax}) that does not require $(Q_{\A}f)^\infty$ to exist.}
Let $h:\R^N \to \R$ be an upper semicontinuous integrand with linear growth at infinity and let $\nu \in \Y_{\A}(\Omega;\R^N)$ be an $\A$-free Young measure. Then, for $\Lcal^d$-almost every $x_0 \in \Omega$ it holds that 
	\[
	Q_{\A^k}h\left(\dprb{\id, \nu_{x_0}} + \dprb{\id , \nu^\infty_{x_0}} \frac{\di \lambda_\nu}{\di \Lcal^d}(x_0) \right) \leq 
	\dprb{h ,\nu_{x_0}} + \dprb{ h^\# , \nu_{x_0}^\infty } \frac{\di \lambda_\nu}{\di \Lcal^d}(x_0).
	\] 
	Moreover, for $\lambda_\nu^s$-a.e. $x_0 \in \Omega$ it holds that
	\[
	(Q_{\A^k}h)^\#\big(\dprb{ \id , \nu^\infty_{x_0}}\big) \leq \dprb{ h^\#, \nu^\infty_{x_0}} 
	\]
\end{corollary}
\begin{proof}
The proof follows by combining Propositions~\ref{prop: jensen regular} and~\ref{prop: jensen singular}, Lemma~\ref{lem:qce_property}, Corollary~\ref{cor:coneconvexity} and the trivial inequalities  $Q_{\A^k}h \le h$, $(Q_{\A^k}h)^\# \le h^\#$.
 
%  then the desired conclusion is an immediate consequence of the $\Acal^k$-quasiconvexity of $Q_{\Acal^k}h$ and the localization principles of Propositions~\ref{prop: jensen regular} and~\ref{prop: jensen singular}. Thus, we are left to show that $Q_{\A^k}h \le h$ and $(Q_{\A^k}h)^\# \le h^\#$. The former inequality is trivial, thence we only discuss the latter inequality:  We claim that $(Q_{\A^k}h)^\# \le Q_{\A^k}h^\#$. Indeed, by a reverse Fatou's  inequality we have 
%\[
%(Q_{\A^k}h)^\#(A)
%%\limsup_{t \to \infty} \frac{Q_{\A^k}h(tA)}{t} 
%\le \limsup_{t \to \infty} \int_\Omega \frac{h(tA + tw(y))}{t} \dd y \le \int_\Omega h^\#(A + w(y)) \dd y,
%\]
%for all $w \in \Crm^\infty_\per(Q;\R^N) \cap \A^k$ with $\int_Q w \dd y = 0$. Hence, taking the infimum over all such $w$'s yields
%\[
%(Q_{\A^k}h)^\#(A) \le Q_{\A^k}h^\#(A) \le h^\#(A).
%\] 
\end{proof}

%\subsection{The support of the singular part of an $\A$-free Young measure} As we have already discussed in the introduction (and in more detail in Section~\ref{sec:rigid}) the recent developments on the analysis of $\A$-free measures point that the differential constraint obligates for singularities to lie in $\Lambda_{\A}$, that is
%\[
%\frac{\dd \mu}{\dd |\mu^s|}(x) \in \Lambda_{\A} \quad \text{for $|\mu^s|$-a.e. in the support of $\mu$}.
%\]
%The counterpart for $\A$-free Young measures does has little chance to hold since for such a Young measure it is possible to capture oscillations and concentrations coming from two or more directions (in $\Lambda_{\A}$). However, given $\nu \in \Y_{\A}(\Omega;\R^N)$, our next result states that there is still a degree of rigidity acting on the support of $\nu_x^\infty$:
%
%\begin{theorem}\label{thm:support}
%Let $\nu \in \Y(\Omega;\R^N)$ be a (generalized) $\A$-free Young measure. Then, 
%\[
%\supp \nu_{x_0}^\infty \subset \textnormal{span}\,\Lambda_{\A} \cap \Sbb^{N-1} \quad \text{for $\lambda_\nu$-a.e $x \in \Omega$}.
%\]
%\end{theorem} 
%\begin{proof}
%The proof follows from the decomposition $\lambda_\nu = \frac{\dd \lambda_\nu}{\dd \Lcal^d} \Lcal^d + \lambda_\nu^s$, Corollary~\ref{cor:three}, and~\eqref{eq: paso} in the proof of Proposition~\ref{prop: jensen regular}.
%\end{proof}

\section{%A Young measure proof of a weaker version of 
Proof of Theorems~\ref{thm:lsc} and~\ref{thm:lsc_partial}}\label{sc:lsc_special}

%Before proceeding to the proof of the the lower semicontinuity theorem (Theorem~\ref{thm:lsc}) in full generality in the next section, we show that under stronger assumptions on the integrand $f:\Omega \times \R^N \to [0,\infty)$ it is possible to give a very short  proof of Theorem~\ref{thm:lsc} based on the Young measure approach. 

%In order to give the prove Theorem~\ref{thm:lsc} we first give the proof of Theorem~\ref{thm:lsc_partial}.

\begin{proof}[Proof of Theorem~\ref{thm:lsc}] 
%We only show the sufficiency part ($\Acal^k$-quasiconvexity implies weak* lower semicontinuity), as the other direction is well-known, see~\cite{FonsecaMuller99}. 
We will prove Theorem~\ref{thm:lsc} in full generality, which means that we consider asymptotically $\A$-free sequences in the $\Wrm^{-k,q}$-norm for some $q \in (1,d/(d - 1))$; see Remark \ref{thm:lsc_ext}. 

\proofstep{Proof under Assumption (i).}  Let $\mu_j$ be a sequence in $\M(\Omega;\R^N)$ weakly* converging to a limit $\mu$ and assume furthermore that $\A \mu_j \to 0$ in $\Wrm^{-k,q}(\Omega;\R^N)$ for some $q \in (1,d/(d-1))$. Up to passing to a subsequence, we might also assume that
\[
\liminf_{j \to \infty} \Fcal[\mu_j] = \lim_{j \to \infty} \Fcal[\mu_j]
\]
and that $\mu_j \toY \nu$ for some $\A$-free Young measure $\nu \in \Y_{\A}(\Omega;\R^N)$. Using the continuity of $f$ and representation of Corollary~\ref{cor:representation} we get
\[
\Fcal[\mu_j] = \ddprb{f,\delta[\mu_j]} \to \ddprb{f,\nu} \quad \text{as $j \to \infty$}.
\]
The positivity of $f$ further lets us discard possible concentration of mass on $\partial \Omega$, 
\begin{equation}\label{eq:mientras}
\begin{split}
\lim_{j \to \infty} \Fcal[\mu_j] & = \int_{\cl \Omega} \dprb{f(x,\frarg),\nu_x}\dd x + \int_{\cl \Omega} \dprb{f^\infty(x,\frarg),\nu_x^\infty} \dd \lambda_\nu(x) \\
& \ge \int_{\Omega} \bigg(\dprb{f(x,\frarg),\nu_x} +\dprb{f^\infty(x,\frarg),\nu_x^\infty}\frac{\dd \lambda_\nu}{\dd \Lcal^d}(x)\bigg) \dd x \\
& \qquad + \int_\Omega  \dprb{f^\infty(x,\frarg),\nu_x^\infty} \dd \lambda^s_\nu(x). 
%\\
%& \ge \int_{\Omega} f\bigg(x,\frac{\dd \mu}{\dd \Lcal^d}(x)\bigg) \dd x + \int_\Omega f^\infty\bigg(x,\frac{\dd \mu}{\dd |\mu^s|}(x)\bigg) = \Fcal[\mu].
\end{split}
\end{equation}

By assumption, $f(x,\frarg) \in \Crm(\R^N)$ has linear growth at infinity. Hence we might apply Proposition~\ref{prop: jensen regular} to get
\[
	f\bigg(x,\dprb{\id, \nu_{x}} + \dprb{\id , \nu^\infty_{x}} \frac{\di \lambda_\nu}{\di \Lcal^d}(x)\bigg) \leq 
	\dprb{f(x,\frarg) ,\nu_{x}} + \dprb{f(x,\frarg)^\infty , \nu_{x}^\infty} \frac{\di \lambda_\nu}{\di \Lcal^d}(x)
	\] 
	for $\Lcal^d$-a.e.~$x\in \Omega$ (recall that under the present assumptions $f^\infty = f^\#$). Likewise, we apply Proposition~\ref{prop: jensen singular} to the functions $f(x,\frarg)^\#$ to obtain 
\[
	f(x,\frarg)^\infty\big(\dprb{\id , \nu^\infty_{x}}\big) \leq \dprb{f(x,\frarg)^\infty, \nu^\infty_{x}} 
	\]
at $\lambda^s_\nu$-a.e.~$x \in \Omega$. Plugging these two Jensen-type inequalities into~\eqref{eq:mientras} yields
\begin{equation}\label{eq:mientras2}
\begin{split}
\lim_{j \to \infty} \Fcal[\mu_j] &  \ge \int_{\Omega} f\bigg(x, \dprb{\id, \nu_{x}} + \dprb{\id , \nu^\infty_{x}} \frac{\di \lambda_\nu}{\di \Lcal^d}(x)\bigg) \dd x \\ 
& \qquad + \int_{\Omega} f^\infty\big(x,\dprb{\id_{\R^N},\nu_x^\infty}\big) \dd \lambda_\nu^s(x).
\end{split}
\end{equation}
Finally, since $\mu_j \toY \nu$, it must hold that 
\[
\dprb{ \id, \nu_{x}} + \dprb{ \id , \nu^\infty_{x}} \frac{\di \lambda_\nu}{\di \Lcal^d}(x) = \frac{\dd \mu}{\dd \Lcal^d}(x) \quad \text{for $\Lcal^d$-a.e. $x\in \Omega$, \quad and}
\]
\[
\dprb{\id_{\R^N},\nu_x^\infty} \lambda^s_\nu = \mu^s \quad \Rightarrow \quad \frac{\dd \mu^s}{\dd |\mu^s|}(x) = \frac{\dprb{\id_{\R^N},\nu_x^\infty}}{\big|\dprb{\id_{\R^N},\nu_x^\infty}\big|} \quad \text{for $\lambda^s_\nu$-a.e. $x \in \Omega$}.
\]
We can use this representation and the fact that $f^\infty(x,\frarg)$ is positively $1$-homogeneous in the right hand side of~\eqref{eq:mientras2} to conclude 
\begin{equation*}
\begin{split}
\lim_{j \to \infty} \Fcal[\mu_j] & \ge \int_{\Omega} f\bigg(x,\frac{\dd \mu}{\dd \Lcal^d}(x)\bigg) \dd x \\
& \qquad +  \int_{\Omega}f^\infty\bigg(x,\frac{\dd \mu^s}{\dd |\mu^s|}(x)\bigg) \dd \big(\big|\dprb{\id_{\R^N},\nu_x^\infty}\big| \lambda^s_\nu\big)(x) \\
& = \int_{\Omega} f\bigg(x,\frac{\dd \mu}{\dd \Lcal^d}(x)\bigg) \dd x \\
& \qquad +  \int_{\Omega}f^\infty\bigg(x,\frac{\dd \mu^s}{\dd |\mu^s|}(x)\bigg) \dd |\mu^s|(x) \\
&= \Fcal[\mu].
\end{split}
\end{equation*}
This proves the claim under Assumption (i).

\proofstep{Proof under Assumption (ii).} For a measure $\mu \in \M(\Omega;\R^N)$, consider the functional 
\[
\Fcal^\#[\mu;B] := \int_B f\bigg(x,\frac{\di \mu}{\di \Lcal^d}(x)\bigg) \dd x + \int_B f^\#\bigg(x,\frac{\di \mu^s}{\di |\mu^s|}(x)\bigg) \dd |\mu^s|(x),
\]
defined for any Borel subset $B \subset \Omega$.

Let $\mu_j$ be a sequence in $\M(\Omega;\R^N)$ that weakly* converges to a limit $\mu$ and assume furthermore that $\A \mu_j \to 0$ in $\Wrm^{-k,q}(\Omega;\R^N)$ for some $q \in (1,d/(d-1))$. Define $\lambda_j \in \M^+(\Omega)$ via
\[
\lambda_j(B) := \Fcal^\#[\mu_j;B] \qquad\text{for every Borel $B \subset \Omega$}.
\]
%Since, up to a subsequence, we can assume without loss of generality that  
%\[
%\liminf_{j \to \infty}  \Fcal^\#[\mu_j;B]=\lim_{j \to \infty}\Fcal^\#[\mu_j;B]<\infty,
%\]
%Using the $k$-uniform boundedness of $\Fcal^\#[\mu_j;B]$ for every $B \subset \Omega$ (which follows from the linear growth of $f$ and the weak* convergence of $\mu_j$),
We may find a (not relabeled) subsequence and positive measures $\lambda,\Lambda \in \M_+(\Omega)$ such that 
\[
\lambda_j \toweakstar \lambda, \;\; |\mu_j| \toweakstar \Lambda \quad \text{in $\M^+({\Omega})$}.
\]
We claim that
\begin{align}
\label{hybrid1}\frac{\di \lambda}{\di \Lcal^{d}}(x_0) &\ge f\bigg(x_0,\frac{\di \mu}{\di \Lcal^d}(x_0)\bigg) &\text{for $\Lcal^d$-a.e.\ $x_0 \in \Omega$}, &&\\
\label{hybrid2}\frac{\di \lambda}{\di |\mu^s|}(x_0) &\ge f^\#\bigg(x_0,\frac{\di \mu^s}{\di |\mu^s|}(x_0)\bigg) &\text{for $|\mu^s|$-a.e.\ $x_0 \in \Omega$}. &&
\end{align}
Notice that, if~\eqref{hybrid1} and~\eqref{hybrid2} hold, then the assertion of the theorem immediately follows. Indeed, by the Radon--Nikod\'{y}m theorem,
\[
\lambda \geq \frac{\di \lambda}{\di \Lcal^d} \Lcal^d + \frac{\di \lambda}{\di |\mu^s|}  |\mu^s|.
\]
Hence, we obtain
\begin{align}
\liminf_{j \to \infty} \Fcal^\#[\mu_j] & = \liminf_{j \to \infty} \lambda_j(\Omega)  \notag\\
&\ge \lambda(\Omega) \notag\\
&{\ge} \int_\Omega \frac{\di \lambda}{\di \Lcal^d} \dd x + \int_{\Omega}\frac{\di \lambda}{\di |\mu^s|} \dd |\mu^s| \notag\\
& \ge \int_\Omega f\bigg(x,\frac{\di \mu}{\di \Lcal^d}(x)\bigg) \dd x + \int_\Omega f^\#\bigg(x,\frac{\di \mu^s}{\di |\mu^s|}(x)\bigg) \dd |\mu^s| \notag\\
& = \Fcal^\#[\mu].\label{mio}
\end{align}
With~\eqref{hybrid1},~\eqref{hybrid2}, which are proved below, the result under Assumption (ii) follows. This completes the proof of the theorem.
\end{proof}

We now prove~\eqref{hybrid1} and~\eqref{hybrid2}. Let us first show the following auxiliary fact.

\begin{lemma}\label{lem:aproximacion} Let $x_0 \in \Omega$ and $R > 0$ be such that $Q_{2R}(x_0) \subset \Omega$.
Then, for every $h \in \Nbb$, there exists a sequence $\big(u^{h}_{j}\big) \subset \Lrm^2(\R^d;\R^N)$ such that
\begin{equation}\label{eq:mollifier approximation}
\begin{split}
& u^{h}_{j} \to \mu_j \quad \text{area-strictly in $\M(Q_R(x_0);\R^N)$}
\qquad\text{and} \\
& {\|\A^k u^{h}_{j} - \A^k \mu_j \|_{\Wrm^{-k,q}(Q_R(x_0))} \to 0.} \qquad\text{as \quad $h \to \infty$}
\end{split}
\end{equation}
\end{lemma}
\begin{proof}
Let $\{\rho_{\varepsilon}\}_{\varepsilon > 0}$ be a family of standard smooth mollifiers.  The sequence defined by
\[u^{h}_{j}:= \big({\mu_j\restrict Q_{3R/2}(x_0)\big) \ast \rho_{1/h}}  \in \Crm^\infty(\cl{Q_{2R}(x_0)};\R^N)
\] 
satisfies all the conclusion properties as a
%. Indeed, that 
%\[
%\big\|\A u_h^{r,j} - \A \mu_j \big\|_{\Wrm^{-k,q}(Q_r(x_0))} \to 0
%\]
consequence of the properties of mollification and Remark~\ref{rem:area}%\begin{gather*}
%u^{(r,k)}_{h} \stackrel{\barea}\longrightarrow \mu_j \quad \text{in $\M(Q_r(x_0);\R^N)$}, \\
%\|\A u^{(r,k)}_{h} - \A \mu_j \|_{\Wrm^{-k,q}(Q_r(x_0))} \to 0, \qquad \text{as $h \to \infty$}.
%\end{gather*}   
\end{proof}

\begin{proof}[Proof of~\eqref{hybrid1}]
We employ the classical blow-up method to organize the proof.
We know from Lebesgue's differentiation theorem and~\eqref{eq:ae tangent} that the following properties hold for $\Lcal^d$-almost every $x_0$ in $\Omega$:
\begin{gather*}
\frac{\di \lambda}{\di \Lcal^{d}}(x_0) = \lim_{r \todown 0} \frac{\lambda(Q_r(x_0))}{r^d} < \infty, \qquad
\lim_{r \todown 0} \frac{|\mu^s|(Q_r(x_0))}{r^d} < \infty,\\
\lim_{r \todown 0}\frac{1}{r^d}\int_{Q_r(x_0)} \bigg|\frac{\di \mu}{\di \Lcal^d}(y) - \frac{\di \mu}{\di \Lcal^d}(x_0)\bigg| \dd y = 0,\\
\lim_{r \todown 0}\frac{1}{r^d}\int_{Q_r(x_0)} \bigg|\frac{\di \Lambda}{\di \Lcal^d}(y) -\frac{\di \Lambda}{\di \Lcal^d}(x_0)\bigg| \dd y = 0,
\end{gather*}
and
\begin{gather}\label{eq:tangent assumption}
\Tan(\mu,x_0) = \setBB{\alpha \cdot \frac{\di \mu}{\di \Lcal^d}(x_0) \, \Lcal^d}{\alpha \in \R^+ \cup \{0\}}.
\end{gather}

Let $x_0 \in \Omega$ be a point where the properties above are satisfied. Since $\Omega$ is an open set, there exists a positive number $R$ such that $Q_{{2R}(x_0)} \subset \Omega$.
From Lemma~\ref{lem:aproximacion}, we infer that for almost every $r \in (0,R)$, it holds that
\begin{align}
\wslim_{j \to \infty} \; \wslim_{h \to \infty} \, \bigl[ u^{h}_j(x_0 + r y) \, \Lcal^d_y \bigr]
&= \wslim_{j \to \infty} \; \wslim_{h \to \infty} \, r^{-d}T^{(x_0,r)}_\# [u^{h}_j \Lcal^d]  \notag\\
&= \wslim_{j \to \infty} \, r^{-d} T_\#^{(x_0,r)} \mu_j  \notag\\
&= r^{-d}T^{(x_0,r)}_\# \mu,   \label{eq:transition limit}
\end{align}
where the weak* convergence is to be understood in $\M(Q;\R^N)$. %Indeed, fixing $r \in (0,1)$ and $k \in \Nbb$, we may take the limit $h \to \infty$ to obtain
%\begin{equation}\label{eq:transition limit}
%r^d u^{(r,k)}_{h}(x_0 + r \frarg) \, \Lcal^d= T_\#^{(x_0,r)} (u^{(r,k)}_{h}\, \Lcal^d) \toweakstar (T_\#^{(x_0,r)} \mu_j)  %\toweakstar (T_\#^{(x_0,r)} \mu)
%\end{equation}
%in $\M(Q;\R^N)$. 
Thus, choosing a sequence $r \todown 0$ with $\lambda_j(\partial Q_r(x_0)) = 0$ and $\Lambda(\partial Q_r(x_0)) = 0$ (by the finiteness of these measures), we get that
\begin{align}
\frac{\di \lambda}{\di \Lcal^{d}}(x_0) & = \lim_{r \to 0} \lim_{j \to \infty}\frac{\lambda_j(Q_r(x_0))}{r^d}  \notag\\
&= \lim_{r \to 0} \lim_{j \to \infty} \frac{\Fcal^\# [\mu_j ; Q_r(x_0)]}{r^d}  \notag\\
& \ge \limsup_{r \to 0} \limsup_{j \to \infty}\limsup_{h \to \infty }\frac{\Fcal^\#[u^{h}_{j} \Lcal^d; Q_r(x_0)]}{r^d}  \notag\\
&= \limsup_{r \to 0} \limsup_{j \to \infty}\limsup_{h \to \infty } \int_Q f \bigl( x_0 + ry, u^{h}_j(x_0 + ry) \bigr) \dd y,    \label{eq:transition limit2}
\end{align}
where we used Corollary~\ref{cor:representation} and Remark~\ref{rem:strict} for the \enquote{$\geq$} estimate. 

Moreover, by the Lebesgue differentiation theorem (see~\eqref{eq:tangent assumption}),
\begin{equation} \label{eq:transition limit3}
  r^{-d}T^{(x_0,r)}_\# \mu \toweakstar \frac{\di \mu}{\di \Lcal^d}(x_0) \, \Lcal^d.
\end{equation}

By~\eqref{eq:transition limit},~\eqref{eq:transition limit2},~\eqref{eq:transition limit3} and a suitable diagonalization procedure (recall that all measures involved have locally uniformly bounded variation), we can find sequences $r_m \todown 0$, $j_m \to \infty$, $h_m \to \infty$ (as $m \to 
\infty$) such that for
\[
u_m := u^{h_m}_{j_m} \quad \text{and} \quad \gamma_m \coloneqq r_m^{-d}T_\#^{(x_0,r_m)} [u_m\, \Lcal^d]\]
it holds that
\begin{enumerate}
\item $\displaystyle \gamma_m \toweakstar \frac{\di \mu}{\di \Lcal^d}(x_0) \, \Lcal^d$;
\item $\displaystyle \frac{\di \lambda}{\di \Lcal^{d}}(x_0) \ge \lim_{m \to \infty} \int_Q f \biggl( x_0 +r_m y, \frac{\di \gamma_m}{\di \Lcal^d}(y) \biggr) \dd y$.
\end{enumerate}

By Proposition~\ref{lem:nonh} and the first property,
\begin{equation*}
\begin{split}
& \gamma_m- \frac{\di \mu}{\di \Lcal^d}(x_0)\, \Lcal^d \toweakstar 0\quad \text{in $\M(Q;\R^N)$}, \quad \text{and}\\
& \A^k \bigg(\gamma_m- \frac{\di \mu}{\di \Lcal^d}(x_0)\, \Lcal^d\bigg)  \to 0 \quad \text{in $\Wrm^{-k,q}(Q;\R^N)$}.
\end{split}
\end{equation*}

We are now in a position to apply Lemma~\ref{lem:boundary values} to the sequence $\gamma_m$ and the Lipschitz function $f(x_0,\frarg)$, whence there exists a sequence $(z_m)\subset \Crm^\infty_\per(Q;\R^N)$ such that
	\[
	\A z_m = 0, \qquad \int_Q z_m = 0, \qquad z_m \toweakstar 0 \quad \text{in $\Mcal(Q;\R^N)$},
	\]
	and  %the sequence $\bigl( \frac{\di \mu}{\di \Lcal^d}(x_0) \, \Lcal^d + z_r \bigr)$  
	\[
	\liminf_{m \to \infty} \int_Q  f\bigg(x_0,\frac{\dd \gamma_m}{\dd \Lcal^d}(y)\bigg) \dd y \ge 
	\liminf_{m \to \infty} \int_Q  f\bigg(x_0,\frac{\dd \mu}{\dd \Lcal^d}(x_0) + z_m(y)\bigg) \dd y.
	\]
	
Hence, using the second property above and our assumption~\eqref{eq:modulus} on the integrand, we have
\begin{align}
\frac{\di \lambda}{\di \Lcal^d}(x_0)
& \ge \lim_{m \to \infty} \int_Q f \biggl( x_0 +r_m y, \frac{\di \gamma_m}{\di \Lcal^d}(y) \biggr) \dd y  \notag\\
& = \lim_{m \to \infty} \int_Q f \biggl( x_0, \frac{\di \gamma_m}{\di \Lcal^d}(y) \biggr) \dd y  \notag\\
& \ge  \liminf_{m \to \infty} \; \int_Q f\bigg(x_0,\frac{\di \mu}{\di \Lcal^d}(x_0) + z_m(y)\bigg)   \notag\\
& \ge f\bigg(x_0,\frac{\di \mu}{\di \Lcal^d}(x_0)\bigg). \label{eq:absabs}
\end{align}
This proves~\eqref{hybrid1}.\end{proof}

%\begin{remark} The proof of~\eqref{hybrid1} does not use at any moment that $f^\infty$ exists in $\Omega \times \text{span}\, \Lambda_{\A}$. Hence, by a similar argument as in~\eqref{mio}, is easy to check that Theorem~\ref{thm:lsc_partial} is an immediate consequence of~\eqref{hybrid1}.
%\end{remark}

\begin{remark}\label{rem:absabs}%\mnote{A: Added} 
If the assumption that $f(x,\frarg)$ is $\A^k$-quasiconvex is dropped, %and one additionally assumes that $f$ is continuous {\color{red}--- this is needed to ensure that $Q_{\A^k}f$ is $\A^k$-quasiconvex (compare with Lemma~\ref{lem:qce_property})}, 
one can still show that
 \[
\frac{\di \lambda}{\di \Lcal^d}(x_0) \ge Q_{\Acal^k}f\bigg(x_0,\frac{\di \mu}{\di \Lcal^d}(x_0)\bigg). 
 \]
Indeed, the $\A^k$-quasiconvexity of $f(x,\frarg)$ has only been used in the last inequality of~\eqref{eq:absabs} where one can first use the inequality $f(x,\frarg) \ge Q_{\Acal^k}f(x,\frarg)$ to get
\[
\int_Q f\bigg(x_0,\frac{\di \mu}{\di \Lcal^d}(x_0) + z_r(y)\bigg) \ge
\int_Q Q_{\A^k}f\bigg(x_0,\frac{\di \mu}{\di \Lcal^d}(x_0) + z_r(y)\bigg).
\]
which follows by the very definition of  $Q_{\Acal^k}f(x,\frarg)$. 
\end{remark}

\begin{proof}[Proof of~\eqref{hybrid2}] Passing to a subsequence if necessary, we may assume that
\[
\mu_j \toY \nu \quad \text{for some $\nu \in \Y_{\A}(\Omega;\R^N)$}.
\]
For each $j \in \N$ set $\nu_j := \delta[\mu_j] \in \Y(\Omega;\R^N)$, the elementary Young measure corresponding to $\mu_j$, so that 
$\nu_j \ysc \nu$ in $\Y(\Omega;\R^N)$. Define the functional 
\[
\Fcal_\#[\sigma;B] \coloneqq \int_B \dprb{f(x,\frarg),\sigma_x} \dd x + \int_{\cl{B}} \dprb{f_\#(x,\frarg),\sigma_x^\infty} \dd \lambda_{\nu}(x), \qquad \sigma \in \Y(\Omega;\R^N),
\]
where $B \subset \Omega$ is an open set. Observe that, as a functional defined on $\Y(\Omega;\R^N)$, $\Fcal_\#$ is sequentially weakly* lower semicontinuous (see Corollary~\ref{cor:representation}). We use Assumption~(ii), which is equivalent to
\[
\text{$f^\#(x,\frarg) \equiv f_\#(x,\frarg)$ \quad on $V_{\A}$},
\]
and the fact, proved in~\eqref{eq: paso}, that
\[
\text{$\supp \nu_x^\infty \subset V_{\A}$ \; for $\lambda^s_\nu$-a.e. $x \in \Omega$,}
\]
%\mnote{{\color{red}A: The highlighted inequality is not right! $\nu^\infty_x \subset V_{\A}$ only for $\lambda^s_\nu$-a.e.\\
%I still believe one should be able to show that $\nu^\infty_x \subset V_{\A}$ for $\lambda_\nu$-a.e.}\\
%G: Agree, but as it is written now is ok, right?}
to get (recall $f \geq 0$) %\mnote{F: I cleaned up here}
%satisfying~\eqref{ass1}, the Young measure representation $\ddpr{ f , \nu_n } \to \ddpr{ f , \nu }$ holds, cf.~Corollary~\ref{cor:representation}. Consequently, also using $f \ge 0$,
\begin{align}
\liminf_{j \to \infty}\Fcal^\#[\mu_j;B]
&\ge \liminf_{j \to \infty}\Fcal_\#[\nu_j;B] \notag\\
&\ge \Fcal_\#[\nu;B] \notag\\
& \ge \int_{B} \left(\dprb{  f(x,\frarg) , \nu_x } + \dprb{ f_\#(x,\frarg),\nu^\infty_x } \frac{\di \lambda_\nu}{\di \Lcal^d}(x) \right)\dd x \notag\\
&\qquad +\int_{B} \dprb{ f_\#(x,\frarg),\nu^\infty_x} \dd\lambda^s_\nu(x) \notag\\
&\ge %\int_{B} \langle  f(x,\frarg) , \nu_x \rangle \dd x +
\int_{B} \dprb{ f^\#(x,\frarg),\nu^\infty_x} \dd\lambda^s_\nu(x). \label{eq:instead}
\end{align}
Recall that, for every $x \in \Omega$, the function $f(x,\frarg) $ is $\A^k$-quasiconvex and hence the function $f^\#(x,\frarg) $ is  $\Lambda_{\A}$-convex and positively $1$-homogeneous. An application of the Jensen-type inequality from Proposition%s~\ref{prop: jensen regular} and
~\ref{prop: jensen singular} to the last line yields
\begin{align*}
\liminf_{j \to \infty}\Fcal^\#[\mu_j;B]  & \ge %\int_{B} \langle f(x,\frarg) , \nu_x \rangle \dd x +
 \int_{B} f^\#\big(x, \dprb{ \id , \nu_x^\infty }\big) \dd\lambda_\nu^s(x).%\\
%&=  \Fcal[\,[\nu]\,] = \Fcal[\mu].
\end{align*}
Thus, also taking into account $\abs{\mu^s} = \abs{\dpr{\id,\nu^\infty_x}} \, \lambda_\nu^s$ and $f^\#(x,\dpr{\id,\nu^\infty_x}) = f^\#(x,0) = 0$ for $\lambda_\nu^\ast$-a.e.\ $x \in \Omega$, where $\lambda_\nu^\ast$ is the singular part of $\lambda_\nu^s$ with respect to $\abs{\mu^s}$, we get
\begin{align*}
\lambda(B) \ge \int_{B}  f^\#\bigg(x,\frac{\dd \mu^s}{\dd |\mu^s|}(x) \bigg) \dd|\mu^s|(x),
\end{align*}
for all open sets $B \subset \Omega$ with $\lambda_\nu^s(\partial B) = 0$. Therefore, by the Besicovitch differentiation theorem and using the continuity of $f$ (see~\eqref{eq:modulus}) in its first argument we get
%Since $\lambda$ is a Borel measure, it follows by a density argument that
\[
\frac{\dd \lambda}{\dd |\mu^s|}(x_0) \ge f^\#\bigg(x_0,\frac{\dd \mu^s}{\dd |\mu^s|}(x_0)\bigg) \qquad \text{for $|\mu^s|$-a.e. $x_0 \in \Omega$}.
\]
This proves~\eqref{hybrid2}.
\end{proof}

\begin{remark}[recession functions]\label{rem:crucial} The only part of the proof where we use the existence of $f^\infty(x,A)$, for $x \in \Omega$ and $A \in V_{\A}$, is in showing that 
\begin{align*}
\Fcal_\#[\nu;B]  & \ge \int_{B} \left(\dprb{  f(x,\frarg) , \nu_x } + \dprb{ f_\#(x,\frarg),\nu^\infty_x } \frac{\di \lambda_\nu}{\di \Lcal^d}(x) \right)\dd x \\
&\qquad +\int_{B} \dprb{ f^\#(x,\frarg),\nu^\infty_x} \dd\lambda^s_\nu(x)
\end{align*}
{The need of such an estimate comes from the fact that, in general, it is unknown whether $f_\#$ is a $\Lambda_{\A}$-convex function}.
\end{remark}

%\begin{remark}
%This proof of Theorem~\ref{thm:lsc} does not make use of the $x$-uniform Lipschitz assumption on $f(x,\frarg)$ nor the modulus of continuity~\eqref{eq:modulus}.
%%However, similar properties to these ones are implicit in the requirement $f \in \E(\Omega;\R^N)$.
%\end{remark}

%\mnote{F: Deleted second recession function remark, seemed obvious}

%\begin{remark}[recession functions II] In light of the characterization from Theorem~\ref{thm:guido filip} (which states that the singular part of an $\A$-free measure essentially belongs to the wave cone $\Lambda_{\A}$) and that by assumption all three recession functions $f^\#, f^\infty$, and $f_\#$ exist and coincide on $V_{\A}$, 
%the same conclusions as in Theorem~\ref{thm:lsc} hold for the functionals $\Fcal^\infty$ and $\Fcal_\#$ which are defined in the same way as $\Fcal$, but with $f^\#$ replaced by $f^\infty$ and $f_\#$.
%\end{remark}

\begin{remark}\label{rem:lb} 
If we drop the assumption that $f(x,\frarg)$ is $\A^k$-quasiconvex for every $x \in \Omega$, we can still show that
\begin{align*}
\int_\Omega &  Q_{\A^k}f\bigg(x,\frac{\dd \mu}{\dd \Lcal^d}(x)\bigg) \dd x + \int_\Omega (Q_{\A^k}f)^\#\bigg(x,\frac{\dd \mu^s}{\dd |\mu^s|}(x)\bigg) \dd |\mu^s|(x) \le  \liminf_{j \to \infty} \Fcal[\mu_j]
\end{align*}
%that
%\[
%\cl\Gcal[\mu]  \le \liminf_{j \to \infty} \Fcal[\mu_j]
%\]
for every sequence $\mu_j \toweakstar \mu$ in $\M(\Omega;\R^N)$ such that $\A \mu_j \to 0$ in $\Wrm^{-k,q}(\Omega)$.
The proof of this fact follows directly from Remark~\ref{rem:absabs}, the last line of~\eqref{eq:instead} together with the continuity of $f$ in its first argument (for the Besicovitch differentiation arguments), %(where the $\A^k$-quasiconvexity assumption is not necessary) 
and Corollary~\ref{cor:principle}. Observe that one does not require the existence of $(Q_{\A^k} f)^\infty$ in $\Omega \times \text{span}\,\Lambda_{\A}$.
\end{remark}

%\begin{remark} The proof of~\eqref{hybrid1} does not use at any moment that $f^\infty$ exists in $\Omega \times \text{span}\, \Lambda_{\A}$. Hence, by a similar argument as in~\eqref{mio}, is easy to check that Theorem~\ref{thm:lsc_partial} is an immediate consequence of~\eqref{hybrid1}.
%\end{remark}

\begin{proof}[Proof of Theorem~~\ref{thm:lsc_partial}]
Note that in the proof of~\eqref{hybrid1} we did not use that $f^\infty$ exists in $\Omega \times \text{span}\, \Lambda_{\A}$. By the very same argument as in~\eqref{mio}, is easy to check that Theorem~\ref{thm:lsc_partial} is an immediate consequence of~\eqref{hybrid1}.

\end{proof}

\section{Proof of Theorems~\ref{thm:relax} and~\ref{thm:relax2}} \label{sc:relax_proof}

We use standard machinery to show the relaxation theorems. Recall that, for Theorems~\ref{thm:relax} and~\ref{thm:relax2}, we assume that $\A$ is a homogeneous partial differential operator.

\subsection{Proof of Theorem~\ref{thm:relax}}
 %We conclude by observing that the proposed upper bound is weakly* lower semicontinuous as a corollary of Theorem~\ref{thm:lsc}.
	\proofstep{Step~1. The lower bound.}
	The lower bound $\overline{\Gcal} \geq \Gcal_*$, where
	\[
	  \Gcal_*[\mu] :=
	\int_\Omega Q_{\A}f\bigg(x,\frac{\dd \mu}{\dd \Lcal^d}(x)\bigg) \dd x + \int_\Omega (Q_{\A}f)^\#\bigg(x,\frac{\dd \mu^s}{\dd |\mu^s|}(x)\bigg) \dd |\mu^s|(x),
	\]
	is a direct consequence of Remark~\ref{rem:lb} and the fact that $\A$ is a homogeneous partial differential operator ($\A = \A^k$).

	We divide the proof of the upper bound in Theorem~\ref{thm:relax} into several steps. First, we prove that any $\A$-free measure may be area-strictly approximated by asymptotically $\A$-free absolutely continuous measures. Next, we prove the upper bound on absolutely continuous measures, from which the general upper bound follows by  approximation.

	\proofstep{Step~2. An area-strictly converging recovery sequence.}
	Let $\mu \in \Mcal(\Omega;\R^N) \cap \ker \A$. We will show that there exists a sequence $(u_j) \subset \Lrm^1(\Omega;\R^N)$ for which 
	\begin{gather*}
	\text{$u_j \Lcal^d \toweakstar  \mu$ \; in $\Mcal( \Omega;\R^N)$, \quad $\area{u_j \Lcal^d}(\Omega) \to \area{\mu}( \Omega)$,} \\
	\text{and \quad $\A u_j \to 0$ \; in $\Wrm^{-k,q}(\Omega)$}.
	\end{gather*}
	Let $\{\varphi_i\}_{i\in \mathbb \N} \subset \Crm^\infty_c(\Omega)$ be a locally finite 
	partition of unity of $\Omega$. Set
	\[
	\mu_{(i)} := \mu\varphi_i \in \Mcal(\Omega;\R^N),
	\] 
	and 
	\[
	\mu^a_{(i)} := \mu^a\varphi_i  \qquad \mu^s_{(i)} := \mu^s\varphi_i.	
	\]
	where, as usual,
	\[
	\mu^a=\frac{\di \mu} {\di \Lcal ^d} \Lcal ^d\qquad\mbox{and}\qquad\mu^s=\mu-\mu^a. 
	\]
%	Note that
%	\[
%	\sum_{i=1}^j \mu_{(i)}\toweakstar \mu \qquad\text{as \(j\to \infty\)}.
%	\]
	Note that, with a slight abuse of notation,
	\begin{equation*}\label{L1}
	\biggl\|\sum_{i=1}^j \mu^a_{(i)}-\mu^a\biggr\|_{\Lrm^1(\Omega)}\to 0\quad\text{as \(j\to \infty\)}.
	\end{equation*}
	Furthermore, for fixed \(i\),
	\begin{equation}\label{pd}
	(\mu_{(i)}* \rho_{\eps})\Lcal^d \toweakstar \mu_{(i)},\qquad |\mu_{(i)}* \rho_{\eps}|(\Omega)\le |\mu_{(i)}|(\Omega)= \int_{\Omega} \varphi_i \dd  |\mu|,
	\end{equation}
	and
	\[
	\mu^a_{(i)} * \rho_{\eps}\to   \mu_{(i)}^a \quad \text{in $\Lrm^1(\Omega)$} 
	\qquad
	\text{as \(\eps \to 0\).}
	\]
%	\mnote{A: I do not think this is true, I think we have to use the proof we had before}{\color{red} In the same way, since \(\A\mu=0\), 
%	\[
%	\biggl\|\sum_{i=1}^j \A\mu_{(i)}\biggr\|_{\Wrm^{-k,q}(\Omega)}=\biggr\|\sum_{i=1}^j \A\mu_{(i)}-\A\mu\biggr\|_{\Wrm^{-k,q}(\Omega)}\to 0\qquad\text{as \(j\to \infty\)}.
%	\]}
Moreover,
	\[
	\A(\mu_{(i)} * \rho_{\eps})\to   \A\mu_{(i)} \quad \text{in $\Wrm^{-k,q}(\Omega)$ \quad as \(\eps \to 0\)}.
	\]
	
Fix $j \in \Nbb$. From~\eqref{pd} and the convergence above we might find a sequence $\eps_i(j) \todown 0$ such that the measures $\mu_{i,j} \coloneqq \mu_{(i)} \ast \rho_{\eps_i(j)}$ and $\mu_{i,j}^a \coloneqq \mu^a_{(i)} \ast \rho_{\eps_i(j)}$ verify
\begin{gather*}
d(\mu_{i,j} \Lcal^d,\mu_{(i)}) \le \frac{1}{2^i j}, \\ 
\|\mu^a_{i,j} - \mu^a_{(i)}\|_{\Lrm^1(\Omega)} \le \frac{1}{2^i j}, \\
\|\mu_{i,j} - \mu_{(i)}\|_{\Wrm^{-k,q}(\Omega)} \le \frac{1}{2^i j},
\end{gather*}
where $d$ is the metric inducing the weak* convergence on a suitable subset of $\M(\Omega;\R^N)$ (the existence of the metric $d$ is a standard result for the duals of separable Banach spaces). 
Define the integrable functions (identifying $\mu^a_{i,j}$ with its density)
	\[
	u_j \coloneqq \sum_{i=1}^\infty\mu_{i,j}, \qquad u^a_j \coloneqq \sum_{i=1}^\infty\mu^a_{i,j}.
	\]
	We get
	\[
	d(u_j \Lcal^d,\mu) \le \sum_{i = 1}^\infty d(\mu_{i,j} \Lcal^d,\mu_{(i)}) \le \sum_{i = 1}^\infty \frac{1}{2^i j} = \frac{1}{j},
	\]
	and in a similar way
	\begin{gather*}
	\|u^a_{j} - \mu^a\|_{\Lrm^1(\Omega)} \le \frac{1}{j},\\
	\|\A u_j \|_{\Wrm^{-k,q}(\Omega)} \le \frac{1}{j},
	\end{gather*}
  where we use that $\mu$ is $\A$-free in the second inequality.
	Observe that~\eqref{pd} and the fact that \(\{\varphi_i\}_{i\in \mathbb \N}\) is a partition of unity imply 
	\begin{equation}\label{metric1}
	\int_{\Omega}  |u_j|\dd x\le \sum_{i=1}^\infty \int_{\Omega} \varphi_i\dd  |\mu|\le |\mu|(\Omega).
	\end{equation}
Therefore $\|u_j\|_{\Lrm^1(\Omega)}$ is uniformly bounded and hence
 \begin{gather}
 u_j \Lcal^d \toweakstar \mu \quad \text{in $\M(\Omega;\R^N)$}, \label{L11}\\
 \bigl\|u_j^a-\mu^a\bigr\|_{\Lrm^1(\Omega)}\to 0, \label{L12}\\
 \|\A u_j\|_{\Wrm^{-k,q}(\Omega)} \to 0\label{L11.5}
 \end{gather}
 as $j \to \infty$. Moreover, the weak* lower semicontinuity of the total variation and~\eqref{metric1} imply the strict convergence
 \begin{equation}\label{strict1}
 |u_j \Lcal^d|(\Omega) \to |\mu|(\Omega).
 \end{equation}
 
%    By a diagonal argument, we can thus find \(\eps_i\to 0\) such that, setting \(\mu_i\coloneqq\mu_{(i)} * \rho_{\eps_i}\),  \(\mu^a_i\coloneqq\mu^a_{(i)} * \rho_{\eps_i}\), and
%	\[
%	u_j =\sum_{i=1}^j\mu_i + \sum_{i = j+1}^\infty \mu_{(i)} \ast \rho_{\eps_{j}} \in \Lrm^1(\Omega;\R^N), \qquad u^a_j=\sum_{i=1}^j\mu^a_i + \sum_{i = j + 1}^\infty \mu_{(i)}^a  \in \Lrm^1(\Omega;\R^N),
%	\]
%	we get 
%	\begin{equation}\label{L11}
%	u_j \Lcal^d \toweakstar \mu,%\qquad \|\A u_j\|_{\Wrm^{-k,q}(\Omega)}\to 0, 
%	\end{equation}
%	and
%	\begin{equation}\label{L12}
%	\bigl\|u_j^a-\mu^a\bigr\|_{\Lrm^1(\Omega)}\to 0 \quad \text{as \(j \to 0\)}.
%	\end{equation}
%	Moreover, since $\A \mu = 0$, the $\eps_i$'s can be chosen so that
%	\begin{equation}\label{L11.5}
%	\begin{split}
%	\|\A u_j\|_{\Wrm^{-k,q}(\Omega)} & = \|\A u_j - \A \mu\|_{\Wrm^{-k,q}(\Omega)} \\
%	& \le \sum_{i=1}^j \|\A (\mu_i - \mu_{(i)})\|_{\Wrm^{-k,q}(\Omega)} \to 0,
%	\end{split}
%	\end{equation}
%	as $j \to \infty$.
	%Note that \(u_j^a\in \Crm_c^\infty(\Omega;\R^N)\), hence, t
	Thanks to~\eqref{L11} and~\eqref{L11.5}, to conclude the proof of the claim it suffices to show that 
	\begin{equation}\label{conc}
	 \lim_{j\to \infty} \area{u_j \Lcal^d}(\Omega)= \area{\mu}(\Omega).
	\end{equation} 
%	For this, observe that~\eqref{pd} and that fact that \(\{\varphi_i\}_{i\in \mathbb \N}\) is a partition of unity imply that 
%	\[
%	\int_{\Omega}  |u_j|\dd x\le \sum_{i=1}^\infty \int_{\Omega} \varphi_i\dd  |\mu|\le |\mu|(\Omega).
%	\]	
Exploiting~\eqref{L11},~\eqref{L12},~\eqref{strict1}, we get
	\begin{equation} \label{lasagne}
	\int_{\Omega}  |u_j-u_j^a|\dd x \to |\mu^s|(\Omega)\quad\text{as \(j\to \infty\)}.
	\end{equation}
	By the inequality \(\sqrt{1+|z|^2}\le \sqrt{1+|z-w|^2}+|w|\) (for \(z,w\in \mathbb R^N\)), we get
	\[
	\area{u_j \Lcal^d}(\Omega)\le \area{u^a_j \Lcal^d}(\Omega)+\int_{\Omega}  |u_j-u_j^a|\dd x.
	\]
 Hence,  again by~\eqref{L12} and~\eqref{lasagne}
	\begin{equation}\label{pizza}
	\limsup_{j\to \infty} \, \area{u_j \Lcal^d}(\Omega)\le \area{\mu}(\Omega).
	\end{equation}
On the other hand, by the weak* convergence \(u_j \Lcal^d\toweakstar \mu\) and the convexity of \(z\mapsto \sqrt{1+|z|^2}\) ,
	 \[
	 \liminf_{j\to \infty} \, \area{u_j \Lcal^d}(\Omega)\ge \area{\mu}(\Omega).
	 \]
	 Thus, together with~\eqref{pizza},~\eqref{conc} follows,  concluding the proof of the claim.

\proofstep{Step~3.a. Upper bound on absolutely continuous fields.} Let us now turn to the derivation of the upper bound for $\overline \Gcal[u] = \overline \Gcal[u \Lcal^d]$ where $u \in \Lrm^1(\Omega;\R^N) \cap \ker \A$. For now let us assume additionally the following strengthening of~\eqref{eq:modulus}:
\begin{equation} \label{eq:modulus_strong}
f(x,A) - f(y,A)\le \omega(|x - y|)(1 + f(y,A))  \qquad
\text{for all $x,y \in {\Omega}$, $A \in \R^N$.}
\end{equation}
It holds that $Q_{\A^k}f(x,\frarg)$ is still uniformly Lipschitz in the second variable and
\begin{equation}\label{eq:modulusQ}
Q_{\Acal}f(x,A) \le Q_{\Acal}f(y,A)  + \omega(|x - y|)(1 + |A|) 
\end{equation}
for every $x, y \in \Omega$ and $A \in \R^N$ with a new modulus of continuity (still denoted by $\omega$), which incorporates another multiplicative constant in comparison to the original $\omega$. Indeed, fix $x,y \in \Omega$, $\eps > 0$, and $A \in \R^N$. Let $w \in \Crm^\infty_\per(Q;\R^N) \cap \ker \A$ be a function with zero mean in $Q$ such that
\[
\int_Q f(y,A + w(z)) \dd z \le Q_{\A}f(y,A) + \eps.  
\]
By assumption, we get
\begin{align*}
\int_Q  f(x,A + w(z))\dd z & \le  \int_Q f(y,A + w(z)) \dd z  \\ 
& \qquad + \omega(|x - y|)\bigg(1 + \int_Q f(y,A + w(z))\dd z\bigg) \\
& \le Q_{\A}f(y,A) + \eps \\ 
& \qquad + \omega(|x - y|)\big(1 + Q_{\A}f(y,A) + \eps\big).
\end{align*}
Thus,
\[
Q_\A f(x,A) \le Q_\A f(y,A) + \eps + \omega(|x - y|)(1 + Q_{\A}f(y,A) + \eps).
\]
The linear growth at infinity of $f$, which is inherited by $Q_\A f$, gives
\[
Q_\A f(x,A) \le Q_\A f(y,A) + \omega(|x - y|)(1 + M(1 + |A|)) + \eps(1 + \omega(|x-y|)).
\]
We may now let $\eps \todown 0$ in the previous inequality to obtain
\[
Q_\A f(x,A) \le Q_\A f(y,A) + \omega(|x - y|)(M+1)(1 + |A|).
\]
This proves~\eqref{eq:modulusQ} provided that~\eqref{eq:modulus_strong} holds.

Fix $m \in \Nbb$ and consider a partition of $\R^d$ of cubes of side length $1/m$. Let  $\{Q_i^m\}_{i = 1}^{L(m)}$ be the maximal collection of those cubes (with centers $\{x_i^m\}_{i =1}^{L(m)}$) that are compactly contained in $\Omega$. %and note that, by the assumption $\Lcal^d(\partial \Omega)=0$, 
We have 
\[
\Lcal^d(\Omega)=\sum_{i=1}^{L(m)}\Lcal^d(Q_i^h)+\SmallO_m(1),
\]
%\mnote{A: added notation\\
%G: Changed, in this way is more flexible }
where $\SmallO_m(1)\to 0$ as \(m\to \infty\).

%
%the mesh of cubes $\{ Q_i^h\}_{i = 1}^{l(h)}$ defined by the grid 
%\[\{h^{-1}\Zbb + Q\} \cap \Omega.\]
We may approximate $u$ strongly in $\Lrm^1$ by functions $z^m \in \Lrm^1(\Omega;\R^N)$ that are piecewise constant on the mesh $\{Q_i^m\}_{i = 1}^{L(m)}$ (as $m \to \infty$). More specifically, we may find functions $z^m \in \Lrm^1(\Omega;\R^N)$ 
%with $\supp \varphi \Subset \Omega$ 
such that $z^m=0$ on $\Omega\setminus \bigcup_i Q_i^m$,
\begin{equation}\label{male1}
z^m = z^m_i \in \R^N \quad \text{on $Q_i^m$ \qquad and \qquad} \|u - z^m\|_{\Lrm^1(\Omega)} = \SmallO_m(1).
\end{equation}
Additionally, for every $m \in \Nbb$, we may find functions $w^m_i \in \Crm_\per^\infty(Q;\R^N) \cap \ker\A$ with the properties 
\begin{equation}\label{eq:QC}
\int_Q f(x_i^m,z^m_i + w^m_i(y)) \dd y \le Q_{\A}f(x_i^m,z^m_i) + \frac{1}{m},  \qquad
\int_Q w_i^m \dd y= 0.
\end{equation}
%Fix also $\varphi \in \Crm_c(Q;[0,1])$. With these properties in mind let us consider the functions
%\[
%v_{n}^{h} \coloneqq \sum_{i = 1}^{l(h)}\mathds 1_{Q_i^h}(x) \varphi(h(x - x_i)) \cdot w^h_i(nhx) \quad m \in \Nbb, x \in \Omega.
%\]
Fix $m \in \Nbb$ and let $\phi_m \in \Crm^\infty_c(Q;[0,1])$ be a function such that
\begin{equation}\label{male2}
\sum_{i = 1}^{L(m)} \|1 - \varphi_m\|_{\Lrm^1(Q)} \|w_i^m\|_{\Lrm^1(Q)}\leq \frac{1}{m},
\end{equation}
We define the functions
\[
v_{j}^{m} \coloneqq \sum_{i = 1}^{L(m)}\mathds \varphi_m(m(x - x_i^m)) \cdot w^m_i(jm(x - x_i^m)) \qquad x \in \Omega, \; j \in \Nbb.
\]
%We can apply Lemma~\ref{lem:homogenization} to each summand of $v_j^m$ to see that
%It is not hard to see, using the properties of two-scale convergence from Step~2,\mnote{F: Or use an analogue of the ``averaging principle'' as in my papers \\
%A: I think Step~2 is fairly straightforward, but I am open to change it. What could be a good idea is to add a proposition in the Young measure section.} 
%that the sequence $(v_n^{h})$ generates the generalized Young measure
%$\nu$ with $\nu_x = \overline{\delta_g(x,y)}_x$ and $\lambda_\nu = 0$, where 
%\[
%g: \Omega \times Q \to \R^N : (x,y) \mapsto \sum_{i = 1}^{l(h)}\varphi(h(x - x_i)) \cdot w_i^h(y).
%\]
%In particular, 
%\[ \dpr{\id,\nu_x} = \sum_{i = 1}^{l(h)}\varphi(h(x - x_i) )\cdot\int_{Q} w_i^h(y) \dd y  = 0 \qquad \text{whenever $x \in Q_i^h$}.
%\]
%%{\color{red}F: Asymptotically as $m \to \infty$? Also, both $w^m_i$ and the argument $mjx$ depend on $h$, so I do not think this is totally obvious.}\\
%%{\color{blue} A: They depend on $h$, but $h$ is fixed when taking the limit.}
%%\[
%%\wslim_{j \to \infty} \, v_j^{m} \Lcal^d = 0 \quad \text{in $\M(\Omega;\R^N)$}.
%%\] 
%%\mnote{F: Or use an analogue of the ``averaging principle'' as in my papers \\
%%A: I think it is fairly straightforward with the inclusion of Lemma~\ref{lem:homogenization}}
{By Lemma~\ref{lem:homogenization}, the sequence $(v^m_j)_j$ generates the Young measure 
\[\nu^m = (\nu_x^m,0,\frarg) \in \Y(\Omega;\R^N),\]
 where for each $x \in \Omega$, $\nu_x^m$ is the probability measure defined by duality trough
 \[
 \dprb{h,\nu^m_x} \coloneqq \sum_{i = 1}^{L(m)} \mathds 1_{Q_i^m}(x) \int_Q h(\phi_m(m(x - x_i^m)) \cdot w_i^m(y)) \dd y, 
 \]
 on functions $h \in \Crm(\R^N)$ with linear growth.
% 
%  $g^m: \Omega \times Q \to \R^N$ is given by 
%\[
%(x,y) \; \mapsto \; \sum_{i = 1}^{L(m)}\varphi_m(m(x - x_i^m)) \cdot w_i^m(y)
%\]
%and the averaging is to be understood with respect to $y$.%\mnote{F: Write out? (OK with me as it is)}}

The central point of this construction is that $w_i^m$ has \emph{zero mean value}, that is, $\int_Q w_i^m \dd y = 0$, whereby it follows that
\begin{equation}\label{eq:avg}
v_j^m \Lcal^d \quad \toweakstar  \quad  \sum_{i = 1}^{L(m)} \int_Q \phi_m(m(x - x_i^m)) \cdot w_i^m(y) \dd y \;  =  \; 0 \quad \text{in $\M(\Omega;\R^N)$}
\end{equation}
as $j \to \infty$.
%By the compact embedding $\Lrm^1(\Omega;\R^N) \cembed \Wrm^{-1,q}(\Omega;\R^N)$, and 
Recall that by construction, $\A w_i^m = 0$ on $Q$, Hence, using that $\A$ is homogeneous we get
\[
  \A [w_i^m(jm \,(\frarg - x_i^m))] = 0 \quad \text{in the sense of distributions on $Q_i^m$}.
\]
Thus, for some coefficients $c_{\alpha,\beta} \in \Nbb$, using the short-hand notation $\psi_m(y) \coloneqq \phi_m(my)$ yields
\begin{align*}
	\A v^m_j & = \sum_{i = 1}^{L(m)}  \bigg(\A[w_i^m(jm \,(\frarg - x_i^m))] \psi_m(\frarg - x_i^m)\\ 
	& \qquad + \sum_{\substack{|\alpha| = k,\\1 \le |\beta| \le k}} c_{\alpha \beta} A_\alpha\partial^{\alpha - \beta} [w_i^m(jm\,(\frarg - x_i^m))] \partial^\beta [\psi_m(\frarg - x_i^m)]\bigg) \\
	& = \sum_{\substack{|\alpha| = k,\\1 \le |\beta| \le k}} \bigg(  \sum_{i = 1}^{L(m)}c_{\alpha \beta} \partial^{\alpha - \beta} [w_i^m(jm\,(\frarg- x_i^m))] \partial^\beta [\psi_m(\frarg - x_i^m)]\bigg)
	\end{align*}
in the sense of distributions on $\Omega$. Applying Lemma \ref{lem:homogenization} to the sequence $(w_i^m(jm (\frarg- x_i^m)))_j$ on each cube $Q_i^m$ we get
\begin{align*}
\sum_{i= 1}^{L(m)} \mathds 1_{Q_i^m}\cdot w_i^m(jm(\frarg- x_i^m)) & \toweak \sum_{i= 1}^{L(m)} \mathds 1_{Q_i^m}\cdot\dashint_{Q_i^m} w_i^m(m(y - x_i^m)) \dd y \\ 
& = \sum_{i= 1}^{L(m)} \mathds 1_{Q_i^m}\cdot \int_{Q} w_i^m(y) \dd y \\
&= 0. %\quad \text{in $\Lrm^1(\Omega;\R^N)$}.
\end{align*}
Hence,~\eqref{eq:avg} and the compact embedding $\Lrm^1(\Omega;\R^N) \cembed \Wrm^{-1,q}(\Omega;\R^N)$ yield
\begin{align*}
\A v_j^m  = \sum_{\substack{|\alpha| = k,\\1 \le |\beta| \le k}} \bigg(\sum_{i = 1}^{L(m)}c_{\alpha \beta} \partial^{\alpha - \beta} [w_i^m(jm(\frarg-x_i^m))]\partial^\beta[\psi_m(\frarg - x_i^m)]\bigg)  \to 0 %\\
%& \qquad \qquad \text{strongly in $\Wrm^{-k,q}(\Omega;\R^N)$}, \quad \text{as $j \to \infty$};
\end{align*}
strongly in $\Wrm^{-k,q}(\Omega;\R^N)$, as $j \to \infty$.\\

%\[
%\A v_j^m \to 0 \quad \text{strongly in $\Wrm^{-k,q}(\Omega;\R^N)$}, \quad \text{as $j \to \infty$}.
%\]
%Up to a possible different choice of $\phi$, the previous construction can be made for every $m \in \N$.

For later use we record: 

\begin{remark}\label{rem:project} By construction, for every $m,j \in \Nbb$, the function $v_j^m$ is compactly supported in $\Omega$. Up to re-scaling, we may thus assume without loss of generality that $\Omega \subset Q$ and subsequently make use of Lemma~\ref{lem:boundary values} on the $j$-indexed sequence $(\tilde v_j^m)$ with $m$ fixed, where $\tilde v_j^m$ is the zero extension of $v_j^m$ to $Q$, to find another sequence $(V^m_j) \subset \Lrm^1(\Omega;\R^N) \cap \ker \A$ generating the same Young measure $\nu^m$ (as $j \to \infty$).

\end{remark}

In the next calculation we use the Lipschitz continuity of $Q_{\A}f(x,\frarg)$ in the second variable, equation~\eqref{male1} and the fact that the sequence $(v_j^m)$ generates the Young measure $\nu^m$ as $j\to\infty$, to get
\begin{align}
\lim_{j \to \infty} \Gcal[u + v_j^m] & = \lim_{j \to \infty} \Gcal[z^m + v_j^m] + \SmallO_m(1) \notag\\
%& = \int_\Omega \int_Q f(x,g(x,y)) \dd y + \BigO(h)\\
%& = \ddprb{f,\nu^m} + \SmallO_m(1) \notag\\
& = \sum_{i = 1}^{L(m)} \int_{Q_i^m} \int_Q f \bigl( x,z_i^m + \varphi_m(m(x-x_i^m))\cdot w_i^m(y) \bigr) \dd y \dd x+ \SmallO_m(1). \label{eq:limsup}
\end{align}
By a change of variables we can estimate every double integral times $m^d = \Lcal^d(Q_i^m)^{-1}$ on the last line on each cube of the mesh:
\begin{align}
\dashint_{Q_i^m} \int_Q & f \bigl( x,z_i^m + \varphi_m(m(x-x_i^m)) \cdot w_i^m(y) \bigr) \dd y \dd x  \notag\\  
&= \int_{Q} \int_Q f \bigl(x_i^m+m^{-1}x,z_i^m + \varphi_m(x)\cdot w_i^m(y) \bigr) \dd y \dd x  \notag\\
& \le  \int_{Q} \int_Q f \bigl(x_i^m+m^{-1}x,z_i^m + w_i^m(y) \bigr) \dd y \dd x + L\|1 - \varphi_m\|_{\Lrm^1(Q)} \|w_i^m\|_{\Lrm^1(Q)}  \notag\\
& \leq \dashint_{Q^m_i} \int_Q f(x,z_i^m + w_i^m(y)) \dd y \dd x + L\|1 - \varphi_m\|_{\Lrm^1(Q)} \|w_i^m\|_{\Lrm^1(Q)}  \notag\\
& \coloneqq I_i^m + II^m_i,  \label{male3}
\end{align}
where here $L$ is the $x$-uniform Lipschitz constant of $f$ with respect to the second argument. Using the modulus of continuity of $f$ from~\eqref{eq:modulus_strong},~\eqref{eq:QC} (twice), and~$Q_{\A}f \leq f$, we get
\begin{align}
I_i^m  & \le \dashint_{Q^m_i} \int_Q f(x_i^m,z_i^m + w_i^m(y)) \dd y \dd x + \omega(m^{-1})\bigg(1 + \int_Q f(x_i^m,z_i^m + w_i^m(y)) \dd y\bigg)  \notag\\
& \le  Q_{\A}f(x_i^m,z^m_i) + \omega(m^{-1})(1 + f(x_i^m,z_i^m)) + \SmallO_m(1).  \label{male4}
\end{align}
Additionally, by~\eqref{male2} %{\color {orange} since the choice of $\varphi$ has not yet played a role in our calculations, we may select $\varphi$ such that}
\begin{equation}\label{male5}
 \sum_{i = 1}^{L(m)} \mathcal L^d(Q^m_i) II^m_i = \SmallO_m(1).
\end{equation}
%{\color{blue}F: I find the notation $\BigO_m(1)$ strange here since that term is not $\SmallO(1)$ for fixed $m$, but goes to zero as $m \to \infty$. I assume this supposed to mean $\SmallO(1)$ as $m\to \infty$, but this should be written out somewhere.}\\{\color{red} A: I agree, I am also not sure about it though it looks better than the standard one, before setting the notation I would like to ask what do you think? We could either use $\BigO(m^{-1})$ or define $\BigO_m(1)$; for our purposes both cases work but we should use the same notation throughout the text.}\\
Returning to~\eqref{eq:limsup}, we can employ~\eqref{eq:modulusQ},~\eqref{male3},~\eqref{male4} and~\eqref{male5} to further estimate 
\begin{align*}
\lim_{j \to \infty} &\Gcal[u + v_j^m] \\
&  \le \sum_{i = 1}^{L(m)} \left\{\int_{Q_i^m} Q_{\A}f(x_i^m,z_i^m) \dd x + \omega(m^{-1})\left(\int_{Q_i^m} 1 + f(x_i^m,z_i^m) \dd x\right)\right\} + \SmallO_m(1)\\
&  \le \sum_{i = 1}^{L(m)} \left\{\int_{Q_i^m} Q_{\A}f(x_i^m,z_i^m) \dd x + C\omega(m^{-1})\left(\int_{Q_i^m}1 + |z^m_i| \dd x\right)\right\} +\SmallO_m(1)\\
&  \le \sum_{i = 1}^{L(m)} \left\{\int_{Q_i^m} Q_{\A}f(x,z_i^m) \dd x + C\omega(m^{-1})\left(\int_{Q_i^m}1 + |z^m_i| \dd x\right)\right\} +\SmallO_m(1)\\
& \le \int_\Omega Q_{\A}f(x,z^m) \dd x +  C\omega(m^{-1})\big(\|1+|z^m|\|_{\Lrm^1(\Omega)}\big) + \SmallO_m(1) \\
& =  \int_\Omega Q_{\A}f(x,u(x)) \dd x +  \SmallO_m(1),
\end{align*}
where $C > 0$ and $\SmallO_m(1)$ may change from line to line.
Here, we have used the (inherited) Lipschitz continuity of $Q_{\A}f(x,\frarg)$ in the second variable and the fact that $\|u-z^m\|_{\Lrm^1(\Omega)} = \SmallO_m(1)$ to pass to the last equality. Hence,
%
%Since $m \in \Nbb$ was fixed, we may do the same construction for every $m \in \Nbb$ and let $m \to \infty$ in the previous inequality to show that
\begin{equation}\label{eq:UBabs}
\overline\Gcal[u] \le \inf_{m > 0} \; \lim_{j \to \infty} \Gcal[u + v_j^m] \le \int_{\Omega} Q_{\A}f(x,u(x)) \dd x.
\end{equation}

\proofstep{Step~3.b. The upper bound.}
Fix $\mu \in \M(\Omega;\R^N) \cap \ker \A$. By Step~2 we may find a sequence $(u_j) \subset \Lrm^1(\Omega;\R^N)$ that area-strictly converges to $\mu \in \M(\Omega;\R^N)$ with $\A u_j \to 0$ in $\Wrm^{-k,q}$. Hence, by~\eqref{eq:UBabs}, Remark~\ref{rem:strict} and Corollary~\ref{cor:representation},
\begin{align*}
\overline\Gcal[\mu] & \le \liminf_{j \to \infty } \; \overline\Gcal[u_j]  \notag\\
& \le  \limsup_{j \to \infty } \; \ddprb{Q_{\A}f(x,\frarg),\delta[u_j \Lcal^d]}  \notag\\
& \le  \int_\Omega \dprb{Q_{\A}f(x,\frarg),\delta[\mu]_x} \dd x + \int_{\Omega} \dprb{(Q_{\A}f)^\#(x,\frarg),\delta[\mu]^\infty_x} \dd \lambda_{\delta[\mu]}(x) \notag\\
& = \int_\Omega Q_{\A}f\bigg(x,\frac{\dd \mu}{\dd \Lcal^d}(x)\bigg) \dd x + \int_\Omega (Q_{\A}f)^\#\bigg(x,\frac{\dd \mu^s}{\dd |\mu^s|}(x)\bigg) \dd |\mu^s|(x) \notag\\
& = \Gcal_*[\mu].
\end{align*}

%If we now let $h \to \infty$, after a suitable diagonal selection $\tilde v_n \coloneqq  v_n^{h(n)}$, the sequence of $u_n$'s have the following properties: 
%\begin{gather*}
%\wslim_{m \to \infty} u_n \Lcal^d = 0 \quad \text{in $\M(\Omega;\R^N)$},\\
%\lim_{m \to \infty} \A u_n = 0 \quad \text{in $\Wrm^{-k,q}(\Omega;\R^N)$}.
%\end{gather*}

\proofstep{Step~4. General continuity condition.}
It remains to show the upper bound in the case where we only  have~\eqref{eq:modulus} instead of~\eqref{eq:modulus_strong}. As in the previous step, it  suffices to show the upper bound on absolutely continuous fields. We let, for fixed $\eps > 0$,
\[
  f^\eps(x,A) := f(x,A) + \eps \abs{A},
\]
which is an integrand satisfying~\eqref{eq:modulus_strong}. Denote the corresponding functionals with $f^\eps$ in place of $f$ by $\Gcal^\eps, \Gcal^\eps_*, \overline{\Gcal^\eps}$. Then, by the argument in Steps~1--3,
\begin{equation*}\label{eq:artificial}
  \Gcal^\eps_* = \overline{\Gcal^\eps}.
\end{equation*}
We claim that
\begin{equation} \label{eq:QA_conv}
Q_{\Acal^k}f^\eps  \todown Q_{\Acal^k}f  \qquad \text{pointwise in $\Omega \times \R^N$}.
\end{equation}
To see this first notice that $\eps \mapsto Q_{\Acal^k}f^\eps(x,A)$ is monotone decreasing for all $x \in \Omega$, $A \in \R^N$, and
\[
  Q_{\A^k} f  + \eps\abs{\frarg} \le  Q_{\A^k} f^\eps
  \leq f+ \eps\abs{\frarg},
\]
which is a simple consequence of Jensen's classical inequality for $\abs{\frarg}$. It follows that the limit
\[
g(x,A) \coloneqq \inf_{\eps > 0} Q_{\Acal^k}f^\eps(x,A) = \lim_{\eps \todown 0}Q_{\Acal^k}f^\eps(x,A)
\]
defines an upper semicontinuous function $g:\Omega\times \R^N \to \R$ with bounds
\[
 Q_{\A^k} f   \le  g \le f.
\]
%In the case of $(Q_{\A^k}f^\eps)^\infty$ we use that $(Q_{\A^k}f^\eps)^\infty \le f^\infty + \eps|\frarg|$ to get that the limit
%\[
%h(x,A) \coloneqq \inf_{\eps > 0} (Q_{\Acal^k}f^\eps)^\infty(x,A) = \lim_{\eps \todown 0}Q_{\Acal^k}f^\eps(x,A),
%\]
%which, similarly to $g$, defines an upper semicontinuous and $\A^k$-quasiconvex function $h: \Omega \times \R^N \to \R$
%with
%\[
%h \le f^\infty,
%\] 
%from where it follows that $h = Q_{\A^k}f^\infty \le (Q_{\A^k}f)^\infty$.
%\mnote{F: This uses that $Q_{\A^k} f$ is the largest $\A$-qc function below $f$, but this is not obvious from our formula in Definition~\ref{def:quasi}. In a way, it is better to define $Q_{\A^k} f$ as the largest $\A$-qc function below $f$ and then to PROVE the formula in Definition~\ref{def:quasi} (we can refer to Dacorogna's book for this).\\
%A: Added {\color{blue} Corollary~\ref{cor:largest}} establishing this property.\\
%G: Joined in one unique Lemma Corollary~\ref{cor:largest} and Lemma~\ref{lem:qce_property}}
Furthermore, by the monotone convergence theorem, it is easy to check that $g$ is $\Acal^k$-quasiconvex, whereby $g = Q_{\Acal^k}f$ (see Corollary~\ref{lem:qce_property}).

Let us now return to the proof of the upper bound on absolutely continuous fields. By construction,
\begin{equation}\label{eq:artificial}
\overline \Gcal \le \overline{\Gcal^\eps} = \Gcal_*^\eps.
\end{equation}
The monotone convergence theorem and~\eqref{eq:QA_conv} yield
\[
\overline \Gcal[u] \le \Gcal_*[u \Lcal^d] \qquad \text{for all $u \in \Lrm^1(\Omega;\R^N) \cap \ker \A$},
\] 
after letting $\eps \todown 0$ in~\eqref{eq:artificial}. 

The general upper bound then follows in a similar way to the proof under the assumption~\eqref{eq:modulus_strong}.
%by a similar token to one where we assumed~\eqref{eq:modulus_strong}. 
This finishes the proof.\qed

\subsection{Proof of Theorem~\ref{thm:relax2}}

The proof works the same as the proof of Theorem~\ref{thm:relax} with the following additional comments:

%Whether~\eqref{density} is or not a necessary assumption remains an open question.\\

\proofstep{Step~1. The lower bound.} Since restricting to $\A$-free sequences is a particular case of 
the more general convergence $\A u_n \to 0$ in the space $\Wrm^{-k,q}(\Omega;\R^N)$, we can still apply Step~2 in the proof of Theorem~\ref{thm:relax} to prove that
$\Gcal_* \le \overline{\Gcal}$, where for $\mu \in \M(\Omega;\R^N) \cap \ker \A$,
\[
\Gcal_*[\mu] \coloneqq \int_\Omega Q_{\A}f \bigg(x,\frac{\dd \mu}{\dd \Lcal^d}(x)\bigg) \dd x + \int_\Omega (Q_{\A}f)^\#\bigg(x,\frac{\dd \mu^s}{\dd |\mu^s|}(x)\bigg) \dd |\mu^s|(x).
\]

\proofstep{Step~2. An $\A$-free strictly convergent recovery sequence.} In this case, this forms part of the assumptions. 

\proofstep{Step~3.a. Upper bound on absolutely continuous $\A$-free fields.} An immediate consequence of Remark~\ref{rem:project} is that one may assume, without loss of generality, that the recovery sequence for the upper bound lies in $\ker \A$. Thus, the upper bound on absolutely continuous fields in the constrained setting also holds.

\proofstep{Step~3.b. The upper bound (assuming~\eqref{eq:modulus_strong}).} The proof is the same as in the proof of Theorem~\ref{thm:relax}.

\proofstep{Step~4. General continuity condition}. Since assumption~\eqref{eq:modulus_strong} is a structural property (coercivity) of the integrand and the arguments do not depend on the underlying space of measures, the argument
remains the same as in the proof of Theorem~\ref{thm:relax}. \qed

\appendix
\section{Proofs of the localization principles}

In this appendix we prove Proposition~\ref{prop: localization regular} and Proposition~\ref{prop: localization}.

\medskip

\proofstep{Proof of Proposition~\ref{prop: localization regular}:} In the following we adapt the main steps in proof of the localization principle at regular points which is contained in Proposition 1 of \cite{Rindler12BV}. The statement on the existence of an $\A$-free and periodic generating sequence is proved in detail.
   
   Let $\mu_j \in \M(\Omega;\R^N)$ be the sequence of asymptotically $\A$-free measures which generates $\nu$. In the following steps, for an open $\Omega' \subset \R^d$, we will often identify a measure $\mu \in \M(\Omega';\R^N)$ with its zero extension in $\M_\loc(\R^d;\R^N)$, and similarly for a Young measure $\sigma \in \Y(\Omega';\R^N)$ and its zero extension in $\Y_\loc(\R^d;\R^N)$.
   
  \proofstep{Step~1.} We start by showing that, for every $r > 0$, there exists a subsequence of $j$'s (the choice of subsequence might depend on $r$) such that 
  \begin{equation}\label{eq:principles1}
  r^{-d}T^{(x_0,r)}_\# \mu_j \toY \sigma^{(r)} \quad \text{in $\Y_\loc(\R^d;\R^N)$}.
  \end{equation}
  Moreover, for $\Lcal^d$-a.e. $x_0 \in \Omega$, one can show that a uniform bound 
  \begin{equation}
  \sup_{r>0} \, \ddprb{\mathds 1_K \otimes |\frarg|, \sigma^{(r)}} < \infty \quad \text{for every $K \Subset \R^d$}
  \end{equation}
holds; thus, by Lemma \ref{lem:YM_compact}, there exists a sequence of positive numbers $r_m \todown0$ and a Young measure $\sigma$ for which
  \[
  \sigma^{(r_m)} \toweakstarY \sigma \quad \text{in $\Y_\loc(\R^d;\R^N)$}.
  \]
%  Once we have passed to the subsequence $(r_i)$ we may, without loss of generality, assume that each sequence $(\mu_{j(i)})$ is a subsequence of $(\mu_{j(i-1)})$; where, for each $h \in \Nbb$, $(\mu_{j(i)})$ is the sequence for which \eqref{eq:principles1} holds with $r = r_i$.

 \proofstep{Step~2.} For an arbitrary measure $\gamma \in \M(\Omega;\R^N)$, the Radon-Nykod\'ym differentiation theorem yields
 \begin{align*}
r^{-d}T^{(x_0,r)}_\# \gamma = \frac{\dd \gamma}{\dd \Lcal^d}(x_0 + r\,\frarg) \, \Lcal^d + \frac{\dd \gamma}{\dd |\gamma^s|}(x_0 + r \, \frarg) \, r^{-d}T^{(x_0,r)}_\#|\gamma^s|.
 \end{align*}
Consider $\sigma^{(r)}$ as an element of $\Y(Q;\R^N)$. Fix $\phi \otimes h \in \Crm(\cl Q) \times \Lip(\R^N)$. Using a simple change of variables, we get
  \begin{equation}\label{eq:principles2}
  \begin{split}
  \ddprb{\phi \otimes h,\sigma^{(r)}} %& = \lim_{j \to \infty} \int_{Q} \phi(y) \cdot h\bigg(\bigg(\frac{\dd \mu_j}{\dd \Lcal^d}\bigg)^{(r)}(y)\bigg) \dd y \\
  & = \lim_{j \to \infty} \bigg(\int_{Q} \phi(y) \cdot h\bigg(\frac{\dd \mu_j}{\dd \Lcal^d}(x_0 + ry)\bigg) \dd y  \\ 
  & \qquad + \int_{\cl Q} \phi(y) \cdot h^\infty\bigg(\frac{\dd \mu_j}{\dd |\mu_j^s|}(x_0 + ry)\bigg) \dd (r^{-d}T^{(x_0,r)}_\#|\mu_j^s|)(y) \bigg) \\
  & = r^{-d}  \lim_{j \to \infty}\bigg(\int_{Q_r(x_0)} \phi \circ T^{(x_0,r)}(x) \cdot h\bigg(\frac{\dd \mu_j}{\dd \Lcal^d}(x)\bigg) \dd x \\
    & \qquad + \int_{\cl {Q_r(x_0)}} \phi \circ T^{(x_0,r)}(x) \cdot h^\infty\bigg(\frac{\dd \mu_j}{\dd |\mu_j^s|}x\bigg) \dd |\mu_j^s|(x) \bigg) \\
  & = r^{-d} \ddprb{(\phi \circ T^{(x_0,r)})\otimes h, \nu}.
%  & = \int_Q \phi(y) \bigg(\dprb{h,\nu_{x_0 + ry}} +  \\ 
%  & \qquad \dprb{h^\infty,\nu_{x_0 + ry}} \frac{\dd \lambda_\nu}{\dd \Lcal^d}(x_0 + ry)\bigg)\dd y + \BigO(r); 
\end{split}
  \end{equation}

  \proofstep{Step~3.}  We now let $r = r_m$ in~\eqref{eq:principles2} and quantify its values as $m \to \infty$. This will allow us to characterize $\sigma$ in terms of $\nu$.
  
   Let $\{g_l \coloneqq \phi_l \otimes h_l\}_l \subset \Crm(\cl Q) \times \Lip(\R^N)$ be the dense subset of $\E(Q;\R^N)$ provided by Lemma \ref{lem:separation} and further assume that $x_0$ verifies the following properties: $x_0$ is a Lebesgue point of the functions
  \begin{equation}\label{eq:principles3}
x \mapsto \dprb{h_l,\nu_x} + \dprb{h^\infty_l,\nu^\infty_x}\frac{\dd \lambda_\nu}{\dd \Lcal^d}(x), \qquad \text{for all $l \in \Nbb$},
  \end{equation}
  and $x_0$ is a regular point of the measure $\lambda_\nu$, that is,
  \begin{equation}\label{eq:principles4}
  \frac{\dd \lambda_\nu^s}{\dd \Lcal^d}(x_0) = \lim_{r \todown 0} \frac{\lambda_\nu^s(Q_r(x_0))}{r^d} = 0.
  \end{equation}
  
 Consider $\sigma$ as an element of $\Y(Q;\R^N)$. Setting $r = r_m$ in~\eqref{eq:principles2} and letting $m \to \infty$ we get
  \begin{align*}
    \ddprb{g_l,\sigma} & = \lim_{m \to \infty} \, r_m^{-d} \ddprb{\phi_l \circ T^{(x_0,r_m)} \otimes h_l, \nu}\\
    %& \ddprb{\mathds 1_Q g_l,\sigma^{(r_m)}} \\
    & =  \lim_{m \to \infty} \, \bigg(  \dashint_{Q_{r_m}(x_0)} \phi_l\bigg(\frac{x - x_0}{r_m}\bigg)\bigg[\dprb{h_l,\nu_{x}} +  \dprb{h^\infty_l,\nu^\infty_{x}} \frac{\dd \lambda_\nu}{\dd \Lcal^d}(x)\bigg] \dd x \\
    &  \qquad  + \frac{1}{r^d}\int_{\cl{Q_{r_m}(x_0)}} \phi_l\bigg(\frac{x - x_0}{r_m}\bigg)\dprb{h_l^\infty,\nu_{x}^\infty} \dd \lambda^s_\nu(x) \bigg) \\
    & = \int_{Q} \dprb{g_l(y,\frarg),\nu_{x_0}} \dd y + \int_{Q} \dprb{g_l^\infty(y,\frarg),\nu^\infty_{x_0}}\frac{\dd \lambda_\nu}{\dd \Lcal^d}(x_0) \dd y.
  \end{align*}
  Here, we have used~\eqref{eq:principles3} and the Dominated Convergence Theorem to pass to the limit in the first summand, and with the help of~\eqref{eq:principles4}, we used that
  \[
\int_{\cl{Q_{r}(x_0)}} \phi_l\bigg(\frac{x - x_0}{r}\bigg)\dprb{h_l^\infty,\nu_{x}^\infty} \dd \lambda^s_\nu(x) \le \|\phi\|_\infty\cdot\text{Lip}(h_l) \cdot \lambda^s_\nu(Q_r(x_0)) = \SmallO(r^d)
  \]
to neglect the second summand in the limiting process.

%  
%  
%  where in the last equality we used that, at regular points $x_0 \in \Omega$ of $\lambda_\nu$,
%  \[
%   \bigg|\int_{\cl Q} \phi(y) \dprb{h^\infty,\nu^\infty_{x_0 + ry}} \dd(T^{(x_0,r)}_\#\lambda^s_\nu)(y)\bigg| \le M \|\phi\|_\infty \cdot \lambda^s_\nu(Q_r(x_0)) = \SmallO(r^d).
%     \]
%Letting $r \todown 0$ we obtain
%\[
%\ddprb{\phi \otimes h,\sigma}  = \int_Q \bigg(\dprb{\phi\otimes h,\nu_{x_0}}  + \dprb{\phi\otimes h^\infty,\nu^\infty_{x_0}}\frac{\dd \lambda_\nu}{\dd \Lcal^d}(x_0)\bigg) \dd y
%\]
Since the set $\{g_l\}$ separates $\Y(Q;\R^N)$, Lemma \ref{lem:separation} tells us that $\sigma_y = \nu_{x_0}, \sigma^\infty_y = \nu^\infty_{x_0}, \lambda_\sigma = \frac{\dd \lambda_\nu}{\dd \Lcal^d}(x_0) \Lcal^d$ for $\Lcal^d$a-e. $y \in Q$, and that $\lambda^s_\sigma$ is the zero measure in $\M(\cl Q)$, as desired. 

\proofstep{Step~4.}  We use a diagonalization principle (where $j$ is the fast index with respect to $m$) to find a subsequence $(\mu_{j(m)})$ such that
	\begin{equation}\label{eq:LPR1}
	\gamma_m \coloneqq r_m^{-d}T^{(x_0,r_m)}_\# \mu_{j(m)} \toY \sigma \quad \text{in $\Y_\loc(\R^d;\R^N)$}. 
	\end{equation}
	
%The \emph{final step} is what truly makes this localization principle \emph{ad hoc} for the differential constraint.  
\proofstep{Step~5.}  Up to this point, the localization principle presented in Proposition 1 of \cite{Rindler12BV} has been adapted to Young measures without imposing any differential constraint.    
Here we additionally require $\sigma$ to be an $\A^k$-free Young measure; this is achieved by showing that $(\gamma_m)$ is asymptotically $\A^k$-free (on bounded subsets of $\R^d$). To this end let us note that 
\[
\A\mu_j=\theta_j
\]
with \(\|\theta_{j}\|_{W^{-k,q}}\to 0\) as \(j\to \infty\). By scaling  we can write
\[
\A^k\gamma_m=r_m^{-d} \A^k (T^{x_0,r_m}_\# \mu_{j(m)})=- \sum_{h = 0}^{k-1} \A^h\big(r^{k -h}r_{m}^{-d}T^{(x_0,r)}_\# \mu_{j(m)}\big)+r_m^{k-d}T^{(x_0,r_m)}_\#\theta_{j(m)}.
\]
Since \(\|r_m^{k-d}T^{(x_0,r_m)}_\#\theta_{j}\|_{W^{-k,q}}\le C(m) \|\theta_{j}\|_{W^{-k,q}}\) and for every open $U \Subset \R^d$ there exists a positive constant $C_U$ such that
	 \[
	  \sup_{m\in\N} r^{-d}_mT^{(x_0,r_m)}_\# |\mu_{j}|(U) \le C_U,
	 \]
arguing as in Proposition~\ref{lem:nonh} we can choose a further subsequence \(j(m)\to \infty\) such that
\[
\A^k\gamma_m=\A^k \big(r_m^{-d}T^{(x_0,r)}_\# \mu_{j(m)}\big)\to 0\qquad\textrm{in \(\Wrm^{-k,q}_{\loc}(\R^d)\)}
\] 
and this shows that  $\sigma$ is an $\A^k$-free Young measure.
%
%
%
%
%	 it follows from~\eqref{eq:LPR1} and Theorem~\ref{thm:Reshet}, that for every open $U \Subset \R^d$ there exists a positive constant $C_U$ such that
%	 \[
%	  r^{-d}_mT^{(x_0,r_m)}_\# |\mu_{j(m)}|(U) \le C_U
%	 \]
%	 whenever $m$ is sufficiently large.
%	 Therefore, the assertion 
%%	 Since each $\mu_j$ is $\A$-free, a change of variable argument gives
%%	\[
%%	\sum_{h = 1}^k r^{h} \A^h(T^{(x_0,r)}_\# \mu_j) =  T^{(x_0,r)}_\#[\A\mu_j] = 0.
%%	\]
%%	
%%	Hence, 
%%	\[
%%	\A^k \gamma_i = -\sum_{h = 1}^{k-1} r_i^{h-k} \A^h(r_i^{-d}T^{(x_0,r_i)}_\# \mu_{j(i)}) = -\sum_{h = 1}^{k-1} r_i^h \A^h\gamma_{r_i}.
%%	\]
%%		On the other hand, since $\gamma_i \toY \sigma$, it must be that 
%%	\[
%%	\gamma_i \toweakstar [\sigma] = \frac{\dd \mu}{\dd \Lcal^d}(x_0) \Lcal^d.
%%	\]
%%	The compactness of the embedding $\M(\R^d;\R^N) \cembed \Wrm^{-k,q}(\R^d;\R^N)$ then yields
%	\[
%	\A^k \gamma_m \to 0 \quad \text{in $\Wrm_\loc^{-k,q}$},
%	\]
%	is an immediate consequence of Proposition~\ref{lem:nonh} applied to the sequence of measures with elements $\mu_m = \mu_{j(m)}$ and the constants $c_m \coloneqq r^{-d}_m$.

	\proofstep{Step~6.}  So far we have shown that $[\sigma] = A_0 \Lcal^d$ with
\[
A_0 := \dprb{ \id, \nu_{x_0} } + \dprb{ \id, \nu_{x_0}^\infty } \frac{\di \lambda_\nu}{\di \Lcal^d}(x_0)  \quad \in \R^N,
\]
and that $\sigma$ is generated by a sequence $(\mu_j) \subset \M(Q;\R^N)$ satisfying $\A^k \mu_j \to 0$. Note that without loss of generality we may assume %--- concerning the generation of $\sigma$ --- 
that the $\mu_j$'s are of the form $u_j \Lcal^d$ where $u_j \in \Lrm^1(Q;\R^N)$. Indeed, since
\begin{gather*}
\gamma_R \coloneqq T^{(x_0,R)}_\# \mu_j \to \mu_j \; \quad \text{area strictly in $\M_\loc(\R^d;\R^N)$}, \\
\|\A^k(\gamma_R - \mu_j)\|_{\Wrm^{-k,q}_\loc(\R^d)} \to 0 \quad \text{as $R \toup 1$},  
\end{gather*}
and
\begin{gather*}
\gamma_R \ast \rho_{\eps} \to \gamma_R \; \quad \text{area strictly in $\M_\loc(\R^d;\R^N)$}, \\
 \|\A^k(\gamma_R - \gamma_R \ast \rho_{\eps})\|_{\Wrm^{-k,q}_\loc(\R^d)} \to 0 \quad \text{as $\eps \todown 0$},
\end{gather*}
we might use a diagonalization argument (relying on the weak*-metrizability of bounded subsets of $\E(Q;\R^N)^*$ and Remarks~\ref{rem:area},~\ref{rem:strict}),  where $\eps$ appears as the faster index with respect to $R$, to find a sequence with elements $u_j \coloneqq \gamma_{R(j)}\ast\rho_{\eps(R(j))}$ such that
\begin{gather}\label{WLOG}
u_j \Lcal^d \toY \sigma \in \Y_\loc(\R^d;\R^N) \quad \text{and} \quad \A^k u_j \to 0 \quad \text{in $\Wrm^{-k,q}_\loc(\R^d)$}.
\end{gather}
%Recall that, by Proposition~\ref{prop:nonhomo} we know that $[\sigma]$ is an $\A^k$-free measure. 
%{\color{red}Furthermore, we know that $\sigma$ is generated, up to mollifying, by a blow-up sequence $(u_j) \subset \Lrm^2(Q,\R^N)$ such that}
%\[
%  \A^k u_j \to 0  \quad \text{in $\Wrm^{-k,q}(Q)$}
%  \qquad\text{and}\qquad
%  u_j  \toY \sigma  \quad \text{in $\Y(Q;\R^N)$}.
%\]
%In particular, we may find a cube $Q_r := rQ$ for some $r \in (0,1)$ with the following properties: $u_j \toweakstar [\sigma]$ in $\Mcal(Q_r;\R^N)$, and
%$|u_j| \toweakstar \Lambda$ in $\Mcal(\cl{Q_r})$ and $\Lambda(\partial Q_r)= 0$. 
Using~\eqref{eq:both}, we get 
\[
|u_j|\Lcal^d \toweakstar |[\sigma]| = |A_0| \Lcal^d\quad \text{in $\M_\loc(\R^d)$}.
\]
Hence, $|u_j|\Lcal^d \toweakstar \Lambda$ in $\M(\cl Q)$ with $\Lambda(\partial Q) = 0$.
%\begin{equation}\label{eq: constant blow}\int_{Q} u_j \dd y \to A_0.\end{equation}
%Up to cut-off and mollification operations, followed by the projection from  Lemma~\ref{lem:fonseca constant} we may assume without any loss of generality that $u_j \in \Crm^\infty_\per(Q;\R^N)$.
%It is straightforward to check that the hypothesis of Lemma~\ref{lem:boundary values} hold for the sequences 
%$\{u_h\},\{A_0\}$ and $h_k$ for every $k \in \N$. Hence, we may find a sequence of corrector functions $z_h \in \Lrm^2(Q;\R^N) \cap \ker \A$ with
%\[
%\int_Q z_h = 0 \quad \forall \; h \in \N, \qquad z_h \toweakstar 0,
%\]
%and
%\begin{equation}
%\label{eq: regular quasi} \lim_{h \to \infty} \int_Q h_k(u_h) \dd y \ge \liminf_{h \to \infty} \int_Q h_k(A_0 + z_h) \dd y, \qquad \text{for every $k \in \N$}.
%\end{equation}
We are now in position to apply Lemma~\ref{lem:boundary values} %(up to re-scaling the cube $Q_r$ into the cube $Q$) 
to the sequences $(u_j)$ and $(v_j := A_0)$ to find a sequence $z_j \in \Crm^\infty_\per(Q;\R^N) \cap \ker \A^k$ with 
$\int_{Q} z_j \dd y = 0$ and such that (up to taking a subsequence)
\begin{equation}\label{eq: promedio}
(A_0 + z_j)\Lcal^d \toY \sigma \quad \text{in $\Y(Q;\R^N)$}.
\end{equation}
Since the properties of $x_0$ that were involved in Steps 1-3 are valid at $\Lcal^d$-a.e. $x_0 \in \Omega$, %and Steps 4-6 are consequences of  them, 
the sought localization principle at regular points is proved.\qed

\medskip

 \proofstep{Proof of Proposition~\ref{prop: localization}:} The proof of the localization principle at singular points resembles the one for regular points, with a few exceptions:

 \proofstep{Step~1.}  %\mnote{A: Please argue how is that at $\lambda_\nu^s$-a.e. $x_0$ we may assume $\sigma$ is non-trivial (via Preiss at the level of $[\sigma]$); connect with (6')}
In comparison to Step~1 from the regular localization principle, we here chose $c_r(x_0) \coloneqq |\lambda_{\nu}^s|(Q_r(x_0))^{-1}>0$  and we define \(\sigma^r\) as 
 \[
c_r(x_0)T^{(x_0,r)}_\# \mu_j \toY \sigma^{(r)} \quad \text{in $\Y_\loc(\R^d;\R^N)$}.
 \]
 Moreover, by \cite[Lemma 2.4 and Theorem 2.5]{Preiss87} and~\eqref{eq:principles6} below, at $\lambda^s_\nu$-a.e. $x_0 \in \Omega$, it is possible to show that
  \begin{equation}
  \sup_{r> 0} \, \ddprb{\mathds 1_K \otimes |\frarg|, \sigma^{(r)}} = \sup_{r> 0} \frac{|\lambda_\nu^s|(x_0+rK)+\Lcal^d(x_0+rK)}{ |\lambda_{\nu}^s|(Q_r(x_0))}< \infty \quad \text{for every $K \Subset \R^d$}.
  \end{equation}
By compactness of $\Y_\loc(\R^d;\R^N)$, see~Lemma \ref{lem:YM_compact}, there exists a sequence of positive numbers $r_m \todown0$ and a Young measure $\sigma$ for which
  \[
  \sigma^{(r_m)} \toweakstarY \sigma \quad \text{in $\Y_\loc(\R^d;\R^N)$}
  \]
%Moreover, by Preiss' existence result for non-zero tangent measures~\cite{Preiss87}, we may assume that $\sigma$ and hence $\lambda_\sigma$ are non-zero.
 
 \proofstep{Step~2.}  The calculations of the second step, for the constant $c_r(x_0)$, is
\begin{align}\label{notag}
  \ddprb{\phi \otimes h,\sigma^{(r)}} %& = \lim_{j \to \infty} \int_{Q} \phi(y) \cdot h\bigg(\bigg(\frac{\dd \mu_j}{\dd \Lcal^d}\bigg)^{(r)}(y)\bigg) \dd y \\ 
  \notag
  & = \lim_{j \to \infty} \bigg(\int_{Q} \phi(y) \cdot h\bigg(c_r(x_0) r^{d}\frac{\dd \mu_j}{\dd \Lcal^d}(x_0 + ry)\bigg) \dd y  \\ \notag
  & \qquad + \int_{\cl Q} \phi(y) \cdot h^\infty\bigg(c_r(x_0) r^{d}\frac{\dd \mu_j}{\dd |\mu_j^s|}(x_0 + ry)\bigg) \dd (r^{-d}T^{(x_0,r)}_\#|\mu_j^s|)(y) \bigg) \\\notag
  & = r^{-d} \lim_{j \to \infty}\bigg(\int_{Q_r(x_0)} \phi \circ T^{(x_0,r)}(x) \cdot h\bigg(c_r(x_0) r^d \frac{\dd \mu_j}{\dd \Lcal^d}(x)\bigg) \dd x \\\notag
    & \qquad + \int_{\cl{Q_r(x_0)}} \phi \circ T^{(x_0,r)}(x) \cdot h^\infty\bigg(c_r(x_0) r^{d}\frac{\dd \mu_j}{\dd |\mu_j^s|}(x)\bigg) \dd |\mu_j^s|(x) \bigg)\\
  & = r^{-d}\ddprb{\phi \circ T^{(x_0,r)}\otimes h(c_r(x_0) r^d \frarg), \nu}.
%  & = \int_Q \phi(y) \bigg(\dprb{h,\nu_{x_0 + ry}} +  \\ 
%  & \qquad \dprb{h^\infty,\nu_{x_0 + ry}} \frac{\dd \lambda_\nu}{\dd \Lcal^d}(x_0 + ry)\bigg)\dd y + \BigO(r); 
  \end{align}
  
 \proofstep{Step~3.}  The assumptions of the third step are substituted by assuming that $x_0$ is a $\lambda_\nu^s$-Lebesgue point of the functions
 \begin{equation}\label{eq:principles5}
 x \mapsto \dprb{|\frarg|,\nu_x^\infty}, \quad x \mapsto \dprb{h_l^\infty,\nu_x^\infty} \quad \text{for all $l \in \Nbb$}.
 \end{equation}
We further require that
 \begin{equation}\label{eq:principles6}
\lim_{r \todown 0} \frac{r^d}{\lambda^s_\nu(Q_r(x_0))} = \lim_{r \todown 0}  c_r(x_0) r^d = 0
 \end{equation}
 and that
 \begin{equation}\label{eq:principles7}
\lim_{r \todown 0} c_r(x_0)\int_{Q_{r}(x_0)} \bigg[\dprb{|\frarg|,\nu_{x}} +  \dprb{|\frarg|,\nu^\infty_{x}} \frac{\dd \lambda_\nu}{\dd \Lcal}(x) \bigg] \dd x=0.
 \end{equation}
Hence, defining  $S \coloneqq \setb{x_0 \in \Omega}{\text{\eqref{eq:principles5},~\eqref{eq:principles6} and~\eqref{eq:principles7} hold}}$, we have  $\lambda^s_\nu(\Omega\setminus S)=0$.
 
 Fix $x_0 \in S$. 
 Setting $r = r_m$ in~\eqref{notag} and letting $m \to \infty$ gives
 \begin{align*}
    \ddprb{\mathds 1_Q \otimes |\frarg|,\sigma} & = \lim_{m \to \infty} \, \ddprb{\mathds 1_Q  \otimes|\frarg|,\sigma^{(r_m)}} \\
     & = \lim_{m \to \infty} \, c_{m}(x_0)\bigg(  \int_{Q_{r_m}(x_0)} \bigg[\dprb{|\frarg|,\nu_{x}} +  \dprb{|\frarg|,\nu^\infty_{x}} \frac{\dd \lambda_\nu}{\dd \Lcal}(x) \bigg] \dd x\\
    &  \qquad  + \int_{\cl{Q_{r_m}(x_0)}} \dprb{|\frarg|,\nu_{x}^\infty} \dd  \lambda^s_\nu(x) \bigg) \\
    & =  \dprb{|\frarg|,\nu^\infty_{x_0}} \lim_{m \to \infty} \bigg(\int_{\cl Q} \dd (c_m(x_0) T^{(x_0,r_m)}_\# \lambda^s_\nu)(y)\bigg)\\ 
    & =  \int_{\cl Q}\dprb{|\frarg|,\nu_{x_0}^\infty} \dd \gamma(y), \quad \text{for some $\gamma \in \Tan(\lambda^s_\nu,x_0)$},
  \end{align*}
  where we have used that $x_0 \in S$. Moreover 
  \[
  \gamma (\overline Q)=   \ddprb{\mathds 1_{\overline Q} \otimes |\frarg|,\sigma} \ge \lim_{m\to \infty} \frac{|\lambda^s_\nu|(Q_{r_m}(x_0))}{|\lambda^s_\nu|(Q_{r_m}(x_0))}=1
  \]
  which implies \(\gamma\ne 0\).  Testing  with $g_l = \phi_l\otimes h_l$, we obtain by~\eqref{eq:principles5} and a similar argument to the  one above, that
 \begin{align*}
    \ddprb{g_l,\sigma} = \int_{\cl Q} \phi_l(y)\dprb{h_l^\infty,\nu^\infty_{x_0}} \dd\gamma(y).
  \end{align*}    
  From the above  equations we deduce that  that $\sigma_y = \delta_0$ for $\Lcal^d$-a.e. $y \in Q$, $\sigma^\infty_y = \nu^\infty_{x_0}$ and $\lambda_\sigma=\gamma \in \Tan(\lambda_\nu^s,x_0)\setminus\{0\}$.
  
 \proofstep{Step~4.}  The arguments of Step 4 remain unchanged except that this time one gets
 \[
 \gamma_m \coloneqq c_m T^{(x_0,r_m)}_\# \mu_{j(m)} \toY \sigma \quad \text{in $\Y(Q,\R^N)$};
 \]

 \proofstep{Step~5.} This is similar to the corresponding step in the proof of the regular localization principle.

 \proofstep{Step~6.}  Differently from the case at regular points, we want to additionally show $\lambda_\sigma(Q) = 1$ and $\lambda_\sigma(\partial Q) = 0$. There exists $0 <  \eps < 1$ such that $\lambda_\sigma(\partial Q_\eps)=0$. Up to taking $r' = \eps r$ (and thus  as $r_m' = r_m\eps$) in the arguments of Steps 1-4 above we may assume without loss of generality that %$\lambda(Q) =1$ and that 
 $\lambda_\sigma(\partial Q) = 0$ and $\lambda_\sigma(Q) = 1$.
 %  where in the last equality we used that, at singular points $x_0 \in \Omega$ of $\lambda_\nu$,
%  \[
%   \bigg|\int_{Q} \phi(y) \dprb{h(c_r(x_0) r^d \frarg),\nu_{x_0 + ry}} \dd y\bigg| \le M(1 + |Q|) \|\phi\|_\infty \cdot \lambda^s_\nu(Q_r(x_0)) = \SmallO(r^d).
%     \]
This proves the localization principle at singular points.\qed

%\bibliographystyle{amsalpha}
%\bibliography{../Bib/AfreeRelax}

% AUTO-GENERATED BIBLIOGRAPHY, CHANGES WILL BE LOST:

\providecommand{\bysame}{\leavevmode\hbox to3em{\hrulefill}\thinspace}
\providecommand{\MR}{\relax\ifhmode\unskip\space\fi MR }
% \MRhref is called by the amsart/book/proc definition of \MR.
\providecommand{\MRhref}[2]{%
  \href{http://www.ams.org/mathscinet-getitem?mr=#1}{#2}
}
\providecommand{\href}[2]{#2}

\end{document}